\documentclass[12pt,reqno]{amsart}
\usepackage[top=1.25in,bottom=1.25in,left=1in,right=1in]{geometry}
\usepackage{amsmath,amssymb,amsthm,pinlabel}
\usepackage{ifthen}
\usepackage{enumitem}
\usepackage{mathtools}
\usepackage[all,cmtip]{xy}
\usepackage[final, 
colorlinks=true,
pdftitle={Thermodynamic metrics on outer space},
pdfauthor={Tarik Aougab, Matt Clay and Yo'av Rieck}]{hyperref}

\usepackage{tikz,tikz-cd,amscd}
\usetikzlibrary{calc}
\usepackage{tikz,mathdots,color}
\usetikzlibrary{decorations.pathreplacing,decorations.markings}
\usetikzlibrary{shapes,shapes.arrows,positioning}

\tikzset{->-/.style={decoration={
  markings,
  mark=at position #1 with {\arrow{>}}},postaction={decorate}}}
\tikzset{->-/.default=.5}

\theoremstyle{plain}
\newtheorem{theorem}{Theorem}
\newtheorem{lemma}[theorem]{Lemma}
\newtheorem{proposition}[theorem]{Proposition}
\newtheorem{corollary}[theorem]{Corollary}
\newtheorem{claim}[theorem]{Claim}
\newtheorem*{claim*}{Claim}

\theoremstyle{definition}
\newtheorem{definition}[theorem]{Definition}

\newtheorem{example}[theorem]{Example}
\newtheorem{remark}[theorem]{Remark}
\newtheorem{question}[theorem]{Question}

\newtheorem{conjecture}[theorem]{Conjecture}
\newtheorem{problem}[theorem]{Problem}

\numberwithin{theorem}{section}
\numberwithin{equation}{section}

\newcommand{\FF}{\mathbb{F}}
\newcommand{\NN}{\mathbb{N}} 
\newcommand{\RR}{\mathbb{R}}
\newcommand{\PP}{\mathbb{P}}

\newcommand{\One}{\mathbf{1}} 
\newcommand{\Zero}{\mathbf{0}}

\newcommand{\bc}{\mathbf{c}}  
  
\newcommand{\bv}{\mathbf{v}}  
  
\newcommand{\bH}{\mathbf{H}}

\newcommand{\calB}{\mathcal{B}}
\newcommand{\calC}{\mathcal{C}}

\newcommand{\calG}{\mathcal{G}}
\newcommand{\calH}{\mathcal{H}}

\newcommand{\calL}{\mathcal{L}}
\newcommand{\calM}{\mathcal{M}}

\newcommand{\calR}{\mathcal{R}}
\newcommand{\calS}{\mathcal{S}}

\newcommand{\calV}{\mathcal{V}}
\newcommand{\calX}{\mathcal{X}}

\newcommand{\entropy}{\mathfrak{h}}

\newcommand{\pressure}{\mathfrak{P}}

\newcommand{\sfm}{\mathsf{m}}

\newcommand{\tG}{{\widetilde G}}

\newcommand{\oA}{\overline{A}}
\newcommand{\oC}{\overline{C}}
\newcommand{\oD}{\overline{D}}
\newcommand{\oF}{\overline{F}}

\newcommand{\hcalM}{\widehat{\mathcal{M}}}
\newcommand{\ocalM}{\overline{\mathcal{M}}}
\newcommand{\ccalX}{\widehat{\mathcal{X}}} 
\newcommand{\ocalX}{\overline{\mathcal{X}}} 
    
\newcommand{\abs}[1]{\left\lvert {#1} \right\rvert} 
 
\newcommand{\I}[1]{\langle #1 \rangle}

\newcommand{\from}{\colon\thinspace}

\newcommand{\norm}[1]{\left\| {#1} \right\|}

\newcommand{\splx}[1]{\abs{#1}} 
\newcommand{\PD}[2]{\partial_{#2} #1}
\newcommand{\PPD}[3]{\partial_{#2 #3} #1}

\newcommand{\vardottilde}[1]{\dot{\raisebox{0pt}[1\height]{$\tilde{#1}$}}}

\newcommand{\param}%
	{{\mathchoice{\mkern1mu\mbox{\raise2.2pt\hbox{$\centerdot$}}\mkern1mu}%
	{\mkern1mu\mbox{\raise2.2pt\hbox{$\centerdot$}}\mkern1mu}%
	{\mkern1.5mu\centerdot\mkern1.5mu}{\mkern1.5mu\centerdot\mkern1.5mu}}}

\DeclareMathOperator{\Aut}{Aut}
\DeclareMathOperator{\Homeo}{Homeo}
\DeclareMathOperator{\Inn}{Inn}
\DeclareMathOperator{\Lip}{Lip}
\DeclareMathOperator{\Mat}{Mat}
\DeclareMathOperator{\MCG}{MCG}
\DeclareMathOperator{\Out}{Out}
\DeclareMathOperator{\diam}{diam}
\DeclareMathOperator{\spec}{spec}
\DeclareMathOperator{\trace}{tr}
\DeclareMathOperator{\vol}{vol}



\begin{document}


\title[Thermodynamic metrics on outer space]{Thermodynamic metrics on outer space}

\author[T. Aougab]{Tarik Aougab}
\address{Department of Mathematics\\
Haverford College\\
370 Lancaster Avenue\\
Haverford, PA 19041, USA\\}
\email{\href{mailto:taougab@haverford.edu}{taougab@haverford.edu}}

\author[M. Clay]{Matt Clay}
\address{Department of Mathematics\\
University of Arkansas\\
Fayetteville, AR 72701, USA\\}
\email{\href{mailto:mattclay@uark.edu }{mattclay@uark.edu}}

\author[Y. Rieck]{Yo'av Rieck}
\address{Department of Mathematics\\ 
University of Arkansas\\ 
Fayetteville, AR 72701, USA\\}
\email{\href{mailto:yoav@uark.edu}{yoav@uark.edu}}
\date{\today}

\begin{abstract}
In this paper we consider two piecewise Riemannian metrics defined on the Culler--Vogtmann outer space which we call the \emph{entropy metric} and the \emph{pressure metric}.  As a result of work of McMullen, these metrics can be seen as analogs of the Weil--Petersson metric on the Teichm\"uller space of a closed surface.  We show that while the geometric analysis of these metrics is similar to that of the Weil--Petersson metric, from the point of view of geometric group theory, these metrics behave very differently to the Weil--Petersson metric.  Specifically, we show that when the rank $r$ is at least 4, the action of $\Out(\FF_r)$ on the completion of the Culler--Vogtmann outer space using the entropy metric has a fixed point.  A similar statement also holds for the pressure metric.
\end{abstract}

\maketitle

\tableofcontents


\section{Introduction}

The purpose of this paper is to introduce and examine two piecewise Riemannian metrics, called the \emph{entropy metric} and the \emph{pressure metric}, on the Culler--Vogtmann outer space $CV(\FF_r)$. The Culler--Vogtmann outer space is the moduli space of unit volume marked metric graphs and as such it is often viewed as the analog of the Teichm\"uller space of an orientable surface $S_g$.  Both the Culler--Vogtamnn outer space and the Teichm\"uller space admit a natural properly discontinuous action by a group.  For the Culler--Vogtmann outer space, the group is the outer automorphism group of a free group $\Out(\FF_r) = \Aut(\FF_r)/\Inn(\FF_r)$.  For the Teichm\"uller space, the group is the mapping class group of the surface $\MCG(S_g) = \pi_0 (\Homeo^+(S_g))$.  Strengthening the connection between these spaces and groups are the facts that (i) $\Out(\FF_r)$ is isomorphic to the group of homotopy equivalences of a graph whose fundamental group is isomorphic to $\FF_r$, i.e., $\Out(\FF_r)$ can be thought of as the mapping class group of a graph, and (ii) the Dehn--Nielsen--Baer theorem which states that the extended mapping class group $\MCG^\pm(S_g)$ (which also includes isotopy classes of orientation reserving homeomorphisms) is isomorphic to $\Out(\pi_1(S_g))$~\cite{bk:Dehn87}.  This analogy has led to much fruitful research on the outer automorphism group of a free group $\Out(\FF_r)$.  

The metrics on the Culler--Vogtmann outer space we consider in this paper are analogs to the classical Weil--Petersson metric on the Teichm\"uller space of an orientable surface. The Weil--Petersson metric has been studied extensively from the point of view of both geometric analysis and geometric group theory. On the one hand, it enjoys many important analytic properties which can be expressed naturally in terms of hyperbolic geometry on $S_g$. Its utility in geometric group theory then stems from the fact that every isometry of the Weil--Petersson metric is induced by a mapping class~\cite{MasurWolf}. Thus, the action of $\MCG(S_g)$ on the Teichm\"uller space equipped with the Weil--Petersson metric encodes information about useful invariants for mapping classes. 
 
As the piecewise Riemannian metrics on the Culler--Vogtmann outer space that we study in this paper are motivated by the classical Weil--Petersson metric on the Teichm{\"u}ller space of a closed surface, it is natural to ask to what extent they are true analogs of the Weil--Petersson metric. A major takeaway from the work in this paper is that they should be seen as natural analogs from the geometric analysis point of view, but not from the geometric group theory perspective. Specifically, while we highlight some similarities between these metrics and the Weil--Petersson metric as seen from the analytic point of view (Theorems~\ref{th:entropy metric not complete} and \ref{th:completion rose})  the main result (Theorem~\ref{th:entropy bounded}) of this paper shows that from the geometric group theoretic perspective, these metrics are not useful (except possibly when $r=3$).  The content of this theorem is summarized below.
\begin{quote}
\emph{The action of $\Out(\FF_r)$ on the metric completion of the Culler--Vogtmann outer space has a fixed point for $r \geq 4$}.    
\end{quote}
The remainder of the introduction discusses these metrics more thoroughly and provides context for the main results.

\subsection{Metrics on Outer Space}\label{subsec:intro metrics}
The topology of $CV(\FF_r)$ has been well-studied; see for instance the survey papers of Bestvina~\cite{pro:Bestvina02} and Vogtmann~\cite{ar:Vogtmann02}.  The metric theory of $CV(\FF_r)$ has been steadily developing over the past decade.  What is desired is a theory that reflects the dynamical properties of the natural action by $\Out(\FF_r)$, that further elucidates the connection between $\Out(\FF_r)$ and $\MCG(S)$, and that leads to useful new discoveries.

The metric that has received the most attention to date is the \emph{Lipschitz metric}.  Points in the Culler--Vogtmann outer space are represented by triples $(G,\rho,\ell)$ where $G$ is a finite connected graph, $\rho \from \calR_r \to G$ is a homotopy equivalence where $\calR_r$ is the $r$--rose, and $\ell$ is a function from the edges of $G$ to $(0,\infty)$ for which the sum of $\ell(e)$ over all edges of $G$ is equal to $1$.  (See Section~\ref{subsec:outer space} for complete details.)  We think of the function $\ell$ as specifying the length of each edge and as such $\ell$ determines a metric on $G$ where the interior of each edge $e$ is locally isometric to the interval $(0,\ell(e))$.  The Lipschitz distance between two unit volume marked metric graphs $(G_1,\rho_1,\ell_1)$ and $(G_2,\rho_2,\ell_2)$ in $CV(\FF_r)$ is defined by:
\begin{equation}\label{eq:lipschitz}
d_{\Lip}\bigl((G_1,\rho_1,\ell_1),(G_2,\rho_2,\ell_2)\bigr) = \log \inf \{ \Lip(f) \mid f \from G_1 \to G_2, \, \rho_2 \simeq f \circ \rho_1 \}.    
\end{equation}
In the above, $\Lip(f)$ is the Lipschitz constant of the function $f \from G_1 \to G_2$ using the metrics induced by $\ell_1$ and $\ell_2$ respectively.  In general the function $d_{\Lip}$ is not symmetric.  As such, $d_{\Lip}(\param,\param)$ is not a true metric, but an \emph{asymmetric metric}.  See~\cite{ar:Algom-Kfir11,ar:A-KB12,ar:FM11} for more on the asymmetric aspects of the Lipschitz metric.  

Regardless, the Lipschitz metric has been essential in several recent developments for $\Out(\FF_r)$.  This is in part due to the fact that the Lipschitz metric connects the dynamical properties of an outer automorphism of $\FF_r$ acting on $CV(\FF_r)$ to its action on conjugacy classes---of elements and of free factors---in $\FF_r$.  Notably is the ``Bers-like proof'' of the existence of train-tracks by Bestvina~\cite{ar:Bestvina11}, the proof of hyperbolicity of the free factor complex by Bestvina--Feighn~\cite{ar:BF14-1}, and the proof of hyperbolicity of certain free group extensions by Dowdall--Taylor~\cite{ar:DT18}.  

In this way, the Lipschitz metric is akin to the Teichm\"uller metric on Teichm\"uller space which was used to prove the corresponding statements for the mapping class group~\cite{ar:Bers78,ar:MM99,ar:FM02}.  One can also define the Lipschitz metric on Teichm\"uller space using the same idea as in~\eqref{eq:lipschitz}, and in this setting it is oftentimes called \emph{Thurston's asymmetric metric}~\cite{un:Thurston}.  This metric has seen renewed attention lately, in part due to the usefulness of the Lipschitz metric on $CV(\FF_r)$.

As a result of McMullen's interpretation of the Weil--Petersson metric on Teichm\"uller space via tools from the thermodynamic formalism applied to the geodesic flow on the hyperbolic surface~\cite[Theorem~1.12]{ar:McMullen08}, there exists a natural candidate for the Weil--Petersson metric on the Culler--Vogtmann outer space.  This idea was originally pursued by Pollicott--Sharp~\cite{ar:PS14}.

\subsection{Thermodynamic Metrics}\label{subsec:intro thermo}

The metrics we consider in this paper arise from the tools of the thermodynamic formalism as developed by Bowen~\cite{bk:Bowen08}, Parry--Pollicott~\cite{ar:PP90}, Ruelle~\cite{bk:Ruelle78} and others.  The central objects involved are the notions of \emph{entropy} and \emph{pressure}.  For a graph $G$, these notions define functions:
\begin{equation}
\entropy_G \from \calM(G) \to \RR \quad \text{and} \quad \pressure_G \from \RR^n \to \RR
\end{equation}
where  $n$ is the number of (geometric) edges in $G$ and $\calM(G) = \RR_{>0}^n$---this space parametrizes the length functions on $G$.  The entropy and pressure functions are real analytic, strictly convex and are related by $\entropy_G(\ell) = 1$ if and only if $\pressure_G(-\ell) = 0$ (see Theorem~\ref{thm:pressure = 0 at entropy}).  As these functions are smooth and strictly convex, their Hessians induce an inner product on the tangent space of the unit entropy subspace $\calM^1(G) = \{ \ell \in \calM(G) \mid \entropy_G(\ell) = 1 \}$ at a length function (see Definition~\ref{def:metrics}).  Hence, the notions of entropy and pressure induce Riemannian metrics on $\calM^1(G)$ which we call the \emph{entropy metric} and \emph{pressure metric} respectively.  By $d_{\entropy,G}$ and $d_{\pressure,G}$ we denote the induced distance functions on $\calM^1(G)$.  We caution the reader that these metrics have been considered by others with conflicting terminology.  Throughout this introduction, we will use the above terminology even when referencing the work of others.  See Remark~\ref{rem:metrics} for a further discussion.

Pollicott--Sharp initiated the study of the thermodynamic metrics when they first defined the pressure metric on $\calM^1(G)$~\cite{ar:PS14}.  They proved that the pressure metric is not complete for the 2--rose $\calR_2$ and they derived formulas for the sectional curvature for the theta graph $\Theta_2$ and barbell graph $\calB_2$ (see Figure~\ref{fig:rank two graphs} for these graphs).  Additionally, Pollicott--Sharp produce a dynamical characterization of the entropy metric in terms of generic geodesics similar to Wolpert's result for the Weil--Petersson metric~\cite{ar:Wolpert86} (see Remark~\ref{rem:metrics}).  Kao furthered these results by showing that the pressure metric is incomplete for $\Theta_2$, $\calB_2$ and the 3--rose $\calR_3$, and by showing that the entropy metric is complete for $\calR_2$~\cite{ar:Kao17}.  Additionally, he derives formulas for the sectional curvature with respect to both the entropy and the pressure metric for $\Theta_2$, $\calB_2$ and $\calR_3$.  Xu shows that for certain graphs, the moduli space $\calM^1(G)$ equipped with the entropy metric arises in the completion of the Teichm\"uller space of a surface with boundary using the pressure metric~\cite{ar:Xu19}.

In this paper, we will investigate the entropy metric not only on a the moduli space of a single graph, but on the full moduli space of all marked graphs.  Let $\calX(\FF_r)$ be the space of marked metric graphs so that contained in $\calX(\FF_r)$ is the Culler--Vogtmann outer space $CV(\FF_r)$.  The notion of entropy extends to $\calX(\FF_r)$ by $\entropy([(G,\rho,\ell)]) = \entropy_G(\ell)$  and we set 
\begin{equation}
\calX^1(\FF_r) = \{ [(G,\rho,\ell)] \in \calX(\FF_r) \mid \entropy([(G,\rho,\ell)]) = 1 \}.
\end{equation}
There is a homeomorphism between $CV(\FF_r)$ and $\calX^1(\FF_r)$ defined by scaling the length function (see Section~\ref{subsec:entropy}).  Fixing a graph $G$ and a marking $\rho \from \calR_r \to G$, the map $\calM^1(G) \to \calX^1(\FF_r)$, that sends a length function $\ell$ in $\calM^1(G)$ to the point determined by $(G,\rho,\ell)$ in $\calX^1(\FF_r)$ is an embedding whose image we denote by $\calX^1(G,\rho)$.  Considering all marked graphs individually, this induces a piecewise Riemannian metric on $\calX^1(\FF_r)$.  See Section~\ref{subsec:metrics} for complete details.  We denote the induced distance function on $\calX^1(\FF_r)$ by $d_\entropy$.

For a closed orientable surface $S_g$, one can repeat the above discussion using the moduli space of marked Riemannian metrics with constant curvature $\calX(S_g)$.  In this case, the entropy and the area of the Riemann surface are directly related.  In particular, the unit entropy constant curvature metrics correspond to the hyperbolic metrics, i.e., those with constant curvature equal to $-1$, and hence to those with area equal to $2\pi(2g - 2)$.  In other words, the entropy and area normalizations on $\calX(S_g)$ result in the same subspace, the Teichm\"uller space of the surface.  McMullen proved that the ensuing entropy metric on the Teichm\"uller space is proportional to the Weil--Petersson metric~\cite[Theorem~1.12]{ar:McMullen08}.  

It is this connection between the entropy metric and the Weil--Petersson metric that drives the research in this paper.  After introducing the framework for both the entropy and the pressure metrics in Section~\ref{sec:thermo}, we specialize the discussion to the entropy metric because of this connection to the Weil--Petersson metric.  All of the main results of this paper have analogous statements for the pressure metric and the proofs are similar, and in most cases substantially easier.  The statements for the pressure metric are given in Section~\ref{subsec:intro pressure}. 

It is not necessary for this paper, but we mention that building on work of McMullen,  Bridgeman~\cite{ar:Bridgeman10} and Bridgeman--Canary--Labourie--Samborino~\cite{ar:BCLS15} used these same ideas to define a metric on the space of conjugacy classes of regular irreducible representations of a hyperbolic group into a special linear group.       

In the next three sections, we explain our main results on the entropy metric on $\calX^1(\FF_r)$ and their relation to the Weil--Petersson metric on Teichm\"uller space.

\subsection{Incompletion of the Metric}\label{subsec:intro incompletion}

Our first main result concerns the completion of the entropy metric on $\calX^1(\FF_r)$.  Wolpert showed that the Weil--Petersson metric on Teichm\"uller space is not complete~\cite{ar:Wolpert75}.  Our first theorem shows that when $r \geq 3$, the same holds for the entropy metric on $\calX^1(\FF_r)$.

\begin{theorem}\label{th:entropy metric not complete}
The metric space $(\calX^{1}(\FF_r),d_{\entropy})$ is complete if $r = 2$ and incomplete if $r \geq 3$.
\end{theorem}

For $r \geq 3$, this theorem is proved by exhibiting a finite length path in $\calM^1(\calR_r)$ that exits every compact subset.  This path is defined by sending the length of one edge in $\calR_r$ to infinity, while shrinking the others to maintain unit entropy (Proposition~\ref{prop:finite length path}).  This strategy---also used by Pollicott--Sharp~\cite{ar:PS14} and Kao~\cite{ar:Kao17}---shows that $(\calM^1(\calR_r),d_{\entropy,\calR_r})$ is incomplete.  We further show in Section~\ref{sec:not complete} that this path also exits every compact set in $\calX^1(\FF_r)$.  As $d_{\entropy,\calR_r}$ is an upper bound to $d_\entropy$ (when defined), this path still has finite length when considered in $\calX^1(\FF_r)$ and thus Theorem~\ref{th:entropy metric not complete} follows.   

The path described above illustrates a general method for producing paths that exit every compact subset and that have finite length in the entropy metric: deform the metric by sending the length of some collection of edges to infinity while shrinking the others to maintain unit entropy. So long as the complement of the collection supports a unit entropy metric, this path will have finite length.  This explains why $(\calX^1(\FF_2),d_\entropy)$ is complete: any metric on a graph where every component has rank at most 1 has entropy equal to zero.  In Section~\ref{sec:rank 2}, we demonstrate the calculations required to prove that $(\calX^1(\FF_2),d_\entropy)$ is complete.  

This is completely analogous to the setting of the Weil--Petersson metric on Teichm\"uller space.  In that setting, deforming a hyperbolic metric on $S_g$ by pinching a simple closed curve results in a path with finite length that exits every compact set.  Moreover, the geometric analysis agrees.  For the path in $\calM^1(\calR_r)$ described above, if we parametrize the long edge by $-\log(t)$ as $t \to 0$, then the entropy norm along this path is $O(t^{-1/2})$ as shown in Proposition~\ref{prop:finite length path}.  For the path in the Teichm\"uller space of $S_g$, if we parametrize the curve which is being pinched by $t$ as $t \to 0$, then the Weil--Petersson norm along this path is $O(t^{-1/2})$~\cite[Section~7]{col:Wolpert09}.  Note that in this case, the length of the shortest curve that intersects the pinched one has length approximately $-\log(t)$.

\subsection{The Moduli Space of the \texorpdfstring{$r$--rose}{r-rose}}\label{subsec:intro rose}

Our second main result is concerned with the entropy metric on the moduli space of the $r$--rose $\calR_r$.    

\begin{theorem}\label{th:completion rose}
The completion of $(\calM^{1}(\calR_{r}),d_{\entropy,\calR_r})$ is homeomorphic to the complement of the vertices of an $(r-1)$--simplex.
\end{theorem}

The space $\calM^1(\calR_r)$ is homeomorphic to the interior of an $(r-1)$--simplex.  The faces added in the completion for $d_{\entropy,\calR_r}$ correspond to unit entropy metrics on subroses.  Such a metric is obtained as the limit of a sequence of length functions on $\calR_r$ by sending the length of a collection of edges to infinity and scaling the others to maintain unit entropy.  Specifically, a $(k-1)$--dimensional face of the completion corresponds to the moduli space of unit entropy metrics on a sub--$k$--rose.  As before, intuitively, the vertices of the $(r-1)$--simplex are missing in the completion as there does not exist a unit entropy metric on $\calR_1$. 

That unit entropy metrics on subroses arise as points in the completion follows from the calculations provided for the proof of incompleteness in Theorem~\ref{th:entropy metric not complete} and some continuity arguments.  This is shown in Section~\ref{subsec:model}.  The difficult part of the proof of Theorem~\ref{th:completion rose} is showing that any path in $(\calM^1(\calR_r),d_{\entropy,\calR_r})$ that sends the length of one edge to 0 (and hence the lengths of the other edges to infinity) necessarily has infinite length.  This argument appears in Lemma~\ref{lem:square root bound} and Proposition~\ref{prop:infinite length}.  In Section~\ref{subsec:proof rose}, we combine these two facts to complete the proof of Theorem~\ref{th:completion rose}.  

In Example~\ref{ex:completion 2,3}, we compare the completion of $(\calM^1(\calR_3),d_{\entropy,\calR_3})$ to the closure of the unit volume metrics on $\calR_3$ in the axes topology on $CV(\FF_3)$ (see Section~\ref{subsec:outer space} for definitions).  By Theorem~\ref{th:completion rose}, the completion in the entropy metric is a $2$--simplex without vertices whereas the closure in the axes topology is a $2$--simplex.  More interestingly, the newly added edges and vertices are dual: edges in the entropy completion correspond to vertices in the axes closure and the missing vertices in the entropy completion correspond to the edges in the axes closure.  This is explained in detail in Example~\ref{ex:completion 2,3} and illustrated in Figure~\ref{fig:completion}.     

While it is not necessary for Theorem~\ref{th:completion rose}, we mention that in Section~\ref{subsec:rose-thin} we prove that the diameter of a cross-section of the $(r-1)$--simplex goes to 0 as the length of one of the edges goes to 0, i.e., as the cross-section moves out toward one of the missing vertices.  In other words, the completion of $(\calM^1(\calR_r),d_{\entropy,\calR_r})$ is geometrically akin to an ideal hyperbolic $(r-1)$--simplex; see Lemma~\ref{lem:thin part}. 

\subsection{A Fixed Point in the Completion}\label{subsec:intro bounded}

Whereas the first two main results demonstrate the similarity between the geometric analysis for the Weil--Petersson metric on the Teichm\"uller space and the entropy metric on the Culler--Vogtmann outer space, our final main result provides a stark contrast between these two metrics with respect to geometric group theory.

\begin{theorem}\label{th:entropy bounded}
The subspace $(\calX^1(\calR_{r},{\rm id}) \cdot \Out(\FF_r),d_\entropy) \subset (\calX^1(\FF_r),d_\entropy)$ is bounded if $r \geq 4$.  Moreover, the action of $\Out(\FF_r)$ on the completion of $(\calX^1(\FF_r),d_\entropy)$ has a fixed point.
\end{theorem}

This subspace consists of the unit entropy metrics on every marked $r$--rose.  To illustrate the difference with the setting of the Weil--Petersson metric on Teichm\"uller space, we mention the fact due to Daskalopoulos--Wentworth that pseudo-Anosov mapping classes have positive translation length in their action on the Teichm\"uller space ~\cite{daskwent}. In particular, the action of the mapping class group does not have a fixed point in the completion of Teichm\"uller space with the Weil--Petersson metric.

The first step in the proof of Theorem~\ref{th:entropy bounded} is to show that the image of the inclusion map $\calM^1(\calR_r) \to \calX^1(\calR_r,{\rm id}) \subset \calX^1(\FF_r)$ has bounded diameter for $r \geq 4$.  This result is particularly striking in contrast to Theorem~\ref{th:completion rose}, since implicit in that theorem is that the space $(\calM^1(\calR_r),d_{\entropy,\calR_r})$ has infinite diameter.  Boundedness of the image of $\calM^1(\calR_r)$ is achieved by showing that the map induced via Theorem~\ref{th:completion rose}, $\Phi \from \Delta^{r-1} - V \to \ccalX^1(\FF_r)$, extends to $\Delta^{r-1}$, where $\Delta^{r-1}$ is a $(r-1)$--simplex, $V \subset \Delta^{r-1}$ is the set of vertices, and $\ccalX^1(\FF_r)$ is the completion of $\calX^1(\FF_r)$ for $d_\entropy$.  The existence of this extension shows that $\calX^1(\calR_r,{\rm id})$ lies in a compact set and hence is bounded.   

In order to show that $\Phi \from \Delta^{r-1} - V \to \ccalX^1(\FF_r)$ extends to the set $V$, we show that $\Phi$ maps every $(k-1)$--dimensional face of $\Delta^{r-1} - V$ to a single point when $1 < k < r-1$.  This is shown in Section~\ref{sec:rose bounded}.  This fact, plus the previously mentioned fact about the diameter of the cross-sections going to 0, gives that $\Phi$ extends to the set $V$ and that the entire $(r-3)$--skeleton of $\Delta^{r-1}$ is mapped to a single point in $\ccalX^1(\FF_r)$.  The collapse of a $(k-1)$--dimensional face of $\Delta^{r-1} - V$ for $1 < k < r-1$ arises from paths in $\calX^1(\FF_r)$ connecting points in $\calX^1(\calR_r,{\rm id})$ whose length is much shorter than paths in $\calM^1(\calR_r)$ connecting the same points.  In other words, there are shortcuts present in $\calX^1(\FF_r)$ that are not present in $\calM^1(\calR_r)$.

These shortcuts are most easily understood in terms of unit entropy metrics on marked subgraphs, i.e., points in the completion of $(\calX^1(\FF_r),d_\entropy)$.  Pathologies arise when the subgraph is not connected.  In this case, the entropy of the metric on the subgraph is the maximum of the entropy---in the previous sense---on a component of the subgraph.  Hence, by holding the length function constant on a component of the subgraph with unit entropy, we are free to modify the length function on the other components at will, so long as the entropy is never greater than 1 on 
any of these components.  In Proposition~\ref{prop:metrics} we show that the entropy and pressure metrics can be computed using the second derivatives of the lengths of edges along a path.  Hence, the length of a path that changes the length of the edges in a component with entropy less than 1 linearly has zero length in either of these metrics.

Figure~\ref{fig:shortcut} illustrates the central idea that is exploited in Section~\ref{sec:separating} to show that many paths have zero length.  This figure is taking place in the completion of $\calM^1(G_{2,2})$ using the metric $d_{\entropy,G_{2,2}}$.  This space has $\calM^1(\calR_4)$ as a face in $\calX^1(\FF_4)$ corresponding to the collapse of the separating edge.  The completion of $\calM^1(\calR_4)$ has an edge corresponding to unit entropy metrics on two of the edges (denoted $a$ and $b$ in Figure~\ref{fig:shortcut}).  This edge also corresponds to a subset in the completion of $\calM^1(G_{2,2})$.  Illustrated in Figure~\ref{fig:shortcut} is a path through unit entropy length functions on subgraphs of $G_{2,2}$, i.e., points in the completion.  As all edge lengths are changing linearly, this path has length 0 and hence all of these length functions correspond to the same point in the completion.  

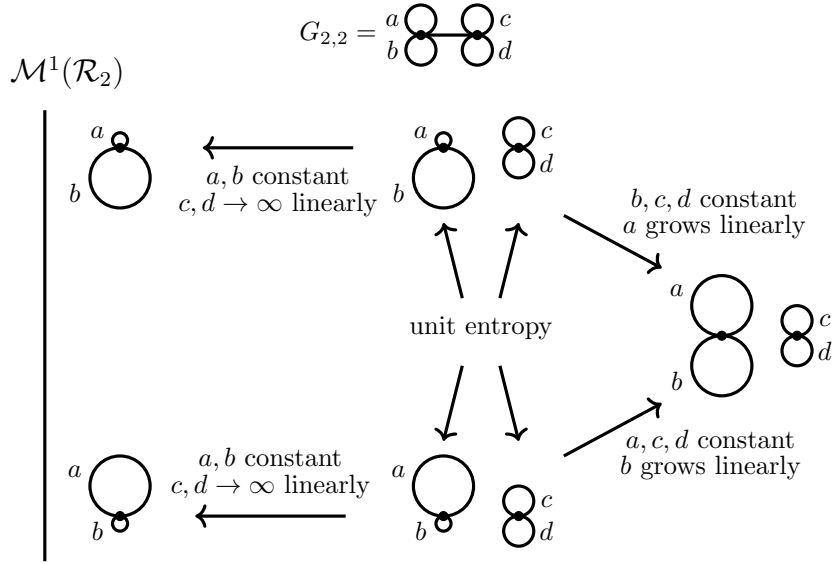
\begin{figure}[ht]
\centering
\begin{tikzpicture}
\draw[very thick] (0,-3) -- (0,3);
\node at (0.3,3.5) {$\calM^{1}(\calR_{2})$};
\begin{scope}[xshift=5cm,yshift=4cm]
\filldraw[black] (0.75,0) circle [radius=0.6mm];
\filldraw[black] (0,0) circle [radius=0.6mm];
\draw[very thick] (0,0.2) circle [radius=2mm];
\draw[very thick] (0,-0.2) circle [radius=2mm];
\draw[very thick] (0.75,0.2) circle [radius=2mm];
\draw[very thick] (0.75,-0.2) circle [radius=2mm];
\draw[very thick] (0,0) -- (0.75,0);
\node at (-0.38,0.2) {\footnotesize $a$};
\node at (-0.38,-0.2) {\footnotesize $b$};
\node at (1.12,0.2) {\footnotesize $c$};
\node at (1.12,-0.2) {\footnotesize $d$};
\node at (-1.1,0) {\footnotesize $G_{2,2} = $};
\end{scope}
\begin{scope}[xshift=1cm,yshift=2.5cm]
\filldraw[black] (0,0) circle [radius=0.6mm];
\draw[very thick] (0,0.1) circle [radius=1mm];
\draw[very thick] (0,-0.4) circle [radius=4mm];
\node at (-0.3,0.2) {\footnotesize $a$};
\node at (-0.6,-0.6) {\footnotesize $b$};
\end{scope}
\begin{scope}[xshift=1cm,yshift=-2.4cm]
\draw[very thick,->] (3,0) -- (1,0); 
\node at (2,0.75) {\footnotesize $a,b$ constant};
\node at (2,0.4) {\footnotesize $c,d \to \infty$ linearly};
\filldraw[black] (0,0) circle [radius=0.6mm];
\draw[very thick] (0,0.4) circle [radius=4mm];
\draw[very thick] (0,-0.1) circle [radius=1mm];
\node at (-0.6,0.6) {\footnotesize $a$};
\node at (-0.3,-0.2) {\footnotesize $b$};
\end{scope}
\begin{scope}[xshift=5.3cm,yshift=2.5cm]
\draw[very thick,<-] (-3.2,0) -- (-1.2,0); 
\node at (-2.2,-0.4) {\footnotesize $a,b$ constant};
\node at (-2.2,-0.75) {\footnotesize $c,d \to \infty$ linearly};
\filldraw[black] (0,0) circle [radius=0.6mm];
\draw[very thick] (0,0.1) circle [radius=1mm];
\draw[very thick] (0,-0.4) circle [radius=4mm];
\node at (-0.3,0.2) {\footnotesize $a$};
\node at (-0.6,-0.6) {\footnotesize $b$};
\filldraw[black] (1,0) circle [radius=0.6mm];
\draw[very thick] (1,0.2) circle [radius=2mm];
\draw[very thick] (1,-0.2) circle [radius=2mm];
\node at (1.37,0.2) {\footnotesize $c$};
\node at (1.37,-0.2) {\footnotesize $d$};
\draw[very thick,<-] (0,-1) -- (0.25,-2);
\draw[very thick,<-] (1,-1) -- (0.75,-2);
\node at (0.5,-2.4) {\footnotesize{unit entropy}};
\end{scope}
\begin{scope}[xshift=9cm]
\filldraw[black] (0,0) circle [radius=0.6mm];
\draw[very thick] (0,0.4) circle [radius=4mm];
\draw[very thick] (0,-0.4) circle [radius=4mm];
\node at (-0.6,0.6) {\footnotesize $a$};
\node at (-0.6,-0.6) {\footnotesize $b$};
\filldraw[black] (1,0) circle [radius=0.6mm];
\draw[very thick] (1,0.2) circle [radius=2mm];
\draw[very thick] (1,-0.2) circle [radius=2mm];
\node at (1.37,0.2) {\footnotesize $c$};
\node at (1.37,-0.2) {\footnotesize $d$};
\draw[very thick,->] (-2.1,1.6) -- (-0.8,0.9);
\node at (-0.1,1.8) {\footnotesize $b,c,d$ constant};
\node at (-0.1,1.45) {\footnotesize $a$ grows linearly};
\end{scope}
\begin{scope}[xshift=5.3cm,yshift=-2.4cm]
\filldraw[black] (0,0) circle [radius=0.6mm];
\draw[very thick] (0,0.4) circle [radius=4mm];
\draw[very thick] (0,-0.1) circle [radius=1mm];
\node at (-0.6,0.6) {\footnotesize $a$};
\node at (-0.3,-0.2) {\footnotesize $b$};
\filldraw[black] (1,0) circle [radius=0.6mm];
\draw[very thick] (1,0.2) circle [radius=2mm];
\draw[very thick] (1,-0.2) circle [radius=2mm];
\node at (1.37,0.2) {\footnotesize $c$};
\node at (1.37,-0.2) {\footnotesize $d$};
\draw[very thick,<-] (2.9,1.5) -- (1.6,0.8);
\node at (3.55,1.0) {\footnotesize $a,c,d$ constant};
\node at (3.55,0.65) {\footnotesize $b$ grows linearly};
\draw[very thick,<-] (0,1) -- (0.25,2);
\draw[very thick,<-] (1,1) -- (0.75,2);
\end{scope}
\end{tikzpicture}
\caption{Illustration of a path with length 0 in the completion of $(\calM^1(G_{2,2}),d_{\entropy,G_{2,2}})$.}\label{fig:shortcut}
\end{figure}

This shows that the edge corresponding to $\calM^1(\calR_2)$ is mapped by $\Phi$ to a point.  The same idea works for any sub--$k$--rose of $\calR_r$ so long as $1 < k < r-1$: it is necessary to separate two subroses, each of which supports a unit entropy metric.  This is the reason we require $r \geq 4$ in Theorem~\ref{th:entropy bounded}. 

Once we know that the entire $(r-3)$--skeleton of $\Delta^{r-1}$ is mapped  by $\Phi$ to a point in $\ccalX^1(\FF_r)$, we utilize the structure of the Culler--Vogtmann outer space to conclude in Section~\ref{sec:entropy bounded} that this point is independent of the marking $\rho \from \calR_r \to \calR_r$ used to define the inclusion $\calM^1(\calR_r) \to \calX^1(\FF_r)$.  This completes the proof of Theorem~\ref{th:entropy bounded}.

\subsection{Analogous Statements for Pressure Metric}\label{subsec:intro pressure}

For the pressure metric on $\calX^1(\FF_r)$ we have the following analogs of Theorems~\ref{th:entropy metric not complete},~\ref{th:completion rose} and~\ref{th:entropy bounded}.  By $d_\pressure$ we denote the induced distance function.
\begin{enumerate}
\item The space $(\calX^1(\FF_r),d_\pressure)$ is incomplete for $r \geq 2$.
\item The completion of $(\calM^{1}(\calR_{r}),d_{\pressure,\calR_r})$ is homeomorphic to an $(r-1)$--simplex.
\item The space $(\calX^1(\FF_r),d_\pressure)$ is bounded if $r \geq 2$ and moreover, the action of $\Out(\FF_r)$ on the completion of $(\calX^1(\FF_r),d_\pressure)$ has a fixed point.
\end{enumerate}
These can be shown using techniques similar---and simpler---to those in this article.  The key source of the distinction between the entropy and pressure metrics is that the length function that assigns 0 to the unique edge on $\calR_1$ has pressure equal to 0 even through the entropy is not defined.  Hence the path in $\calM^1(\calR_2)$ that sends the length of one edge to infinity while shrinking the length of the other (necessarily to 0) to maintain unit entropy has finite length in the pressure metric, where as the length in the entropy metric is infinite.

\subsection{Further Discussion and Questions}\label{subsec:intro discussion}

This work raises a number of questions.

\medskip 

Our proof that the action of $\Out(\FF_r)$ on the completion of $(\calX^1(\FF_r),d_\entropy)$ has a fixed point relies heavily on the assumption that $r \geq 4$: the key construction uses an edge that separates a given graph into two subgraphs, each with rank at least $2$. This leaves the door open to a negative answer for the following question, which would allow for interesting applications specifically for $\FF_3$:

\begin{question}\label{que:rank 3}
Does $(\calX^1(\FF_3),d_\entropy)$ admit an $\Out(\FF_3)$--invariant bounded subcomplex?
\end{question}

Theorem~\ref{th:entropy bounded} demonstrates the existence of an $\Out(\FF_r)$ orbit in $(\calX^1(\FF_r),d_{\entropy})$ with bounded diameter but we do not yet know that the entire space has bounded diameter. We therefore ask:  

\begin{question}\label{que:bounded}
Is $(\calX^1(\FF_r),d_\entropy)$ bounded for $r \geq 4$?
\end{question}

We believe the answer to this question is yes.  Indeed, the only way $(\calX^1(\FF_r),d_\entropy)$ could fail to be bounded is if the subspace $(\calX^1(G,\rho),d_\entropy)$ has infinite diameter for some marked graph $\rho \from \calR_r \to G$.  As the diameter of $(\calX^1(\calR_r,{\rm id}) \cdot \Out(\FF_r),d_\entropy)$ is bounded, to answer the question in the affirmative, it would suffice to find a bound (in terms of $r$) on the distance from any point $\calX^1(\FF_r)$ to a point in $\calX^1(\calR_r,{\rm id}) \cdot \phi$ for some $\phi \in \Out(\FF_r)$.  Another approach to answer Question~\ref{que:bounded} in the affirmative would be to show the existence of a bound (in terms of $r$) on distance from any point in $\calX^1(G, \rho)$ to a completion point represented by a unit entropy metric on a proper subgraph (with the goal of getting to a point in the completion of a marked rose via induction). This led us to the following question, which is of independent interest and we pose here as a conjecture:

\begin{conjecture}\label{con:entropy bound}
For any $r \geq 3$ there exists $C > 0$ so that any metric graph of rank $r$ with unit entropy contains a proper subgraph with entropy at least $C$.
\end{conjecture} 

It suffices to show the conjecture for a fixed topological type of graph since for a given rank $r$, there are only finitely many topological types of graph of rank $r$.

One can also define the notion of the entropy metric on the Teichm{\"u}ller space of a surface with boundary.  In~\cite{ar:Xu19}, Xu shows that this metric is incomplete.  As mentioned previously, McMullen proved that for closed surfaces, the entropy metric is a constant multiple of the Weil--Petersson metric.  However, by partially characterizing the completion of the entropy metric in the bordered setting, Xu is able to show that this is not true in the presence of boundary.  Concretely, Xu identifies certain graphs $G$ so that, in the notation of this paper, $(\calM^{1}(G),d_{\entropy})$ isometrically embeds in the completion of the Teichm{\"u}ller space of the surface equipped with the entropy metric. We therefore ask if the work in this paper can be used to fully understand the completion of the Teichm{\"u}ller space of a bordered surface equipped with the entropy metric. 

\begin{problem}\label{prob:completion}
Fully characterize the completion of $(\calM^{1}(G), d_{\entropy,G})$ for an arbitrary graph $G$ and use this to study the completion of the Teichm{\"u}ller space of a bordered surface, equipped with the entropy metric.
\end{problem}

The pathology exhibited by Theorem~\ref{th:entropy bounded} relies on the existence of a sequence of unit entropy length functions whose limiting metric is supported on a subgraph with multiple components where the metric on some component need not have entropy equal to 1. This behavior does not occur in the Teichm\"uller space of a closed surface since unit entropy is equivalent to the hyperbolic metric for constant curvature surfaces, and thus for the subsurface supporting the limit of a sequence of unit entropy metrics, the metric on each component also has entropy equal to 1. One can also consider an entropy function defined over the moduli space of singular flat metrics on a closed surface. This setting appears more similar to the situation of Theorem ~\ref{th:entropy bounded} in that the unit entropy condition is not encoded by the local geometry. It appears likely that some version of Theorem~\ref{th:entropy bounded} holds for singular flat metrics, and so we therefore ask:  

\begin{question}\label{que:flat}
Can the techniques used in this paper in the setting of metric graphs apply to the study of an entropy metric on the moduli space of singular flat metrics on a closed surface? 
\end{question}



\subsection{Acknowledgements} The authors thank the Institute for Computational and Experimental Research in Mathematics (ICERM) for hosting the workshop \textit{Effective and Algorithmic Methods in Hyperbolic Geometry and Free Groups}, at which work on this project began. We also thank Jing Tao for suggesting Question~\ref{que:flat}. The first author is supported by NSF grant DMS-1807319. The second author is supported by Simons Foundation Grant No.~316383.  The third author is supported by Simons Foundation Grant No.~637880.


\section{Graphs and Outer Space}\label{sec:preliminaries}

In this section we introduce some concepts that are necessary for the sequel.  First, we set some notation for dealing with graphs.  Next, we define the Culler--Vogtmann outer space---including its topology---and the $\Out(\FF_r)$ action on this space.  


\subsection{Graphs}\label{subsec:graphs} We use Serre's convention for graphs~\cite{bk:Serre03}.  That is, an \emph{\textup{(}undirected\textup{)} graph} is a tuple $G = (V,E,o,\tau,\bar{\phantom{e}})$ where:
\begin{enumerate}
    \item $V$ and $E$ are sets, called the \emph{vertices} and the \emph{directed edges} (we think of $E$ as  containing two copies, with opposite orientations, of each undirected edge), 
\item  $o,\tau \from E \to V$ are functions that specify the \emph{originating} and \emph{terminating} vertices of an edge,
\item $\bar{\phantom{e}} \from E \to E$ is a fixed point free involution such that $o(e) = \tau(\bar{e})$ ($\bar{\phantom{e}}$ flips edges).
\end{enumerate}
We fix an \emph{orientation} on $G$, that is, a subset $E_{+} \subset E$ that contains exactly one edge from each pair $\{e,\bar{e}\}$.  
Since we consider the pair $\{e,\bar{e}\}$ to be a single edge, the number of edges of $G$ is $|E_{+}| = |E|/2$.  
The \em valance \em of a vertex $v$ 
is the number of edges from $e \in E_+$ with $o(e)=v$ 
plus the number of edges from $e \in E_+$ with $\tau(e)=v$ (an edge $e$ for which $o(e) = \tau(e) =v$ contributes $2$ to the valance). 
Oftentimes when defining a graph we only specify the edges in $E_+$ (together with the restrictions of $o$ and $\tau$ to $E_+$).  The complete set of edges is then defined as 
$E = E_+ \cup \overline{E}_+$, where 
$\overline{E}_+$ 
is a copy of $E_+$, and $o,\tau$, and $\bar{\phantom{e}}$ are defined in the obvious way.  We blur the distinction between the tuple $(V,E,o,\tau,\bar{\phantom{e}})$ and the corresponding one dimensional CW--complex with $0$--cells $V$ and 1--cells $E_+$.

The space of \emph{length functions on $G$} is the open convex cone
\begin{equation}
\calM(G) = \{ \ell \from E_{+} \to \RR_{>0} \}. 
\end{equation}
We consider this set as a subset of $\RR^{\abs{E_+}}$.  A length function $\ell \from E_{+} \to \RR_{>0}$ extends to a function $\ell \from E \to \RR_{>0}$ by $\ell(e) = \ell(\bar{e})$ if $e \notin E_{+}$.  By $\One \in \calM(G)$ we denote the constant function with value 1.

An edge path is a sequence of edges $(e_{1},\ldots,e_{n})$ in $E$ such that $\tau(e_{i}) = o(e_{i+1})$ for $i = 1,\ldots, n-1$.  A function $f \from E \to \RR$ (in particular a length function) extends to a function on edge paths $\gamma = (e_{1},\ldots,e_{n})$ by
\begin{equation}
f(\gamma) = \sum_{i=1}^{n} f(e_{i}).
\end{equation}



\subsection{Outer space}\label{subsec:outer space}  
We will introduce some definitions and notation for the Culler--Vogtmann outer space.  This space was originally defined by Culler and Vogtmann~\cite{ar:CV86}.  For more information, see for example the survey papers by Vogtmann~\cite{ar:Vogtmann02} or Bestvina~\cite{pro:Bestvina02}.

Let $\calR_r$ be the $r$--rose.  That is, $\calR_r$ the graph with a unique vertex $v$ and $r$ edges.  
Fix an isomorphism $\FF_r \cong \pi_1(\calR_r,v)$.  A \emph{marked metric graph \textup{(}of rank $r$\textup{)}} is a triple $(G,\rho,\ell)$ where 
\begin{enumerate}
\item $G$ is a finite connected graph without valence one or two vertices,
\item $\rho \from \calR_r \to G$ is a homotopy equivalence, and 
\item $\ell$ is a length function on $G$.
\end{enumerate}   There is an equivalence relation on the set of marked metric graphs defined by $(G_1,\rho_1,\ell_1) \sim (G_2,\rho_2,\ell_2)$ if there exists a graph automorphism $\alpha \from G_1 \to G_2$ such that $\ell_1 = \ell_2 \circ \alpha$ and such that the following diagram commutes up to homotopy:
\begin{equation*}
\xymatrix@R=3mm{ & G_1\ar[dd]^\alpha \\ \calR_r\ar[ur]^{\rho_1}\ar[dr]_{\rho_2} & \\ & G_2}
\end{equation*}

We let $\calX(\FF_r)$ denote the set of equivalence classes of marked metric graphs of rank $r$.  The group $\Out(\FF_r)$ acts on $\calX(\FF_r)$ on the right by precomposing the marking.  Specifically, for any outer automorphism $\phi \in \Out(\FF_r)$, there is a homotopy equivalence $g_\phi \from \calR_r \to \calR_r$ that induces $\phi$ on $\pi_1(\calR_r)$ via the aforementioned fixed isomorphism $\FF_r \cong \pi_1(\calR_r,*)$.  Moreover, this homotopy equivalence is unique up to homotopy.  With this, we define 
\begin{equation}
(G,\rho,\ell) \cdot \phi = (G,\rho \circ g_\phi,\ell).  
\end{equation}
This action respects the equivalence relation on marked metric graphs and so defines an action on $\calX(\FF_r)$ as claimed.  

Let $\calG_r$ denote the set of finite connected graphs without valence one or two vertices whose fundamental group has rank $r$.  We observe that this is a finite set.  Given a graph $G \in \calG_r$ and homotopy equivalence $\rho \from \calR_r \to  G$, we set
\[ \calX(G,\rho) = \{ [(G_0,\rho_0,\ell_0)] \in \calX(\FF_r) \mid G_0 = G \mbox{ and } \rho_0 \simeq \rho \}. \]    
There is a bijection $\calX(G,\rho) \to \calM(G)$ defined by $[(G_0,\rho_0,\ell_0)] \mapsto \ell_0$. These sets partition the set $\calX(\FF_r)$ and are permuted under the action by $\Out(\FF_r)$.  Specifically, for each $G \in \calG_r$ we fix a marking $\rho_G \from \calR_r \to G$.  Then   
\[ \calX(\FF_r) = \bigcup_{G \in \calG_r} \bigcup_{\phi \in \Out(\FF_r)} \calX(G,\rho_G) \cdot \phi. \]

There is a topology on $\calX(\FF_r)$ that is often defined in three different ways.  We will need to use the first two and for completeness we explain all three here.

\bigskip

\noindent
{\bf The weak topology.} 
The notion of a collapse induces a partial order on the set of marked graphs.  Specifically, for two graphs $G$ and $G_0$, we say that $G$ \emph{collapses} to $G_0$ if there is a surjection $c \from G \to G_0$ such that the image of any edge in $G$ is either a vertex or an edge of $G_0$ and such that $c^{-1}(x)$ is a contractible subgraph of $G$ for each point $x$ of $G_0$. The map $c$ is a called a \emph{collapse}.  Observe that if the map $c \from G \to G_0$ is a collapse, then a length function $\ell \in \calM(G_0)$ can be considered as a degenerate length function $\ell_{G_0}$ on $G$ by
\begin{equation}
\ell_{G_0}(e) = \begin{cases}
\ell(c(e)) & \mbox{ if $c(e)$ is an edge in $G_0$}, \\
0 & \mbox{ else}.
\end{cases}  
\end{equation}
This defines a map $c^* \from \calM(G_0) \to \RR_{\geq 0}^{\abs{E_+}}$ by $c^*(\ell) = \ell_{G_0}$.  We now define the following subset of $\RR_{\geq 0}^{\abs{E_+}}$:
\begin{equation}
\ocalM(G) = \bigcup_{c \from G \to G_0} c^*(\calM(G_0)).
\end{equation}  
We note the $\calM(G)$ is a subset of $\ocalM(G)$ as the identity map ${\rm id} \from G \to G$ is a collapse.

Next, given two marked graphs $\rho \from \calR_r \to G$ and $\rho_0 \from \calR_r \to G_0$, we say that $(G,\rho)$ \emph{collapses} to $(G_0,\rho_0)$ if there is a collapse $c \from G \to G_0$ such that $\rho_0 \simeq c \circ \rho$.  In this case we write $(G_0,\rho_0) \leq (G_,\rho)$.  We now define the following subset of $\calX(\FF_r)$:
\begin{equation}\label{eq:simplex closure}
\ocalX(G,\rho) = \bigcup_{(G_0,\rho_0) \leq (G,\rho)} \calX(G_0,\rho_0)
\end{equation}
The bijection $\calX(G,\rho) \to \calM(G)$ extends in a natural way to a bijection $\ocalX(G,\rho) \to \ocalM(G)$ and allows us to consider $\ocalX(G,\rho)$ as a subset of $\RR_{\geq 0}^{\abs{E_+}}$.  

The \emph{weak topology} is defined using this collection of subsets.  Specifically, a set $U \subseteq \calX(\FF_r)$ is open if $U \cap \ocalX(G,\rho)$ is open as a subset of $\RR_{\geq 0}^{\abs{E_+}}$ for all marked graphs $(G,\rho)$.

\bigskip

\noindent
{\bf The axes topology.}  Given a marked metric graph $(G,\rho,\ell)$ and an element $g \in \FF_r$, we denoted by $\ell([g])$ the $\ell$--length of the shortest loop in $G$ representing the conjugacy class $[\rho(g)]$.  This induces a function ${\rm Len} \from \calX(\FF_r) \to \RR_{\geq 0}^{\FF_r}$ where ${\rm Len}([(G,\rho,\ell)]) \from \FF_r \to \RR_{\geq 0}$ is the function defined by
\begin{equation*}
{\rm Len}([(G,\rho,\ell)])(g) = \ell([g]).
\end{equation*}
Culler--Morgan proved that the map ${\rm Len}$ is injective~\cite[3.7~Theorem]{ar:CM87}.  The resulting subspace topology on ${\rm Len}(\calX(\FF_r)) \subset \RR_{\geq 0}^{\FF_r}$ is called the \emph{axes topology}.  It is known that this topology agrees with the weak topology.  (See \cite[Section~1.1]{ar:CV86} or \cite[Proposition~5.4]{ar:GL07-2}.)  

\bigskip

\noindent
{\bf The equivariant Gromov--Hausdorff topology.}
We will not need this definition, and we only remark that Paulin showed that it is equivalent to the axes topology~\cite[Main~Theorem]{ar:Paulin89}.

\bigskip

There is an action of $\RR_{>0}$ on $\calX(\FF_r)$ given by scaling the length  function.  Specifically, $a \cdot (G,\rho,\ell) = (G,\rho,a\cdot \ell)$.  The quotient of $\calX(\FF_r)$ is denoted $\PP\calX(\FF_r)$.  

There are many continuous sections of the quotient map $\calX(\FF_r) \to \PP\calX(\FF_r)$.  An often used choice uses the notion of \emph{volume of a length function}, $\vol(\ell)$ that we define now.  For a length function $\ell \in \calM(G)$, we define the volume of $\ell$ by $\vol(\ell) = \sum_{e \in E_+}\ell(e)$.  There is a section $\calV \from \PP\calX(\FF_r) \to \calX(\FF_r)$ defined by
\[ \calV\left(\bigl[ [(G,\rho,\ell)]\bigr]\right) = [(G,\rho,\frac{1}{\vol(\ell)}\ell)]. \]
We denote the image of this section by $CV(\FF_r)$, it is known as the \emph{Culler--Vogtmann outer space}.  Further, given a marked graph $\rho \from \calR_r \to G$, we set $CV(G,\rho) = \calX(G,\rho) \cap CV(\FF_r)$.  This set is homeomorphic to an open simplex of dimension $\abs{E_+} - 1$.

\begin{example}\label{ex:cv2}
There are three graphs in $\calG_2$: the $2$--rose $\calR_2$, the theta graph $\Theta_2$ and the barbell graph $\calB_2$; see Figure~\ref{fig:rank two graphs}.  Figure~\ref{fig:CV2} shows a portion of $CV(\FF_2)$ and how these simplices piece together.  The homotopy equivalences used for Figure~\ref{fig:CV2} are as follows:
\begin{align*}
\rho_\Theta \from \calR_2 \to \Theta_2 & : e_1 \mapsto e_1\bar{e}_3, \, e_2 \mapsto e_2\bar{e}_3 \\
\rho_\calB \from \calR_2 \to \calB_2 & : e_1 \mapsto e_1, \, e_2 \mapsto e_3e_2\bar{e}_3 \\
a \from \calR_2 \to \calR_2 & : e_1 \mapsto e_1, \, e_2 \mapsto \bar{e}_2 \\
b \from \calR_2 \to \calR_2 & : e_1 \mapsto \bar{e}_2, \, e_2 \mapsto e_1\bar{e}_2
\end{align*}
Notice that $\rho_\Theta$ is the homotopy inverse to the map $\Theta_2 \to \calR_2$ that collapses the edge $e_3$.  Likewise $\rho_\Theta \circ b$ is the homotopy inverse to the collapse of $e_2$ and $\rho_\Theta \circ b^2$ is the homotopy inverse to the collapse of $e_1$.  Similary, $\rho_\calB$ is the homotopy inverse to the map $\calB_2 \to \calR_2$ that collapses $e_3$.
\end{example}

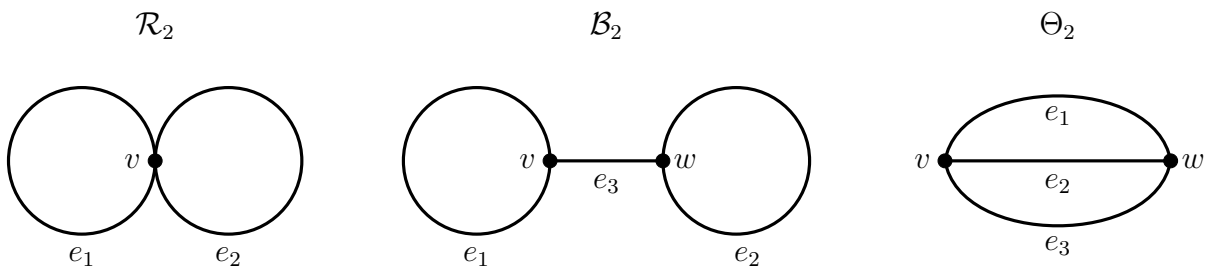
\begin{figure}[ht]
\centering
\begin{tikzpicture}[scale=1.5]
\node at (-1.5,1.2) {$\calR_2$};
\filldraw (-1.5,0) circle [radius=1.75pt];
\draw[very thick] (-2.15,0) circle[radius=0.65cm];
\draw[very thick] (-0.85,0) circle[radius=0.65cm];
\node at (-2.15,-0.85) {$e_1$};
\node at (-0.85,-0.85) {$e_2$};
\node at (-1.7,0) {$v$};
\begin{scope}[xshift=3.5cm]
\node at (-1,1.2) {$\calB_2$};
\draw[very thick] (-1.5,0) -- (-0.5,0);
\filldraw (-1.5,0) circle [radius=1.75pt];
\filldraw (-0.5,0) circle [radius=1.75pt];
\draw[very thick] (-2.15,0) circle[radius=0.65cm];
\draw[very thick] (0.15,0) circle[radius=0.65cm];
\node at (-2.15,-0.85) {$e_1$};
\node at (0.25,-0.85) {$e_2$};
\node at (-1,-0.2) {$e_3$};
\node at (-1.7,0) {$v$};
\node at (-0.3,0) {$w$};
\end{scope}
\begin{scope}[xshift=7cm]
\node at (-0.5,1.2) {$\Theta_2$};
\filldraw (-1.5,0) circle [radius=1.75pt];
\filldraw (0.5,0) circle [radius=1.75pt];
\draw[very thick] (-1.5,0) -- (0.5,0);
\draw[very thick] (-1.5,0) to[out=80,in=100] (0.5,0);
\draw[very thick] (-1.5,0) to[out=-80,in=-100] (0.5,0);
\node at (-0.5,0.4) {$e_1$};
\node at (-0.5,-0.2) {$e_2$};
\node at (-0.5,-0.75) {$e_3$};
\node at (-1.7,0) {$v$};
\node at (0.7,0) {$w$};
\end{scope}
\end{tikzpicture}
\caption{The three homeomorphism types of graphs in $\calG_2$.}\label{fig:rank two graphs}
\end{figure}

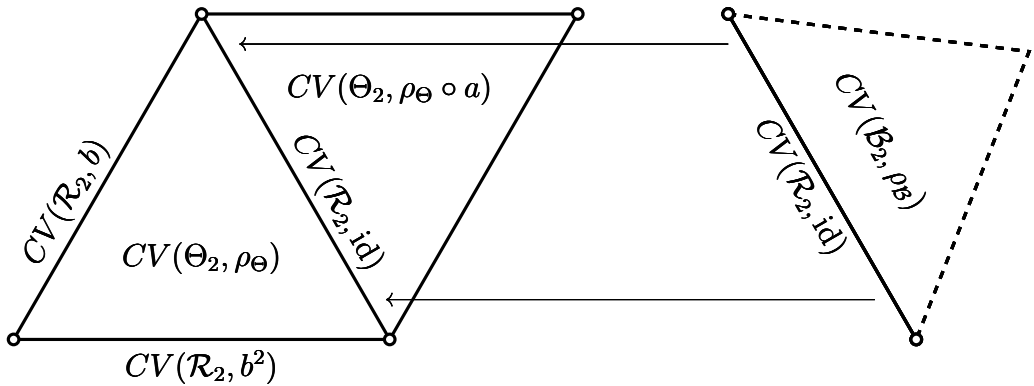
\begin{figure}[ht]
\begin{tikzpicture}
\def\s{5}
\draw[very thick] (0,0) -- (\s,0) node[pos=0](a){} node[pos=1,inner sep=0pt](b1){};
\draw[very thick,rotate=-60] (0,0) -- (\s,0) node[pos=1,inner sep=0pt](b2){};
\draw[very thick,rotate=-120] (0,0) -- (\s,0) node[pos=1,inner sep=0pt](b3){};
\draw[very thick] (b1) -- (b2) -- (b3);
\foreach \x in {a,b1,b2,b3} {
\draw[very thick,fill=white] (\x) circle [radius=2pt];
\node at (0,-3.25) {$CV(\Theta_2,\rho_\Theta)$};
\node at (2.5,-1) {$CV(\Theta_2,\rho_\Theta \circ a)$};
\node[rotate=-60] at (1.8,-2.5) {$CV(\calR_2,{\rm id})$};
\node[rotate=60] at (-1.8,-2.5) {$CV(\calR_2,b)$};
\node at (0,-4.7) {$CV(\calR_2,b^2)$};
\begin{scope}[xshift=7cm]
\draw[very thick,rotate=-60] (0,0) -- (\s,0) node[pos=1,inner sep=0pt](c2){};
\draw[very thick,dashed] (0,0) -- (4,-0.5) -- (c2);
\node[rotate=-60] at (0.95,-2.3) {$CV(\calR_2,{\rm id})$};
\node[rotate=-60] at (2,-1.7) {$CV(\calB_2,\rho_\calB)$};
\draw[very thick,fill=white] (0,0) circle [radius=2pt];
\draw[very thick,fill=white] (c2) circle [radius=2pt];
\end{scope}
\draw[->] (7,-0.4) -- (0.5,-0.4);
\draw[->] (8.95,-3.8) -- (2.45,-3.8);
}
\end{tikzpicture}
\caption{A portion of the Culler--Vogtmann outer space $CV(\FF_2)$.}\label{fig:CV2}
\end{figure}

One of the goals of this paper is to investigate a different continuous section of the quotient map.  This uses the notion of \emph{entropy}, $\entropy_G(\ell)$, defined in Section~\ref{subsec:entropy}.  Using this notion, there is a section $\calH \from \PP\calX(\FF_r) \to \calX(\FF_r)$ defined by
\[ \calH\left(\bigl[ [(G,\rho,\ell)]\bigr]\right) = [(G,\rho,\entropy_G(\ell)\ell)]. \]
We will denote the image of this section by $\calX^1(\FF_r)$.


\section{Thermodynamic Metrics}\label{sec:thermo}

In this section we introduce the \em entropy \em and \em pressure \em of a length function in $\calM(G)$, for a graph $G$ as in Section~\ref{subsec:graphs}.  By normalizing the entropy to be equal to 1, we realize $\calX^1(\FF_r)$ (as defined in Section~\ref{subsec:outer space}) as a section of $\calX(\FF_r) \to \PP\calX(\FF_r)$; it will follow that  $\calX^1(\FF_r)$ is homeomorphic to $CV(\FF_r)$ (see Theorem~\ref{thm:entropy continuous}). We use entropy and pressure to construct piecewise Riemannian metrics on $\calX^1(\FF_r)$, which we call the \emph{thermodynamic metrics}.  Pollicott--Sharp were the first to consider one of these metrics~\cite{ar:PS14}.  Kao~\cite{ar:Kao17} and Xu~\cite{ar:Xu19} have also investigated these metrics.  In these papers, the metric is only considered for a single marked graph and never on the entire outer space, as we will do here.


\subsection{Entropy}\label{subsec:entropy}
Fix a finite connected graph $G = (V,E,o,\tau,\bar{\phantom{e}})$.  An edge path $(e_1,\ldots,e_n)$ in a graph $G$ is \emph{reduced} if $e_{i} \neq \bar{e}_{i+1}$ for $i = 1,\ldots,n-1$.  A reduced edge path $(e_1,\ldots,e_n)$ is a \emph{based circuit} if $\tau(e_{n}) = o(e_{1})$ and $e_{n} \neq \bar{e}_{1}$.  The set of all based circuits in $G$ is denoted by $\calC(G)$.  For a length function $\ell \in \calM(G)$ and a real number $t \geq 0$, define:
\begin{gather*}
\calC_{G,\ell}(t) = \{\gamma \in \calC(G) \mid \ell(\gamma) \leq t\}.
\end{gather*}

\begin{definition}\label{def:entropy}
The \emph{entropy} of a length function $\ell \in \calM(G)$ is:
\begin{equation*}
\entropy_G(\ell) = \lim_{t \to \infty} \frac{1}{t} \log \abs{\calC_{G,\ell}(t)}.
\end{equation*}
\end{definition}

\begin{remark}
We defined entropy as the growth rate of the number of reduced based circuits. In the literature, there exist many equivalent definitions of entropy.  In particular, one can count the growth of reduced edge paths in $G$ starting at a particular vertex.  Also, the adjective ``based'' can be removed from the count of circuits, or one may only consider paths with length bounded between $t - L$ and $t$ for a sufficiently long $L$ (cf.~\cite[Proposition~2.3]{ar:KR08}).  
This shows that $\entropy_G(\ell)$ equals the \emph{volume entropy} of $(G,\ell)$.  The volume entropy is defined as the exponential growth rate of the volume of balls in $(\tG, g_\ell)$, where $\tG$ is the universal cover of $G$ and $g_\ell$ is the piecewise Riemannian metric obtained by pulling back the length function $\ell$.  That is,
\[ \entropy_G(\ell) = \lim_{t \to \infty} \frac{1}{t}\log \vol_{g_\ell} B(x,t) \]
where $B(x,t)$ is the ball of radius $t$ centered at $x \in \tG$, which is an arbitrary basepoint.
\end{remark}

\begin{example}\label{ex:rose}
The number of reduced edge paths in $\calR_{r}$ with $\One$--length equal to $n$ is exactly $2r(2r-1)^{n-1}$.  Thus for any vertex $v \in \widetilde{\calR}_r$ we have
\begin{equation*}
\vol_{g_\One} B(v,n) = \frac{r}{r-1}\bigl((2r-1)^{n} - 1\bigr).
\end{equation*}
Hence $\entropy_{\calR_r}(\One) = \log (2r-1)$.
\end{example}

The next lemma shows that entropy is homogeneous of degree $-1$ and thus any length function $\ell \in \calM(G)$ can be scaled to have unit entropy.  Specifically, $\entropy_G(a \cdot \ell) = 1$ if and only if $a = \entropy_G(\ell)$.   

\begin{lemma}\label{lem:homogeneous -1}
Let $G$ be a finite connected graph and fix $\ell \in \calM(G)$.  If $a \in \RR_{>0}$, then 
\[\entropy_G(a \cdot \ell) = \frac{1}{a}\entropy_G(\ell).\]    
\end{lemma}

\begin{proof}
We reparametrize the limit defining entropy by setting $s = a t$.  Then
\begin{align*}
\entropy_G(\ell) & = \lim_{t \to \infty} \frac{1}{t} \log \abs{\{\gamma \in \calC(G) \mid \ell(\gamma) \leq t\}} \\
& = \lim_{s \to \infty} \frac{a}{s} \log \abs{\{\gamma \in \calC(G) \mid a \cdot \ell(\gamma) \leq s\}} = a\entropy_G(a \cdot \ell).\qedhere
\end{align*}
\end{proof}

Entropy defines an $\Out(\FF_r)$--invariant function on $\calX(\FF_r)$ by $\entropy\left([(G,\rho,\ell)]\right) = \entropy_G(\ell)$.  This function was investigated by Kapovich--Nagnibeda, who showed: 

\begin{theorem}[{\cite[Theorem~A]{ar:KN07}}]\label{thm:entropy continuous}
The entropy function $\entropy \from \calX(\FF_r) \to \RR$ is continuous. 
\end{theorem}

In particular, the map $\calH \from \PP\calX(\FF_r) \to \calX(\FF_r)$ defined by normalizing to have unit entropy
\[ \calH\left(\bigl[ [(G,\rho,\ell)]\bigr]\right) = [(G,\rho,\entropy_G(\ell)\ell)] \] is a section.  Hence the image $\calX^1(\FF_r) = \{[(G,\rho,\ell)] \in \calX(\FF_r) \mid \entropy_G(\ell) = 1 \}$ is homeomorphic to $CV(\FF_r)$.


\subsection{Pressure}\label{subsec:pressure}  
Fix a finite connected graph $G = (V,E,o,\tau,\bar{\phantom{e}})$.  We assume throughout this section and the next that $\chi(G) < 0$ and that $G$ has no vertices with valence equal to $1$ or $2$.  

Define $A_{G} \in \Mat_{\abs{E}}(\RR)$ by
\begin{equation}
A_{G}(e,e') = \begin{cases} 1 & \mbox{ if } \tau(e) = o(e') \mbox{ and } \bar{e} \neq e', \\
0 & \mbox{ else. }
\end{cases}
\end{equation}
It follows that the entry $A_{G}^{n}(e,e')$ is the number of reduced edge paths of the form $(e_{1},\ldots, e_{n})$ where $e_1 = e$, $\tau(e_{n}) = o(e')$, and $\bar{e}_n \neq e'$.  In particular, $\trace(A_{G}^{n})$ is the number of based edge circuits with $\One$--length equal to $n$.  
Denoting the spectral radius of a matrix by $\spec(\param)$ we get, from the definition of entropy, that  
\begin{equation}
\label{equation:EntropyAsSpectralRadius}
\entropy_G(\One) = \log(\spec(A_{G}))
\end{equation}

We remark that the above assumptions on $G$ ensure that $A_G$ is irreducible.

In order to get a matrix that incorporates the metric and is related to entropy, we scale the rows of $A_G$ as follows: given a function $f \from E \to \RR$, we define $A_{G,f} \in \Mat_{\abs{E}}(\RR)$ by 
\begin{equation}
A_{G,f}(e,e') = A_{G}(e,e')\exp(-f(e)).
\end{equation}  
As for $A_{G}$, it follows that $A_{G,f}^{n}(e,e')$ is the sum of $\exp(-f(\gamma))$ over all edge paths of the form $\gamma = (e_{1},\ldots, e_{n})$ where $e_1 = e$, $\tau(e_{n}) = o(e')$ and $\bar{e}_n \neq e'$. 

\begin{definition}\label{def:pressure}
The \emph{pressure} of a function $f \from E \to \RR$ is defined as $\pressure_G(f) = \log \spec (A_{G,-f})$.
\end{definition}

By Equation~\eqref{equation:EntropyAsSpectralRadius} we have that $\pressure_G(\Zero) = \entropy_G(\One)$ as $A_{G,-\Zero} = A_G$ where $\Zero$ is the zero function.

The connection between entropy and pressure is given by the following theorem.  In a wider setting, this was shown originally by Parry--Pollicott~\cite{ar:PP90}. 
McMullen gave an alternative proof in the current setting~\cite{ar:McMullen15}.  Kapovich--Nagnibeda gave an alternative proof of the third item in the current setting~\cite{ar:KN07}.  As it is phrased below, this theorem appears in the work of Pollicott--Sharp~\cite[Lemma~3.1]{ar:PS14}.

\begin{theorem}\label{thm:pressure = 0 at entropy}
Suppose that $G = (V,E,o,\tau,\bar{\phantom{e}})$ is a finite connected graph.  Then the following hold.
\begin{enumerate}
\item\label{item:common level set} For any length function $\ell \in \calM(G)$, $\pressure_G(-\ell) = 0$ if and only if $\entropy_G(\ell) = 1$.
\item\label{item:pressure} The pressure function $\pressure_G \from \RR^{\abs{E_{+}}} \to \RR$ is real analytic and strictly convex.
\item\label{item:entropy} The entropy function $\entropy_G \from \calM(G) \to \RR$ is real analytic and strictly convex.
\end{enumerate}
\end{theorem}

Let $\calM^1(G) = \{ \ell \in \calM(G) \mid \entropy_G(\ell) = 1 \}$.  By the first item above, we have that alternatively $\calM^{1}(G) = \{ \ell \in \calM(G) \mid \pressure_G(-\ell) = 0 \}$.  To see that $\calM^1(G)$ is a codimension 1 submanifold of $\RR^{\abs{E_{+}}}$ we need to argue that $1$ is a regular value of $\entropy_G$.  This follows from the following lemma.   We denote the standard Euclidean inner product on $\RR^n$ by $\I{\param,\param}$.   

\begin{lemma}\label{lem:regular value}
Let $G$ be a finite connected graph and fix $\ell \in \calM(G)$.  Then $\I{\ell,\nabla \entropy_G(\ell)} = -\entropy_G(\ell)$.
\end{lemma}

\begin{proof}
This follows from the homogeneity of the entropy function (Lemma~\ref{lem:homogeneous -1}).  Indeed:
\begin{align*}
\I{\ell,\nabla \entropy_G(\ell)} &= \lim_{s \to 0} \frac{\entropy_G(\ell + s\ell) - \entropy_G(\ell)}{s}
= \lim_{s \to 0} \frac{\entropy_G((1+s)\ell) - \entropy_G(\ell)}{s} \\
&= \lim_{s \to 0} \frac{\frac{1}{1+s}\entropy_G(\ell) - \entropy_G(\ell)}{s} = \entropy_G(\ell) \lim_{s \to 0} \frac{-s}{s(s+1)} = -\entropy_G(\ell).\qedhere
\end{align*}
\end{proof}

We record the following properties of the partial derivatives and the gradient of the pressure function.  Given a function $f \from \RR^{\abs{E_+}} \to \RR$ and an edge $e \in E_+$, we denote the partial derivative of $f$ with respect to the $e$th coordinate by $\PD{f}{e}$.  Let $\norm{\param}_{1}$ denote the usual $L^{1}$-norm on vectors in $\RR^{n}$. 

\begin{lemma}[{\cite[Proposition~4.10]{ar:PP90}}]\label{lem:pressure properties}
Let $G  = (V,E,o,\tau,\bar{\phantom{e}})$ be a finite connected graph and fix $\ell \in \calM(G)$.  Then the following hold.
\begin{enumerate}
\item $\PD{\pressure_G}{e}(\ell) > 0$ for any $e \in E_+$, and\label{item:positive}
\item $\norm{\nabla \pressure_G(\ell)}_1 = 1$.\label{item:norm = 1}
\end{enumerate}
\end{lemma}

\subsection{Thermodynamic Metrics}\label{subsec:metrics}  Fix a finite connected graph $G = (V,E,o,\tau,\bar{\phantom{e}})$.  As in the previous section, we assume that $\chi(G) < 0$ and that $G$ has no vertices with valence equal to $1$ or $2$.   The \emph{tangent space} $T_{\ell}\calM^{1}(G)$ at the length function $\ell \in \calM^{1}(G)$ is the space of vectors $\bv \in \RR^{\abs{E_{+}}}$ such that $\I{\bv,\nabla \entropy_G(\ell)} = 0$.  The \emph{tangent bundle} $T\calM^{1}(G)$ is the subspace of $\calM^1(G) \times \RR^{\abs{E_{+}}}$ consisting of pairs $(\ell,\bv)$ where $\bv \in T_{\ell}\calM^{1}(G)$.

We now define two Riemannian metrics on $\calM^{1}(G)$.  
We denote the Hessian (i.e., the matrix of second derivatives) of a smooth function $f \from \RR^{n} \to \RR$ by $\bH[f(x)]$.

\begin{definition}\label{def:metrics}
Given a length function $\ell \in \calM^{1}(G)$ and tangent vectors $\bv_{1},\bv_{2} \in T_{\ell} \calM^{1}(G)$ we define the \emph{entropy metric} by:
\begin{align*}
\I{\bv_{1},\bv_{2}}_{\entropy,G} &=  \I{\bv_{1},\bH[\entropy_G(\ell)]\bv_{2}}, \\
\intertext{and the \emph{pressure metric} by:}
\I{\bv_{1},\bv_{2}}_{\pressure,G} &= \I{\bv_{1},\bH[\pressure_G(-\ell)]\bv_{2}}.
\end{align*}
The associated norms on the tangent bundle $T\calM^{1}(G)$ are denoted by
\begin{align*}
\norm{(\ell,\bv)}_{\entropy,G}^{2} &= \I{\bv,\bH[\entropy_G(\ell)]\bv} \quad \mbox{and} \quad 
\norm{(\ell,\bv)}_{\pressure,G}^{2} = \I{\bv,\bH[\pressure_G(-\ell)]\bv}.
\end{align*}
\end{definition}

\begin{remark}\label{rem:metrics}
Other authors have considered these metrics with different and conflicting terminology.  We discuss this now using the notation introduced above.  Pollicott--Sharp defined $\norm{\param}_{\pressure,G}$, calling it the \emph{Weil--Petersson metric}~\cite{ar:PS14}.  Kao defined $\norm{\param}_{\entropy,G}$, calling it the \emph{Weil--Petersson metric} and also studied $\norm{\param}_{\pressure,G}$, calling it the \emph{pressure metric}~\cite{ar:Kao17}.  Xu considered $\norm{\param}_{\entropy,G}$, calling it the \emph{pressure metric}~\cite{ar:Xu19}.  We use the terminology as stated in Definition~\ref{def:metrics} as it accurately reflects the functions on which the metrics are based.  The definitions of these metrics in the literature are not those as given in Definition~\ref{def:metrics}, but are equivalent as can be seen by Proposition~\ref{prop:metrics}.  

We note that Theorem 3 in the paper by Pollicott--Sharp~\cite{ar:PS14} holds for the metric $\norm{\param}_{\entropy,G}$ and not with $\norm{\param}_{\pressure,G}$ as claimed.
\end{remark}

The following proposition shows that these metrics lie in the same conformal class and shows that they can be calculated using the second derivative along a path.  This feature is essential particularly for the material in Section~\ref{sec:separating}.

\begin{proposition}\label{prop:metrics}
Let $G$ be a finite connected graph.  If $\ell_{t} \from (-1,1) \to \calM^{1}(G)$ is a smooth path, then
\begin{equation*}
\norm{(\ell_{t},\dot\ell_{t})}_{\entropy,G}^{2}  = -\I{\ddot\ell_{t},\nabla \entropy_G(\ell_{t})}
\quad \mbox{and} \quad
\norm{(\ell_{t},\dot\ell_{t})}_{\pressure,G}^{2}  = \I{\ddot\ell_{t},\nabla \pressure_G(-\ell_{t})}.
\end{equation*}
Additionally, given a length function $\ell \in \calM^{1}(G)$ and tangent vectors $\bv_{1},\bv_{2} \in T_{\ell} \calM^{1}(G)$ we have
\begin{equation*} 
\I{\bv_{1},\bv_{2}}_{\entropy,G} = \frac{\I{\bv_{1},\bv_{2}}_{\pressure,G}}{\I{\ell,\nabla \pressure_G(-\ell)}}. 
\end{equation*}
\end{proposition}

\begin{proof}
Differentiating the equation $\pressure_G(-\ell_{t}) = 0$ with respect to $t$ we have $\I{\dot\ell_{t},\nabla\pressure_G(-\ell_{t})} = 0$.  Differentiating again we find
\[ \I{\ddot \ell_{t}, \nabla\pressure_G(-\ell_{t})} - \I{\dot \ell_{t},\bH [\pressure_G(-\ell_{t})]\dot\ell_{t}} = 0. \]
Hence $\norm{(\ell_{t},\dot\ell_{t})}_{\pressure}^{2}  = \I{\dot\ell_{t},\bH [\pressure(-\ell_{t})]\dot\ell_{t}} = \I{\ddot\ell_{t},\nabla \pressure(-\ell_{t})}$ as claimed.

The proof of the analogous statement for the entropy norm is similar observing that $\entropy_G(\ell_t) = 1$.

By Lemma~\ref{lem:regular value}, we have that $\I{\ell,\nabla \entropy_G(\ell)} = -1$ for any $\ell \in \calM^{1}(G)$.  Further, we have that $\nabla \pressure_G(-\ell)$ is nonzero by Lemma~\ref{lem:pressure properties} and hence is parallel to $\nabla \entropy_G(\ell)$ by Theorem~\ref{thm:pressure = 0 at entropy}\eqref{item:common level set}.  Hence we find:
\begin{align*}
\norm{(\ell_{t},\dot\ell_{t})}_{\entropy,G}^{2} &= -\I{\ddot\ell_{t},\nabla\entropy_G(\ell_{t})} = \frac{\I{\ddot\ell_{t},\nabla\entropy_G(\ell_{t})}}{\I{\ell_{t},\nabla \entropy_G(\ell_{t})}} = \frac{\I{\ddot\ell_{t},\nabla\pressure_G(-\ell_{t})}}{\I{\ell_{t},\nabla \pressure_G(-\ell_{t})}} = \frac{\norm{(\ell_{t},\dot\ell_{t})}_{\pressure,G}^{2}}{\I{\ell_{t},\nabla \pressure_G(-\ell_{t})}}.
\end{align*}
By polarization, the norm determines the inner product and so the claim follows.
\end{proof}

Positive definiteness of the Hessians follow from strict convexity of $\entropy_G$ and $\pressure_G$ on $\calM(G)$ and $\RR^{\abs{E_{+}}}$ respectively (Theorem~\ref{thm:pressure = 0 at entropy}).

Using these norms, we can define the entropy or pressure length of a piecewise smooth path $\ell_{t} \from [t_0,t_1] \to \calM^{1}(G)$ by
\begin{align*}
\calL_{\entropy,G}(\ell_{t}|[t_0,t_1]) & = \int_{t_0}^{t_1} \norm{(\ell_{t},\dot\ell_{t})}_{\entropy,G} \, dt \quad \mbox{and} \quad 
\calL_{\pressure,G}(\ell_{t}|[t_0,t_1])  = \int_{t_0}^{t_1} \norm{(\ell_{t},\dot\ell_{t})}_{\pressure,G} \, dt.
\end{align*}
These induce the entropy and pressure distance functions on $\calM^1(G)$ by
\begin{align*}
d_{\entropy,G}(x,y) & = \inf\{ \calL_{\entropy,G}(\ell_{t}|[0,1]) \mid \ell_{t} \from [0,1] \to \calM^{1}(G), \, \ell_{0} = x, \, \ell_{1} = y  \} \mbox{ and}\\
d_{\pressure,G}(x,y) & = \inf\{ \calL_{\pressure,G}(\ell_{t}|[0,1]) \mid \ell_{t} \from [0,1] \to \calM^{1}(G), \, \ell_{0} = x, \, \ell_{1} = y  \}.
\end{align*}

Given a marked graph $(G,\rho)$, we set $\calX^1(G,\rho) = \calX(G,\rho) \cap \calX^1(\FF_r)$.  Using the natural bijection $\calX^1(G,\rho) \leftrightarrow \calM^1(G)$, we get metrics and distance functions on $\calX^1(G,\rho)$ that we denote using the same notation as above.

Next we explain how these fit together to get distance functions on $\calX^1(\FF_r)$.  Suppose $\alpha_t \from [0,1] \to \calX^1(\FF_r)$ is a piecewise smooth path and that there is a partition $t_1 = 0 < t_2 < \cdots < t_{n+1} = 1$ and marked graphs $(G_k,\rho_k)$ for $k = 1,\ldots,n$ such that $\alpha_t \in \calX^1(G_k,\rho_k)$ for $t \in (t_k,t_{k+1})$.  
We set
\begin{align*}
\calL_\entropy(\alpha_t) & = \sum_{k=1}^n \calL_{\entropy,G_k}(\alpha_t | (t_k,t_{k+1})) \quad \mbox{and} \quad
\calL_\pressure(\alpha_t)  = \sum_{k=1}^n \calL_{\pressure,G_k}((\alpha_t | (t_k,t_{k+1})).
\end{align*}
These define distance functions on $\calX^1(\FF_r)$ as usual---by taking the infimum of the lengths of paths---that we denote by $d_\entropy$ and $d_\pressure$.  
It is obvious that the proposed distance functions are symmetric and satisfy the triangle inequality; positive definiteness of the Hessians implies non-degeneracy.  However, 
it is not obvious that the distances they define are finite: {\it a priori} it may be possible that the length of a path that collapses an edge is infinite.  This will be addressed in Section~\ref{sec:entropy topology}.

\begin{remark}\label{rem:pressure}
For the remainder of the article, we will mainly be concerned with the entropy metric.  As stated in Section~\ref{subsec:intro pressure} in the Introduction, the main results of this article also hold for the pressure metric with slightly altered hypotheses. 
\end{remark}


\section{A Determinant Defining Equation for \texorpdfstring{$\calM^{1}(G)$}{M\textasciicircum1(G)}}
\label{sec:determinant}

The purpose of this section is to derive formulas to assist in computing the metrics introduced in the previous section.  The first of these appears in Section~\ref{subsec:coefficient}, specifically Proposition~\ref{prop:metric formulas}, where it is shown that these metrics can be computed using finite sums of exponential functions on $\calM^1(G)$.  Next, in Section~\ref{subsec:simplification}, we present a simplification for certain graphs that is useful in the sequel.


\subsection{Determinant equation}\label{subsec:coefficient}
Fix a finite connected graph $G = (V,E,o,\tau,\bar{\phantom{e}})$.  We always assume 
that $\chi(G) < 0$ and $G$ has no vertices of valence $1$ or $2$.

Let $D_{G}$ be the directed graph with adjacency matrix $A_{G}$.  Thus the vertex set for $D_G$ is the set $E$ (recall our notation $E = E_+ \cup \overline{E}_+$) and there is an edge from $e$ to $e'$ if $A_G(e,e') = 1$, that is, if
$\tau(e) = o(e')$ and $\bar{e} \neq e'$.  The \emph{cycle complex of $D_{G}$}, denoted by $C_{G}$, is the abstract simplicial complex with an $n$--simplex for each collection $\Delta = \{ z_{1},\ldots,z_{n+1}\}$ of pairwise disjoint simple cycles, i.e., embedded loops, in $D_{G}$. 
  
\begin{example}\label{ex:barbell}
Suppose that $G$ is the barbell graph as shown below:

\begin{figure}[ht]
\centering
\begin{tikzpicture}
\tikzstyle{vertex} =[circle,draw,fill=black,thick, inner sep=0pt,minimum size= 1 mm]
\node[vertex] (v) at (-1,0) {};
\node[vertex] (w) at (1,0) {};
\draw[thick,black,->-] (v) -- (w) node[pos=0.5,label=above:{$c$}] {};
\draw[thick,->-] (v) arc (0:360:8mm) node[pos=0.5,label=left:{$a$}] {};
\draw[thick,->-] (w) arc (180:-180:8mm) node[pos=0.5,label=right:{$b$}] {};
\end{tikzpicture}%
\end{figure}

\noindent Order the edges of $G$ by $a$, $\bar{a}$, $b$, $\bar{b}$, $c$, $\bar{c}$.  The matrix $A_{G}$ and directed graph $D_{G}$ are as presented below.

\medskip
\begin{minipage}{0.5\textwidth}
\begin{equation*}
\begin{bmatrix}
1 & 0 & 0 & 0 & 1 & 0 \\
0 & 1 & 0 & 0 & 1 & 0 \\
0 & 0 & 1 & 0 & 0 & 1 \\
0 & 0 & 0 & 1 & 0 & 1 \\
0 & 0 & 1 & 1 & 0 & 0 \\
1 & 1 & 0 & 0 & 0 & 0
\end{bmatrix}
\end{equation*}
\end{minipage}
\begin{minipage}{0.5\textwidth}
\begin{tikzpicture}
\tikzstyle{vertex} =[circle,draw,fill=black,thick, inner sep=0pt,minimum size= 1 mm]
\node[vertex,label=left:{$a$}] (a) at (-1.5,1) {};
\node[vertex,label=left:{$\bar{a}$}] (A) at (-1.5,-1) {};
\node[vertex,label=above:{$c$}] (c) at (0,1) {};
\node[vertex,label=below:{$\bar{c}$}] (C) at (0,-1) {};
\node[vertex,label=right:{$b$}] (b) at (1.5,1) {};
\node[vertex,label=right:{$\bar{b}$}] (B) at (1.5,-1) {};
\draw[thick,black,->-] (a) -- (c);
\draw[thick,black,->-=0.7] (A) -- (c);
\draw[thick,black,->-=0.7] (b) -- (C);
\draw[thick,black,->-] (B) -- (C);
\draw[thick,black,->-] (c) -- (b);
\draw[thick,black,->-=0.7] (c) -- (B);
\draw[thick,black,->-=0.7] (C) -- (a);
\draw[thick,black,->-] (C) -- (A);
\draw[thick,->-] (a) arc (0:360:6mm);
\draw[thick,->-] (b) arc (180:-180:6mm);
\draw[thick,->-] (A) arc (0:360:6mm);
\draw[thick,->-] (B) arc (180:-180:6mm);
\end{tikzpicture}
\end{minipage}

\medskip \noindent  There are eight simple cycles in $D_G$: $\gamma_{a} = (a)$, $\gamma_{\bar{a}} = (\bar{a})$, $\gamma_{b} = (b)$, $\gamma_{\bar{b}} = (\bar{b})$, $\gamma_{ab} = (a,c,b,\bar{c})$, $\gamma_{a\bar{b}} = (a,c,\bar{b},\bar{c})$, $\gamma_{\bar{a}b} = (\bar{a},c,b,\bar{c})$, and $\gamma_{\bar{a}\bar{b}} = (\bar{a},c,\bar{b},\bar{c})$.  The cycle complex $C_{G}$ is the flag complex whose 1--skeleton is shown below:
\begin{figure}[ht]
\centering
\begin{tikzpicture}
\tikzstyle{vertex} =[circle,draw,fill=black,thick, inner sep=0pt,minimum size= 1 mm]
\node[vertex,label=left:{$\gamma_{b}$}] (b) at (-1.5,0) {};
\node[vertex,label=right:{$\gamma_{\bar{b}}$}] (B) at (1.5,0) {};
\node[vertex,label=above:{$\gamma_{a}$}] (a) at (0,1.5) {};
\node[vertex,label=below:{$\gamma_{\bar{a}}$}] (A) at (0,-1.5) {};
\node[vertex,label=left:{$\gamma_{\bar{a}\bar{b}}$}] (AB) at (-2.25,2.25) {};
\node[vertex,label=left:{$\gamma_{a\bar{b}}$}] (aB) at (-2.25,-2.25) {};
\node[vertex,label=right:{$\gamma_{ab}$}] (ab) at (2.25,-2.25) {};
\node[vertex,label=right:{$\gamma_{\bar{a}b}$}] (Ab) at (2.25,2.25) {};
\draw[thick,black] (b) -- (B);
\draw[line width=6pt,white]  (0,0.2) -- (0,-0.2);
\draw[thick,black] (a) -- (A);
\draw[thick,black] (a) -- (b);
\draw[thick,black] (b) -- (A);
\draw[thick,black] (A) -- (B);
\draw[thick,black] (B) -- (a);
\draw[thick,black] (a) -- (AB);
\draw[thick,black] (b) -- (AB);
\draw[thick,black] (b) -- (aB);
\draw[thick,black] (A) -- (aB);
\draw[thick,black] (B) -- (Ab);
\draw[thick,black] (a) -- (Ab);
\draw[thick,black] (B) -- (ab);
\draw[thick,black] (A) -- (ab);
\end{tikzpicture}
\end{figure}
\end{example}

Given a function $f \from E \to \RR$ (in particular a length function) and a simple cycle $z$ in $D_G$, we set $f(z) = \sum_{i=1}^n f(e_i)$ where $e_1,\ldots,e_n$ are the vertices in $D_G$ traversed by $z$ (each corresponding to an oriented edge in $G$).  Likewise, for a simplex $\Delta = \{ z_{1},\ldots,z_{n}\}$ in $C_G$ we set $f(\Delta) = \sum_{k=1}^{n} f(z_{k})$.  We consider the empty set as a simplex and define $f(\emptyset) = 0$ for any function $f \from E \to \RR$.  Lastly, for a simplex $\Delta = \{z_1,\ldots,z_n\}$ we set $\splx{\Delta} = n$. 
 
 Recall that given a length function $\ell \in \calM(G)$, the matrix $A_{G,\ell}$ is defined by $A_{G,\ell}(e,e') = \exp(-\ell(e))A_G(e,e')$.  We consider the function $F_{G} \from \calM(G) \to \RR$ given by
\begin{equation}\label{eq:F}
F_{G}(\ell) = \det\left(I - A_{G,\ell}\right).
\end{equation}
This function can be expressed using the cycle complex $C_G$ as follows.

\begin{theorem}\label{thm:F}
Let $G$ be a finite connected graph and fix $\ell \in \calM(G)$.  Then
\begin{equation*}\label{eq:F thm}
F_{G}(\ell) = \sum_{\Delta \in C_{G}} (-1)^{\splx{\Delta}} \exp(-\ell(\Delta)).
\end{equation*}
\end{theorem}

\begin{proof}
This follows from the Coefficient Theorem for Digraphs.  See for instance~\cite{ar:CR90} and~\cite[Theorems~2.5~and~2.14]{ar:A-KHR15}. 
\end{proof}

\begin{example}\label{example:barbell poly}
We apply Theorem~\ref{thm:F} to the case when $G$ is the barbell graph as in Example~\ref{ex:barbell}.  Using the change of variables $x = \exp(-\ell(a))$, $y = \exp(-\ell(b))$ and $z = \exp(-\ell(c))$ we find:
\begin{align*}
F_{G}(\ell) =& \, 1 - \bigl(2x + 2y + 4xyz^{2}\bigr) + \bigl(x^{2} + y^{2} + 4xy + 4x^{2}yz^{2} + 4xy^{2}z^{2} \bigr) \\
& -\bigl(2x^{2}y + 2xy^{2} + 4x^{2}y^{2}z^{2}\bigr) + x^{2}y^{2}.
\end{align*}
\end{example}

The following statements show how the function $F_{G}$ is related to $\calM^{1}(G)$.

\begin{lemma}\label{lem:F}
For $\ell \in \calM^{1}(G)$ we have
\begin{enumerate}
\item\label{item:F = 0} $F_{G}(\ell) = 0$;  
\item\label{item:grad F} $\nabla F_{G}(\ell) \neq 0$; and moreover
\item\label{item:positive partial} $\PD{F_G}{e}(\ell) > 0$ for any $e \in E_+$.
\end{enumerate}
\end{lemma}

\begin{proof}
Since $\ell$ lies in $\calM^{1}(G)$, Theorem~\ref{thm:pressure = 0 at entropy}\eqref{item:common level set} and the definition of pressure imply that $\spec(A_{G,\ell}) = 1$.  Above we remarked that the assumptions on $G$ imply that $A_G$ is irreducible; hence so is $A_{G,\ell}$. By the Perron--Frobenius Theorem the spectral radius of $A_{G,\ell}$ 
is realized by a positive, real, simple eigenvalue;~\eqref{item:F = 0} follows.

Now consider the function $p \from \RR \to \RR$ defined by $p(t) = \det(I - tA_{G,\ell})$.  Let $1 = \lambda_{1}, \ldots, \lambda_{\abs{E}}$ 
be the roots of the characteristic polynomial of $A_{G,\ell}$.  Then we can write:
\[ p(t) = (1-t)\prod_{i=2}^{\abs{E}} (1-t\lambda_{i}). \]  
Therefore taking the derivative we find:
\[ p'(1) = -\prod_{i=2}^{\abs{E}} (1-\lambda_{i}).\] 
For $i = 2,\ldots, \abs{E}$ we have that $|\lambda_i| \leq 1$ and $\lambda_i \neq 1$.  Combining these observations with the fact that complex eigenvalues come in conjugate pairs, it follows that $p'(1) < 0$.  Observe that $F_{G}(\ell + s \cdot \One) = \det(I - \exp(-s)A_{G,\ell}) = p(\exp(-s))$.  Therefore we have that $\I{\One,\nabla F_{G}(\ell)} = -p'(1)\exp(0) > 0$,
giving~\eqref{item:grad F}.

Now by \eqref{item:F = 0}, \eqref{item:grad F} and Theorem~\ref{thm:pressure = 0 at entropy}, we have that $\nabla F_G(\ell)$ is parallel to $\nabla \pressure_G(\ell)$ for $\ell \in \calM^1(G)$.  Hence by Lemma~\ref{lem:pressure properties}\eqref{item:positive}, we have that either $\PD{F_G}{e}(\ell) > 0$ or $\PD{F_G}{e}(\ell) < 0$ for all $e \in E_+$ and $\ell \in \calM^1(G)$.  Since $\I{\One,\nabla F_G(\ell)} >0$, we must have the former, whence~\eqref{item:positive partial}.
\end{proof}

As a consequence of Lemma~\ref{lem:F}, we have the following corollary.

\begin{corollary}\label{co:F}
The unit entropy moduli space $\calM^{1}(G)$ is a connected component of the level set $\{ \ell \in \calM(G) \mid F_{G}(\ell) = 0 \}$.
\end{corollary}

Using these observations, we can compute the entropy and pressure norms using the function $F_{G}$.   

\begin{proposition}\label{prop:metric formulas}
If $\ell_{t} \from (-1,1) \to \calM^{1}(G)$ is a smooth path, then:\begin{align*} 
\norm{(\ell_{t},\dot\ell_{t})}_{\entropy,G}^{2} &= \frac{-\I{\dot\ell_{t},\bH [F_{G}(\ell_{t})]\dot\ell_{t}}}{\I{\ell_{t},\nabla F_{G}(\ell_{t})}}= \frac{\I{\ddot\ell_{t},\nabla F_{G}(\ell_{t})}}{\I{\ell_{t},\nabla F_{G}(\ell_{t})}}, \mbox{ and} \\
\norm{(\ell_{t},\dot\ell_{t})}_{\pressure,G}^{2} &= \frac{-\I{\dot\ell_{t},\bH [F_{G}(\ell_{t})]\dot\ell_{t}}}{\norm{\nabla F_{G}(\ell_{t})}_{1}}= \frac{\I{\ddot\ell_{t},\nabla F_{G}(\ell_{t})}}{\norm{\nabla F_{G}(\ell_{t})}_{1}}. 
\end{align*}
\end{proposition}

\begin{proof}
The proof of the formula for the entropy norm is similar to that of Proposition~\ref{prop:metrics} and left to the reader.

The proof of the formula for the pressure norm is again similar noting that $\norm{\nabla \pressure(\ell)}_{1} = 1$ as stated in Lemma~\ref{lem:pressure properties}\eqref{item:norm = 1}.
\end{proof}

Using Theorem~\ref{thm:F}, we can compute the partial derivatives of $F_G$.  We find for any edges $e,e' \in E_+$ that 
\begin{align}
\PD{F_{G}}{e}(\ell) &= -\sum_{\Delta \in C_G} (-1)^{\splx{\Delta}}\Delta(e)\exp(-\ell(\Delta)), \mbox{ and}\label{eq:partial-e} \\[5pt]
\PPD{F_{G}}{e}{e'}(\ell) &= \sum_{\Delta \in C_G} (-1)^{\splx{\Delta}}\Delta(e)\Delta(e')\exp(-\ell(\Delta)).\label{eq:partial-ee'}
\end{align}  
where $\Delta(e) \in \{0,1,2\}$ denotes the cardinality of the intersection $\{e,\bar{e}\} \cap \Delta$.  Using this notation, we remark that $\ell(\Delta) = \sum_{e \in E_+} \Delta(e)\ell(e)$ for a length function $\ell \in \calM(G)$ and simplex $\Delta \in C_G$.  Given a vector $\bv \in \RR^{\abs{E_+}}$ and a simplex $\Delta \in C_G$, we set $\bv(\Delta) = \sum_{e \in E_+} \Delta(e)\bv(e)$.  Using these expressions, we can rewrite the dot products appearing in the formulas for the metrics in Proposition~\ref{prop:metric formulas} as sums over simplicies in $C_{G}$ rather than over the edges of $G$ as follows:

\begin{lemma}\label{lem:dot product}
For  $\ell \in \calM(G)$ we have
\begin{equation*}
\I{\ell, \nabla F_{G}(\ell)} = -\sum_{\Delta \in C_{G}} (-1)^{\splx{\Delta}}\ell(\Delta)\exp(-\ell(\Delta)).
\end{equation*}
\end{lemma}

\begin{proof}
We compute:
\begin{align*}
\I{\ell, \nabla F_{G}(\ell)} & = \sum_{e \in E_{+}} \ell(e) \left(-\sum_{\Delta \in {C_{G}}}(-1)^{\splx{\Delta}}\Delta(e)\exp(-\ell(\Delta))\right) \\
& = -\sum_{\Delta \in {C_{G}}} (-1)^{\splx{\Delta}}\exp(-\ell(\Delta)) \left(\sum_{e \in E_{+}}  \ell(e)\Delta(e)\right) \\
& = -\sum_{\Delta \in C_{G}} (-1)^{\splx{\Delta}}\ell(\Delta)\exp(-\ell(\Delta)).\qedhere
\end{align*}
\end{proof}

\begin{lemma}\label{lem:hessian product}
Let $G$ be a finite connected graph.  If $(\ell, \bv) \in T\calM^{1}(G)$, then
\begin{equation*}
\I{\bv, \bH[F_{G}(\ell)]\bv} = \sum_{\Delta \in C_{G}} (-1)^{\splx{\Delta}}\bv(\Delta)^{2}\exp(-\ell(\Delta)).
\end{equation*}
\end{lemma}

\begin{proof}
This is similar to Lemma~\ref{lem:dot product}.  The $e$th component of $\bH[F_{G}(\ell)]\bv$ is
\begin{multline*}
\sum_{e' \in E_{+}} \left( \sum_{\Delta \in C_{G}} (-1)^{\splx{\Delta}} \Delta(e)\Delta(e')\exp(-\ell(\Delta))\right) \bv(e') = 
\sum_{\Delta \in C_{G}} (-1)^{\splx{\Delta}} \Delta(e)\bv(\Delta)\exp(-\ell(\Delta)).
\end{multline*}
Hence:
\begin{align*}
\I{\bv, \bH[\nabla F_{G}(\ell)]\bv} & = \sum_{e \in E_{+}} \bv(e) \left(\sum_{\Delta \in {C_{G}}}(-1)^{\splx{\Delta}}\Delta(e)\bv(\Delta)\exp(-\ell(\Delta))\right) \\
 &= \sum_{\Delta \in C_{G}} (-1)^{\splx{\Delta}}\bv(\Delta)^{2}\exp(-\ell(\Delta)).\qedhere
\end{align*}
\end{proof}

By Lemma~\ref{lem:F}, if $\entropy_G(\ell) = 1$ then $F_{G}(\ell) = 0$.  A partial converse is:

\begin{lemma}\label{lem: h < 1}
If $\entropy_G(\ell) < 1$, then $F_G(\ell) > 0$.
\end{lemma}

\begin{proof}
We begin by showing that if $\entropy_G(\ell) < 1$, then $F_G(\ell) \neq 0$.  To begin, we observe that if $\entropy_G(\ell) < 1$, then $\pressure_G(-\ell) < 0$.  Indeed, let $\ell_t \from [0,1] \to \calM(G)$ be the path defined by
\[ \ell_t = (1-t)\ell + t\entropy_G(\ell)\ell. \]
We have that $\I{\nabla \pressure_G(-\ell_t), - \dot \ell_t} > 0$ as each component of $\nabla \pressure_G(-\ell_t)$ is positive by Lemma~\ref{lem:pressure properties}\eqref{item:positive} and each component of $-\dot\ell_t$ is positive by construction. Notice that $\pressure_G(-\ell_1) = 0$ by Theorem~\ref{thm:pressure = 0 at entropy}\eqref{item:common level set} since $\entropy_G(\ell_1) = 1$.  Therefore, we find that
\[ -\pressure_G(-\ell_0) = \int_0^1 \I{\nabla \pressure_G(-\ell_t),-\dot \ell_t} \, dt > 0 \]
and hence $\pressure_G(-\ell) = \pressure_G(-\ell_0)$ is negative as claimed.  Therefore we have that $\spec (A_{G,\ell}) < 1$ and in particular, $1$ is not an eigenvalue of $A_{G,\ell}$.  Hence $F_G(\ell) \neq 0$.  This completes the claim that $F_G(\ell) \neq 0$ for any $\ell \in \calM(G)$ where $\entropy_G(\ell) < 1$.  

We now complete the proof of the lemma.  Suppose that $\entropy_G(\ell) < 1$ and consider the continuous function $p \from [0,\infty) \to \RR$ defined by $p(t) = F_G(\ell + t \cdot \One)$.  As $\entropy_G(\ell + t \cdot \One) < 1$ for all $t \in [0,\infty)$, by the above claim we have that $p(t) \neq 0$.  Since $p(t) \to 1$ as $t \to \infty$, we have that $p(t) > 0$ for all $t \in [0,\infty)$.  In particular, we have that $F_G(\ell) = p(0) > 0$.      
\end{proof}


\subsection{A Simplification}\label{subsec:simplification}
The function $F_{G}$ factors nontrivially in special cases as a result of certain aspects of the graph $G$.  In such a case, we can replace $F_G$ with one of these factors and simplify the expressions for the entropy and pressure norm from Proposition~\ref{prop:metric formulas}.  

For instance, one factorization occurs if $e$ is a loop edge.  When $\ell(e) = 0$ the vector $\bv \in \RR^{\abs{E}}$, where $\bv(e) = 1$, $\bv(\bar{e}) = -1$, and the rest of the entries equal to 0, is an eigenvector of $A_{G,\ell}$ with eigenvalue 1.  This means $1-\exp(-\ell(e))$ is a factor of $F_{G}$.

\begin{example}\label{ex:barbell factor}
Using the notation from Example~\ref{example:barbell poly}, we have that both $1-x$ and $1-y$ are factors of $F_{G}$.  Factoring, we have:
\begin{equation*}
F_{G}(\ell) = (1-x)(1-y)(1 - x - y + xy - 4xyz^{2}).
\end{equation*}
\end{example}

Another case where there is a factorization of $F_{G}$ is when the edge involution $e \leftrightarrow \bar{e}$ is a graph automorphism of $G$. 
There are only two types of graphs for which such an automorphism exists.
\begin{enumerate}
\item the $r$--rose, $\calR_{r}$; and 
\item the graph $\Theta_{r}$ with vertices, $v$ and $w$, and edges $e_1, \ldots,e_{r+1}$ where $o(e_i) = v$ and $\tau(e_i) = w$ for $i = 1,\ldots,r+1$.
\end{enumerate}
In this case ordering the edges in $E_+$ first and then ordering the edges in $E - E_+$ accordingly, we have that
\[ A_G = \begin{bmatrix}
B_G & B'_G \\ B'_G & B_G
\end{bmatrix} \] 
for two matrices $B_G,B'_G \in \Mat_{\abs{E_+}}(\RR)$.  Thus
\begin{align*}
F_G(\ell) &= \det(I - A_{G,\ell}) = \det \begin{bmatrix} I - B_{G,\ell} & -B'_{G,\ell} \\ -B'_{G,\ell} & I - B_{G,\ell} \end{bmatrix} \\
& = \det(I - B_{G,\ell} - B'_{G,\ell})\det(I - B_{G,\ell} + B'_{G,\ell}).
\end{align*}
Since each row of $B_G$ or $B'_G$ corresponds to an edge $e \in E_+$ the notation $B_{G,\ell}$ and $B'_{G,\ell}$ still makes sense.

For $\calR_r$, we have that $B_{\calR_r}$ is the $r \times r$ matrix consisting of all $1$'s and $B'_{\calR_r} = B_{\calR_r} - I$.  In this case
\[ \det(I - B_{\calR_r,\ell} + B'_{\calR_r,\ell}) = \prod_{e \in E_+}(1-\exp(-\ell(e))). \]  
These are precisely the factors which were observed above for loop edges.

For $\Theta_{r}$, we have that $B_{\Theta_r}$ is the $(r+1) \times (r+1)$ matrix consisting of all $0$'s and $B'_{\Theta_r}$ is the $(r + 1) \times (r+1)$ matrix where all diagonal entries are $0$ and all non-diagonal entries are $1$.  In this case we have
\begin{align*}
F_{\Theta_{r}}(\ell) & = \det(I - B'_{\Theta_r,\ell})\det(I + B'_{\Theta_r,\ell}).
\end{align*}

In general, we now construct a graph quotient $D_{G} \to \oD_{G}$ that identifies certain edge pairs $\{e,\bar{e}\}$ resulting in a new matrix $\oA_{G}$, which selects the appropriate factor.  In this new matrix, every row corresponds to either an edge $e \in E$ or an edge pair $\{ e,\bar{e}\}$ and so we can still make sense of $\oA_{G,f}$ for a function $f \from E_+ \to \RR$.  

When $G$ is $\calR_r$ or $\Theta_{r}$, we take $\oD_{\Theta_{r}}$ to be the quotient of $D_{\Theta_{r}}$ by the orientation reversing automorphism $e \mapsto \bar{e}$.  In this case $\oA_{G,\ell} = B_{G,\ell} + B'_{G,\ell}$.

Else, for each pair $\{e,\bar{e}\}$ that is a loop edge of $G$, we identify the vertices of $D_{G}$ corresponding to $e$ and $\bar{e}$, now denoted $e\bar{e}$, keep the incoming edges and identify the outgoing edges that have the same terminal vertex.  We call the resulting graph $\oD_{G}$.  Algebraically, we add together the columns corresponding to $e$ and $\bar{e}$ and delete one of the rows corresponding to $e$ and $\bar{e}$.

We define $\oF_{G} \from \calM(G) \to \RR$ by
\begin{equation}\label{eq:oF}
\oF_{G}(\ell) = \det\bigl(I - \oA_{G,\ell}\bigr).
\end{equation}
The formula in Theorem~\ref{thm:F}, the formulas for the partial derivatives in~\eqref{eq:partial-e} and \eqref{eq:partial-ee'}, and the inner products in Lemmas~\ref{lem:dot product} and \ref{lem:hessian product} hold are true for $\oF_{G}$ using the complex $\oC_G$, which is the cycle complex of the directed graph $\oD_G$.

\begin{example}\label{ex:barbell quotient}
For $G$ equal to the barbell graph as in Example~\ref{ex:barbell} we have $\oA_{G}$ and $\oD_{G}$ as shown below (columns of the matrix are ordered as $a\bar{a}$, $b\bar{b}$, $c$, $\bar{c}$):

\medskip
\begin{minipage}{0.4\textwidth}
\begin{equation*}
\begin{bmatrix}
1 & 0 & 1 & 0 \\
0 & 1 & 0 & 1 \\
0 & 2 & 0 & 0 \\
2 & 0 & 0 & 0 
\end{bmatrix}
\end{equation*}
\end{minipage}
\begin{minipage}{0.6\textwidth}
\begin{tikzpicture}
\tikzstyle{vertex} =[circle,draw,fill=black,thick, inner sep=0pt,minimum size= 1 mm]
\node[vertex,label=left:{$a\bar{a}$}] (a) at (-1.5,0) {};
\node[vertex,label=above:{$c$}] (c) at (0,1) {};
\node[vertex,label=below:{$\bar{c}$}] (C) at (0,-1) {};
\node[vertex,label=right:{$b\bar{b}$}] (b) at (1.5,0) {};
\draw[thick,black,->-] (a) -- (c);
\draw[thick,black,->-] (b) -- (C);
\draw[thick,black,->-] (c) to[in=180,out=-80] (b);
\draw[thick,black,->-] (c) to[in=100,out=0] (b);
\draw[thick,black,->-] (C) to[in=0,out=100] (a);
\draw[thick,black,->-] (C) to[in=-80,out=180] (a);
\draw[thick,->-] (a) arc (0:360:8mm);
\draw[thick,->-] (b) arc (180:-180:8mm);
\end{tikzpicture}
\end{minipage}

\medskip
The two directed edges from $\bar{c}$ to $a\bar{a}$ are identified with the set $\{a ,\bar{a}\}$ so that we think of the sequence $c,a$ 
or $c,\bar{a}$ as specifying which of the two edges between $\bar{c}$ and $a\bar{a}$ to traverse in a cycle.  There are six simple cycles: $\gamma_{a\bar{a}} = (a\bar{a})$, $\gamma_{b\bar{b}} = (b\bar{b})$, $\gamma_{ab} = (a,c,b,\bar{c})$, $\gamma_{a\bar{b}} = (a,c,\bar{b},\bar{c})$, $\gamma_{\bar{a}b} = (\bar{a},c,b,\bar{c})$, and $\gamma_{\bar{a}\bar{b}} = (\bar{a},c,\bar{b},\bar{c})$.  The cycle complex $\oC_{G}$ is shown below:

\begin{figure}[h]
\centering
\begin{tikzpicture}
\tikzstyle{vertex} =[circle,draw,fill=black,thick, inner sep=0pt,minimum size= 1 mm]
\node[vertex,label=above:{$\gamma_{a\bar{a}}$}] (a) at (0,1.5) {};
\node[vertex,label=below:{$\gamma_{b\bar{b}}$}] (b) at (0,-1.5) {};
\node[vertex,label=left:{$\gamma_{\bar{a}\bar{b}}$}] (AB) at (6,-1.5) {};
\node[vertex,label=left:{$\gamma_{a\bar{b}}$}] (aB) at (3,-1.5) {};
\node[vertex,label=left:{$\gamma_{ab}$}] (ab) at (3,1.5) {};
\node[vertex,label=left:{$\gamma_{\bar{a}b}$}] (Ab) at (6,1.5) {};
\draw[thick,black] (a) -- (b);
\end{tikzpicture}
\end{figure}

Using Theorem~\ref{thm:F}, we find (with $x = \exp(-\ell(a))$, $y = \exp(-\ell(b))$ and $z = \exp(-\ell(c))$):
\begin{align*}
\oF_{G}(\ell) =& \, \sum_{\Delta \in \oC_{G}} (-1)^{\splx{\Delta}} \exp(-\ell(\Delta)) 
=  1 - \bigl(x + y + 4xyz^{2}\bigr) + xy.
\end{align*}
The astute reader will notice the comparison with Example~\ref{ex:barbell factor}.
\end{example}

\begin{lemma}\label{lem:same spectral}
With the above set-up: $\spec(\oA_{G,\ell}) = \spec(A_{G,\ell})$.
\end{lemma}

\begin{proof}
Each circuit in $\oD_{G}$ lifts to at most two circuits of the same length in $D_{G}$.  Thus $\trace(\oA_{G,\ell}^{n}) \leq \trace(A_{G,\ell}^{n}) \leq 2\trace(\oA_{G,\ell}^{n})$  for all $n \in \NN$ and so the lemma follows.
\end{proof}

In particular, we have that $\pressure_G(-\ell) = \log \spec(\oA_{G,\ell})$.  As in Corollary~\ref{co:F}, we have the following statement.   This follows for the same reasons as in Lemma~\ref{lem:F} as $\oF_G(\ell) = 0$ and $\nabla \oF_G(\ell) \neq 0$ for $\ell \in \calM^1(G)$.

\begin{proposition}\label{prop:oF}
The unit entropy moduli space $\calM^{1}(G)$ is a connected component of the level set $\{ \ell \in \calM(G) \mid \oF_{G}(\ell) = 0 \}$.
\end{proposition}

We also observe that the formulas for the metrics in Proposition~\ref{prop:metric formulas} hold for $\oF_G$.

\begin{proposition}\label{prop:o metric formulas}
If $\ell_{t} \from (-1,1) \to \calM^{1}(G)$ is a smooth path, then:\begin{align*} 
\norm{(\ell_{t},\dot\ell_{t})}_{\entropy,G}^{2} &= \frac{-\I{\dot\ell_{t},\bH [\oF_{G}(\ell_{t})]\dot\ell_{t}}}{\I{\ell_{t},\nabla \oF_{G}(\ell_{t})}}= \frac{\I{\ddot\ell_{t},\nabla \oF_{G}(\ell_{t})}}{\I{\ell_{t},\nabla \oF_{G}(\ell_{t})}}, \mbox{ and} \\
\norm{(\ell_{t},\dot\ell_{t})}_{\pressure,G}^{2} &= \frac{-\I{\dot\ell_{t},\bH [\oF_{G}(\ell_{t})]\dot\ell_{t}}}{\norm{\nabla \oF_{G}(\ell_{t})}_{1}}= \frac{\I{\ddot\ell_{t},\nabla \oF_{G}(\ell_{t})}}{\norm{\nabla \oF_{G}(\ell_{t})}_{1}}. 
\end{align*}
\end{proposition}


\section{The Topology Induced by the Entropy Metric}\label{sec:entropy topology}

The purpose of this section is to show that the metric topology induced by $d_\entropy$ on $\calX^1(\FF_r)$ is the same as the weak topology on $\calX^1(\FF_r)$.  We do so using the formulas for the entropy metric derived in Section~\ref{sec:determinant} and seeing that they behave as one might anticipate with regards to collapses.  We refer the reader back to Section~\ref{subsec:outer space} for the notation used in this section.

By Theorem~\ref{thm:entropy continuous} we have that $\entropy_G \from \calM(G) \to \RR$ extends to a continuous function on $\ocalM(G)$.  
Indeed, if $c\from G \to G_0$ is a collapse, then for $\ell \in \calM(G_0)$ we have $\entropy_G(c^*(\ell)) = \entropy_{G_0}(\ell)$.  We set $\ocalM^1(G) = \{ \ell \in \ocalM(G) \mid \entropy_G(\ell) = 1 \}$ and observe that we have 
\[ \ocalM^1(G) = \bigcup_{c \from G \to G_0} c^*(\calM^1(G_0)) \]
as well.  This set is homeomorphic to the closure of $CV(G,\rho)$ in $CV(\FF_r)$ for any marking $\rho \from \calR_r \to G$. 
%
%
Given a graph $G$, we observe that $F_G:\calM(G) \to \RR$ admits an extension (still denoted $F_G$) to $\RR^{\abs{E_+}}$.
In particular $F_G$ is defined on $\ocalM(G) \subset \RR^{\abs{E_+}}$.
The next result shows that this function behaves as expected with respect to collapses.    

\begin{lemma}\label{lem:F_G for collapse}
If $c \from G \to G_0$ is a collapse, then $F_G \circ c^* = F_{G_0}$.
\end{lemma}

\begin{proof}
It suffices to consider the case when $c \from G \to G_0$ is the collapse of a single edge $e \in E_{+}$.  Order the edges in $E$ starting with $e$ and $\bar{e}$.  Since $e$ can be collapsed, it is not a loop and so we have that $A_{G}(e,e) = A_{G}(\bar{e},\bar{e}) = 0$.  By definition, $A_{G}(e,\bar{e}) = A_{G}(\bar{e},e) = 0$.  Thus the top-leftmost $2 \times 2$ block of the matrix $I - A_{G}$ is the $2 \times 2$ identity matrix. 

Let $\ell \in \calM(G_0)$.  For an edge $e' \notin \{e,\bar{e}\}$, the entry $\bigl[I - A_{G,c^*(\ell)}\bigr](e,e')$ is either $-1$ or $0$ depending on whether or not $e'$ can follow $e$.  Likewise for $\bigl[I - A_{G,c^*(\ell)}\bigr](\bar{e},e')$.  Again, as $e$ is not a loop, for any edge $e' \notin \{e,\bar{e}\}$, at most one of these entries is nonzero.  

For each edge $e' \notin \{e, \bar{e}\}$, where $\bigl[I - A_{G,c^*(\ell)}\bigr](e,e') = -1$, we consider the column operation that adds the column for $e$ to the column for $e'$.  This zeros the $(e,e')$ entry.  The $(\bar{e},e')$ entry was previously 0 and is uneffected by this operation.  We next see what effect this has on the remaining rows.  In the row for $e'' \notin \{e,\bar{e}\}$, this adds $-\exp(-\ell(e''))$ to $A_{G,c^*(\ell)}(e'',e')$ if $e$ can follow $e''$ and 0 else.  In the former case, the previous entry was either 0 ($e' \neq e''$) or 1 ($e' = e''$) as $e$ is not a loop edge.  In other words, this adds $-\exp(-\ell(e''))$ whenever $e'$ can follow $e''$ in $G_0$.  Therefore, the remaining entries in the column for $e'$ agree with the corresponding entries in the column of $I - A_{G_0,\ell}$ for $e'$.  

Hence after performing column operations to $I - A_{G,c^*(\ell)}$ with the column for $e$ to clear out the rest of the row for $e$ and column operations with the column for $\bar{e}$ to clear out the rest of the row for $\bar{e}$, the resulting matrix has lower block triangular form.  The top-leftmost $2 \times 2$ block is still the $2 \times 2$ identity matrix and the bottom-rightmost $(\abs{E} - 2) \times (\abs{E} - 2)$ block is $I - A_{G_0,\ell}$.  

As these column operations do not change the determinant, we have for $\ell \in \calM(G_0)$ that
\[F_{G}(c^*(\ell)) = \det\bigl(I - A_{G,c^*(\ell)}\bigr) = \det\bigl(I - A_{G_0,\ell}\bigr) = F_{G_0}(\ell).\qedhere\]
\end{proof}

As a consequence of Lemma~\ref{lem:F_G for collapse} we deduce the following.  If $c \from G \to G_0$ is a collapse and $e \in E_+$ is an edge such that $c(e)$ is not a vertex, then $\PD{F_{G_0}}{c(e)}(\ell) = \PD{F_G}{e}(c^*(\ell))$ for all $\ell \in \calM(G_0)$.  Similarly, in this same setting if additionally $c(e')$ is not a vertex for an edge $e' \in E_+$, then $\PPD{F_{G_0}}{c(e)}{c(e')}(\ell) = \PPD{F_G}{e}{e'}(c^*(\ell))$.  

The \emph{tangent bundle} $T\ocalM^1(G)$ is the subspace of $(\ell,\bv) \in \RR^{\abs{E_+}} \times \RR^{\abs{E_+}}$ such that $\ell \in \ocalM^1(G)$ and $\I{\bv,\nabla F_G(\ell)} = 0$.
Given a collapse, we let $\bc^* \from \RR^{\abs{(E_0)_+}} \to \RR^{\abs{E_+}}$ be the derivative of the map $c^* \from \calM(G_0) \to \calM(G)$ and $Tc^* \from T\calM^1(G_0) \to T\ocalM^1(G)$ be the map given by $Tc^*(\ell,\bv) = (c^*(\ell),\bc^*(\bv))$.

Using this notation, we see that the following hold:
\begin{align*}
\I{\bc^*(\bv),\bH[F_{G}(c^*(\ell))]\bc^*(\bv)} &= \I{\bv, \bH[F_{G_0}(\ell)]\bv}, \mbox{ and} \\
\I{c^*(\ell),\nabla F_G(c^*(\ell))} &= \I{\ell,\nabla F_{G_0}(\ell)}.
\end{align*}   
Hence, by Proposition~\ref{prop:metric formulas}, we get the following.

\begin{proposition}\label{prop:topologies agree}
Let $G$ be a finite connected graph.  The entropy norm $\norm{\param}_{\entropy,G} \from T\calM^1(G) \to \RR$ extends to a continuous semi-norm $\norm{\param}_{\entropy,G} \from T\ocalM^1(G) \to \RR$.  Specifically, if $c \from G \to G_0$ is a collapse and $(\ell,\bv) = Tc^*(\ell_0,\bv_0)$, then the extension satisfies
\[  \norm{(\ell,\bv)}_{\entropy,G} = \norm{(\ell_0,\bv_0)}_{\entropy,G_0}.  \]
\end{proposition}

Consequently, we see that the distance function $d_{\entropy,G}$ extends to a distance function on $\ocalM^1(G)$ and that the induced topology is the same as the subspace topology.  As $\calX^1(\FF_r)$ is locally finite, the metric topology agrees with the weak topology as we show now.

\begin{theorem}\label{th:topologies agree}
The metric topology on $(\calX^1(\FF_r),d_\entropy)$ is the same as the weak topology on $\calX^1(\FF_r)$.
\end{theorem}

\begin{proof}
Let $U \subseteq \calX^1(\FF_r)$ be an open set in the weak topology and fix a marked metric graph $x = [(G,\rho,\ell)] \in U$.  There are finitely many marked graphs $(G_1,\rho_1), \ldots, (G_n,\rho_n)$ such that $(G,\rho) \leq (G_i,\rho_i)$.  By definition of the weak topology, the set $U \cap \ocalX^1(G_i,\rho_i)$ is open in $\ocalX^1(G_i,\rho_i)$ in the subspace topology inherited from $\RR_{\geq 0}^{\abs{(E_i)_+}}$, where $E_i$ is the set of edges for $G_i$.  As remarked above after Proposition~\ref{prop:topologies agree}, this set is also open in the metric topology induced by $d_{\entropy,G_i}$.  Hence, there is an $\epsilon_i > 0$ such that \[ \{ y \in \ocalX^1(G_i,\rho_i) \mid d_{\entropy,G_i}(x,y) < \epsilon_i \} \subseteq U \cap \ocalX^1(G_i,\rho_i).\]   
Let $\epsilon = \min\{ \epsilon_i \mid 1 \leq i \leq n \}$.  As $d_\entropy(x,y) \leq d_{\entropy,G_i}(x,y)$ for any $y \in \ocalX^1(G_i,\rho_i)$ we have
\[ \{ y \in \ocalX^1(\FF_r) \mid d_{\entropy}(x,y) < \epsilon \} \subseteq \bigcup_{i=1}^n U \cap \ocalX^1(G_i,\rho_i) \subseteq U.\] 
Hence the metric topology is finer than the subspace topology.

Next, fix a marked metric graph $x = [(G,\rho,\ell)] \in \calX^1(\FF_r)$ and let $\epsilon > 0$.  Enumerate the finitely many marked graphs $(G_1,\rho_1), \ldots, (G_n,\rho_n)$ such that $(G,\rho) \leq (G_i,\rho_i$) and such that $G_i$ is trivalent.  In other words, $(G_i,\rho_i)$ are maximal elements in the partial order on marked graphs.  As the norm varies continuously by Proposition~\ref{prop:topologies agree}, there is an $L$ and an open neighborhood $V \subseteq \cup_{i=1}^n \ocalX^1(G_i,\rho_i)$ of $x$ in the weak topology such that $\norm{(y,\bv)}_{\entropy,G_i} \leq L$ whenever $y \in \ocalX^1(G_i) \cap V$ and $\I{\bv,\bv} = 1$.  Therefore, there is an open neighborhood $U$ of $x$ in the weak topology such that   
\[ U \subseteq \{ y \in \ocalX^1(\FF_r) \mid d_{\entropy}(x,y) < \epsilon \}.\] 
Hence the subspace topology is finer than the metric topology.
\end{proof}


\section{The Entropy Metric on \texorpdfstring{$\calX^1(\FF_2)$}{X\textasciicircum1(F\_2) }}\label{sec:rank 2}

The goal of this section is to show that $(\calX^1(\FF_2),d_\entropy)$ is complete.  This appears as Proposition~\ref{prop:rank 2 complete} in Section~\ref{subsec:rank 2 complete}.  The results in this section are not necessary for the remainder of the paper and can safely be skipped by a reader primarily interested in Theorems~\ref{th:completion rose} and \ref{th:entropy bounded}.  However, the calculations can serve as a good introduction to estimating lengths with the entropy metric.  In particular, the main strategy in each of Lemmas~\ref{lem:2-rose}, \ref{lem:2-barbell} and \ref{lem:2-theta} is very similar to the main strategy of Lemma~\ref{lem:square root bound} in Section~\ref{subsec:infinite diameter} which is the key tool used to show that $(\calM^1(\calR_r),d_{\entropy,\calR_r})$ has infinite diameter.

To begin, we analyze the metric for each of the three topological types of graphs that appear in rank $2$: the 2--rose $\calR_2$, the barbell graph $\calB_2$ and the theta graph $\Theta_2$. We refer the reader back to Figure~\ref{fig:rank two graphs} for these graphs.  To this end, we define a continuous function $\sfm \from \calX^1(\FF_2) \to \RR$ which is a slight variation of the volume function in that it counts separating edges twice.  In particular, it does not depend on the marking.  The exact definition appears in Section~\ref{subsec:rank 2 complete}.  The strategy is to show that for any path $\ell_t \from [0,1] \to \calM^1(G)$ for $G \in \{\calR_2,\calB_2,\Theta_2\}$, if $\sfm(\ell_0)$ and $\sfm(\ell_1)$ are large enough, then the length of $\ell_t$ is bounded below by
\[ \frac{1}{\sqrt{5}}\left(\sqrt{\sfm(\ell_1)} - \sqrt{\sfm(\ell_0)}\right). \]
These calculations appear in the next three sections (Propositions~\ref{prop:2-rose},~\ref{prop:2-barbell},~and~\ref{prop:2-theta}).

Using these estimates, it is not too hard to see that if $(x_n)_{n \in \NN} \subset \calX^1(\FF_2)$ is Cauchy, then there is an $L$ such that $\sfm(x_n) \leq L$ for all $n$ (Lemma~\ref{lem:cauchy is thick}).  From here, using local finiteness of $\calX^1(\FF_2)$ and a compactness argument, the completeness of $(\calX^1(\FF_2),d_\entropy)$ follows.  In the calculations, we make use of Lemmas~\ref{lem:dot product} and \ref{lem:hessian product}.


\subsection{The 2--Rose}\label{subsec:rose}
Denote the edges of $\calR_2$ by $e_1$ and $e_2$.  To make the calculations in this subsection easier to read, given a length function $\ell \in \calM(\calR_2)$ we set $a = \ell(e_1)$, $b = \ell(e_2)$ and $m = \ell(e_1) + \ell(e_2)$.  Applying the definition of $\oF_G$ from Section~\ref{subsec:simplification} to $\calR_2$, we find the following formula.
\begin{equation}\label{eq:2-rose F} 
\oF_{\calR_2}(\ell) = 1 - \exp(-a) - \exp(-b) - 3\exp(-m)    
\end{equation} 

\begin{lemma}\label{lem:2-rose}
Suppose $\ell_t \from [0,1] \to \calM^{1}(\calR_{2})$ is a smooth path such that for all $t \in [0,1]$ we have $\dot m_t > 0$.  If $m_0 \geq 4$, then
\[ \calL_{\entropy,\calR_2}(\ell_t|[0,1]) \geq \sqrt{m_1} - \sqrt{m_0}.
\]
\end{lemma}

\begin{proof}
Suppose that $\ell_t \from [0,1] \to \calM^1(\calR_2)$ is a path where $\dot m_t > 0$ as in the statement of the lemma and assume that $m_0 \geq 4$.  We reparameterize the path $\ell_t$ so that $m_t = t$.  

Let $n_t = \min\{a_t,b_t\}$.  As $\oF_{\calR_2}(\ell_t) = 0$, we have
\[ 1 - 2\exp(-n_t) \leq  1 - \exp(-a_t) - \exp(-b_t) = 3\exp(-m_t) = 3\exp(-t). \]
In particular, $2\exp(-n_t) \geq 1 - 3\exp(-t) \geq 2\exp(-1)$ as $t \geq 4$ and so $n_t < 1$ for all $t$.  

Setting $p_t = \max\{a_t,b_t\}$, we find that $p_t = t - n_t \geq t-1$.  Therefore, as $t - 1 \geq 1$ for $t \geq 4$, the edge that realizes the minimum of $\{a_t,b_t\}$ does not depend on $t$ and hence we may assume that $b_t = n_t$ and that $a_t = p_t \geq t-1$.  This gives us that $\exp(-a_t) \leq \exp(-t + 1)$.  Hence, again as $\oF_{\calR_2}(\ell_t) = 0$, we have
\[ 1 - \exp(-b_t) = \exp(-a_t) + 3\exp(-m_t) \leq  (\exp(1) + 3)\exp(-t) \
\leq 8\exp(-t). \]

This enables us to give an upper bound on the denominator in the expression for the entropy norm.  Specifically, using the fact that $x\exp(-x) \leq 1 - \exp(-x)$ for $x \geq 0$ we have
\begin{align*}
\I{\ell_t,\nabla \oF_{\calR_2}(\ell_t)} & = a_t\exp(-a_t) + b_t\exp(-b_t) +3m_t\exp(-m_t) \\
& \leq t\exp(-t+1) + (1 - \exp(-b_t)) +3t\exp(-t) \\
& \leq 8t\exp(-t) + 8\exp(-t) \\
& \leq 12t\exp(-t).
\end{align*}
In the final inequality we used that fact that $t \geq 4$.  The expression for $\I{\ell_t,\nabla \oF_{\calR_2}(\ell_t)}$ can either be computed directly from \eqref{eq:2-rose F} or via Lemma~\ref{lem:dot product}.

Next, we get an upper bound on the numerator in the expression for the entropy norm by just using $m_t$.  Specifically,
\begin{align*}
-\I{\dot \ell_t,\bH[\oF_{\calR_2}(\ell_t)]\dot \ell_t} & = (\dot a_t)^2\exp(-a_t) + (\dot b_t)^2\exp(-b_t) + 3(\dot m_t)^2\exp(-m_t) \geq 3\exp(-t).
\end{align*}
As above, the expression for $-\I{\dot \ell_t,\bH[\oF_{\calR_2}(\ell_t)]\dot \ell_t}$ can either be computed directly from \eqref{eq:2-rose F} or via Lemma~\ref{lem:hessian product}.

Hence we find that the entropy norm along this path is bounded below by
\[ \norm{(\ell_t,\dot \ell_t)}^2_{\entropy,\calR_2} = \frac{-\I{\dot \ell_t,\bH[\oF_{\calR_2}(\ell_t)]\dot \ell_t}}{\I{\ell_t,\nabla \oF_{\calR_2}(\ell_t)}}\geq \frac{1}{4t}. \]
Therefore the length of this path in the entropy metric is at least
\begin{align*}
\int_{m_0}^{m_1} \sqrt{\frac{1}{4t}} \, dt &= \sqrt{m_1} - \sqrt{m_0}.\qedhere
\end{align*}
\end{proof}

Using this lemma, we can get a lower bound on the distance between length functions in $\calM^1(\calR_2)$ in terms of the sum of the lengths of edges, so long as they are sufficiently large.

\begin{proposition}\label{prop:2-rose}
Suppose $\ell$ and $\ell'$ are length functions in $\calM^1(\calR_2)$ where $m = \ell(e_1) + \ell(e_2)$ and $m' = \ell'(e_1) + \ell'(e_2)$ are at least $4$.  Then
\begin{equation*}
d_{\entropy,\calR_2}(\ell,\ell') \geq \sqrt{m'} - \sqrt{m}.
\end{equation*}
\end{proposition}

\begin{proof}
Let $\ell_t \from [0,1] \to \calM^1(\calR_2)$ be a piecewise smooth path such that $\ell_0 = \ell$ and $\ell_1 = \ell'$.  Let $\delta \in [0,1]$ be the minimal value such that $m_t \geq 4$ for $t \in [\delta,1]$.  In particular, $m_\delta \leq m$.

By only considering the smooth subpaths of $\ell_t|[\delta,1]$ for which $\dot m_t > 0$, by Lemma~\ref{lem:2-rose}, we find that
\begin{equation*} 
\calL_{\entropy,\calR_2}(\ell_t|[\delta,1]) \geq \sqrt{m'} - \sqrt{m_\delta} \geq \sqrt{m'} - \sqrt{m}. 
\end{equation*}
This also provides a lower bound on $\calL_{\entropy,\calR_2}(\ell_t|[0,1])$.  Since the path was arbitrary this also is a lower bound on the distance between $\ell$ and $\ell'$.
\end{proof}


\subsection{The Barbell Graph}\label{subsec:barbell}
Let $\calB_2$ denote the graph with vertices $v$ and $w$, and edges $e_1$, $e_2$ and $e_3$ where $o(e_1) = \tau(e_1) = v$, $o(e_2) = \tau(e_2) = w$, and $o(e_3) = v$ and $\tau(e_3) = w$.  To make the calculations in this section easier to read, given a length function $\ell \in \calM(\calB_2)$ we set $a = \ell(e_1)$, $b = \ell(e_2)$, and $m = \ell(e_1) + \ell(e_2) + 2\ell(e_3)$.    Applying the definition of $\oF_G$ from Section~\ref{subsec:simplification} to $\calB_2$, we find the following formula. 
\begin{align*} 
\oF_{\calB_2}(\ell) & = (1 - \exp(-a))(1 - \exp(-b)) - 4\exp(-m)
\end{align*}

\begin{lemma}\label{lem:2-barbell}
Suppose $\ell_t \from [0,1] \to \calM^{1}(\calB_{2})$ is a smooth path such that for all $t \in [0,1]$ we have $\dot m_t > 0$.  If $m_0 \geq 4$, then
\[ \calL_{\entropy,\calB_2}(\ell_t|[0,1]) \geq \frac{1}{\sqrt{2}} \left(\sqrt{m_1} - \sqrt{m_0}\right).
 \]
\end{lemma}

\begin{proof}
Suppose that $\ell_t \from [0,1] \to \calM^1(\calB_2)$ is a path where $\dot m_t > 0$ as in the statement of the lemma and assume that $m_0 \geq 4$.  We reparameterize the path $\ell_t$ so that $m_t = t$.

As $\oF_{\calB_2}(\ell_t) = 0$, we have
\begin{equation}\label{barbell F}
(1-\exp(-a_t))(1 - \exp(-b_t)) = 4\exp(-m_t) = 4\exp(-t). 
\end{equation} 
This enables us to give an upper bound on the denominator in the expression for the entropy norm.  Specifically, using the fact that $x\exp(-x) \leq 1 - \exp(-x)$ for $x \geq 0$ we have
\begin{align*}
\I{\ell_t,\nabla \oF_{\calB_2}(\ell_t)} & = a_t\exp(-a_t)(1 - \exp(-b_t)) + b_t\exp(-b_t)(1 - \exp(-a_t)) + 4m_t\exp(-m_t) \\
& \leq 2(1-\exp(-a_t))(1-\exp(-b_t)) + 4t\exp(-t) \\
& = 4t\exp(-t) + 8\exp(-t) \\
& \leq 8t\exp(-t).
\end{align*}
The penultimate line follows from~\eqref{barbell F}, and in the final inequality we use the fact that $t \geq 4$. 

Next, we get a lower bound on the numerator in the expression for the entropy norm.  We claim that $\I{\dot \ell_t,\bH[\oF_{\calB_2}(\ell_t)]\dot \ell_t} \geq (\dot m_t)^2\exp(-m_t)$.  We have that
\begin{align*}
-\I{\dot \ell_t,\bH[\oF_{\calB_2}(\ell_t)]\dot \ell_t} & = (\dot a_t)^2\exp(-a_t)(1 - \exp(-b_t)) + (\dot b_t)^2\exp(-b_t)(1 - \exp(-a_t)) \\
& \qquad  - 2\dot a_t \dot b_t\exp(-a_t - b_t) + 4(\dot m_t)^2\exp(-m_t).
\end{align*}
Therefore if $\dot a_t \dot b_t < 0$ then each term is positive and so the claim holds.  Therefore, we assume that $\dot a_t$ and $\dot b_t$ have the same sign.  As $\I{\dot \ell_t , \nabla \oF_{\calB_2}(\ell_t)} = 0$, we have that
\begin{equation*}
4\dot m_t\exp(-m_t) = -\dot a_t\exp(-a_t)(1 - \exp(-b_t)) -\dot b_t\exp(-b_t)(1 - \exp(-a_t)).
\end{equation*}
We write this equation as $w = u + v$.  As $\oF_{\calB_2}(\ell_t) = 0$, we find that
\begin{align*} 
2uv & = 2\dot a_t \dot b_t\exp(-a_t)\exp(-b_t)(1 - \exp(-a_t))(1 - \exp(-b_t)) = 2\dot a_t \dot b_t\exp(-a_t-b_t)\bigl(4\exp(-m_t)\bigr).
\end{align*}
As $2xy \leq \frac{3}{4}(x+y)^2$ for all $x$ and $y$, we find that
\begin{align*}
2\dot a_t \dot b_t\exp(-a_t-b_t)\bigl(4\exp(-m_t)\bigr) & \leq \frac{3}{4}\bigl(4\dot m_t\exp(- m_t)\bigr)^2 
 = 3(\dot m_t)^2\bigl(4\exp(-2m_t)\bigr).
\end{align*}
Therefore $2\dot a_t \dot b_t\exp(-a_t-b_t) \leq 3(\dot m_t)^2\exp(-m_t)$.  From this the claim now follows, and furthermore that \[-\I{\dot \ell_t,\bH[\oF_{\calB_2}(\ell_t)]\dot \ell_t} \geq (\dot m_t)^2\exp(- m_t) = \exp(-t).\] 

Hence we find that the entropy norm along this path is bounded below by
\[ \norm{(\ell_t,\dot \ell_t)}^2_{\entropy,\calB_2} = \frac{-\I{\dot \ell_t,\bH[\oF_{\calB_2}(\ell_t)]\dot \ell_t}}{\I{\ell_t,\nabla \oF_{\calB_2}(\ell_t)}}\geq \frac{1}{8t}. \]
Therefore the length of this path in the entropy metric is at least
\begin{align*}
\int_{m_0}^{m_1} \sqrt{\frac{1}{8t}} \, dt &= \frac{1}{\sqrt{2}} \left(\sqrt{m_1} - \sqrt{m_0}\right).\qedhere
\end{align*}
\end{proof}

As for the 2--rose, we obtain the following proposition.

\begin{proposition}\label{prop:2-barbell}
Suppose $\ell$ and $\ell'$ are length functions in $\calM^1(\calB_2)$ where $m = \ell(e_1) + \ell(e_2) + 2\ell(e_3)$ and $m' = \ell'(e_1) + \ell'(e_2) + 2\ell'(e_3)$ are at least $4$.  Then
\begin{equation*}
d_{\entropy,\calB_2}(\ell,\ell') \geq \frac{1}{\sqrt{2}} \left(\sqrt{m'} - \sqrt{m}\right).
\end{equation*}
\end{proposition}


\subsection{The Theta Graph}\label{subsec:theta}
Let $\Theta_2$ denote the graph with vertices $v$ and $w$, and edges $e_1$, $e_2$ and $e_3$ where $o(e_i) = v$ and $\tau(e_i) = w$ for $i \in \{1,2,3\}$.  To make the calculations in this section easier to read, given a length function $\ell \in \calM(\Theta_2)$ we set $a = \ell(e_1)  + \ell(e_2)$, $b = \ell(e_2) + \ell(e_3)$, $c = \ell(e_3) + \ell(e_1)$ and $m = \ell(e_1) + \ell(e_2) + \ell(e_3)$.  Applying the definition of $\oF_G$ from Section~\ref{subsec:simplification} to $\Theta_2$, we find the following formula. 
\[ \oF_{\Theta_2}(\ell) = 1 - \exp(-a) - \exp(-b) - \exp(-c) - 2\exp(-m). \]

\begin{lemma}\label{lem:2-theta}
Suppose $\ell_t \from [0,1] \to \calM^{1}(\Theta_{2})$ is a smooth path such that for all $t \in [0,1]$ we have $\dot m_t > 0$.  If $m_0 \geq 4$, then
\[ \calL_{\entropy,\Theta_2}(\ell_t|[0,1]) \geq \frac{1}{\sqrt{5}} \left(\sqrt{m_1} - \sqrt{m_0}\right).
 \]
\end{lemma}

\begin{proof}
Suppose that $\ell_t \from [0,1] \to \calM^1(\Theta_2)$ is a path where $\dot m_t > 0$ as in the statement of the lemma and assume that $m_0 \geq 4$.  We reparameterize the path $\ell_t$ so that $m_t = t$.

Let $n_t = \min\{a_t,b_t,c_t\}$.  As $\oF_{\Theta_2}(\ell_t) = 0$, we have
\[ 1 - 3\exp(-n_t) \leq 1 - \exp(-a_t) - \exp(-b_t) - \exp(-c_t) = 2\exp(-m_t) = 2\exp(-t). \]
In particular, $3\exp(-n_t) \geq 1- 2\exp(-t) > 3\exp(-2)$ as $t \geq 4$ and thus $n_t < 2$ for all $t$. 

Setting $p_t = \max\{a_t,b_t,c_t\}$ and $q_t = a_t + b_t + c_t - p_t - n_t$ so that $\{ p_t,q_t,n_t \} = \{ a_t,b_t,c_t \}$ for all $t$, we find that $p_t,q_t \geq t-2$ as $p_t + q_t = 2t - n_t \geq 2t - 2$ and $p_t,q_t \leq t$.  Therefore, as $t - 2 \geq 2$ for $t \geq 4$, the cycle that realizes the minimum of $\{a_t,b_t\,c_t\}$ does not depend on $t$ and therefore we may assume that $c_t = \min\{a_t,b_t,c_t\}$ and thus $a_t,b_t \geq t-2$.  Therefore $\exp(-a_t),\exp(-b_t) \leq \exp(-t+2)$.  Hence, again as $\oF_{\Theta_2}(\ell_t) = 0$, we have
\[ 1 - \exp(-c_t) = \exp(-a_t) + \exp(-b_t) + 2\exp(-m_t) \leq (2\exp(2) + 2)\exp(-t) \leq 20\exp(-t). \]     


This enables us to give an upper bound on the denominator in the expression for the entropy norm.  Specifically, using the fact that $x\exp(-x) \leq 1 - \exp(-x)$ for $x \geq 0$ we have
\begin{align*}
\I{\ell_t,\nabla \oF_{\Theta_2}(\ell_t)} & = a_t\exp(-a_t) + b_t\exp(-b_t) + c_t\exp(-c_t)+ 2m_t\exp(-m_t) \\
& \leq a_t\exp(-a_t) + b_t\exp(-b_t) + 1 - \exp(-c_t)+ 2m_t\exp(-m_t) \\
& \leq t\exp(-t+2) + t\exp(-t+2) + 20\exp(-t) + 2t\exp(-t) \\
& \leq 20t\exp(-t) + 20\exp(-t) \\
& \leq 40t\exp(-t). 
\end{align*}
In the final inequality we used the fact that $t \geq 4$.

Next, we get a lower bound on the numerator in the expression for the entropy norm by just using $m_t$.  Specifically,  
\begin{align*}
-\I{\dot \ell_t,\bH[\oF_{\Theta_2}(\ell_t)]\dot \ell_t} & = (\dot a_t)^2\exp(-a_t) + (\dot b_t)^2\exp(-b_t) + (\dot c_t)^2\exp(-c_t) + 2(\dot m_t)^2\exp(-m_t) \\
& \geq 2\exp(-t).
\end{align*}

Hence we find that the entropy norm along this path is bounded below by
\[ \norm{(\ell_t,\dot \ell_t)}^2_{\entropy,\Theta_2} = \frac{-\I{\dot \ell_t,\bH[\oF_{\Theta_2}(\ell_t)]\dot \ell_t}}{\I{\ell_t,\nabla \oF_{\Theta_2}(\ell_t)}}\geq \frac{1}{20t}. \]
Therefore the length of this path in the entropy metric is at least
\begin{align*}
\int_{m_0}^{m_1} \sqrt{\frac{1}{20t}} \, dt &= \frac{1}{\sqrt{5}} \left(\sqrt{m_1} - \sqrt{m_0}\right).\qedhere
\end{align*}
\end{proof}

Again, as for the 2--rose, we obtain the following proposition.

\begin{proposition}\label{prop:2-theta}
Suppose $\ell$ and $\ell'$ are length functions in $\calM^1(\Theta_2)$ where $m = \ell(e_1) + \ell(e_2) + \ell(e_3)$ and $m' = \ell'(e_1) + \ell'(e_2) + \ell'(e_3)$ are at least $4$.  Then
\begin{equation*}
d_{\entropy,\Theta_2}(\ell,\ell') \geq \frac{1}{\sqrt{5}} \left(\sqrt{m'} - \sqrt{m}\right).
\end{equation*}
\end{proposition}

\subsection{\texorpdfstring{$(\calX^1(\FF_2),d_\entropy)$}{(X\textasciicircum1(F\_2),d\_h)} Is Complete}\label{subsec:rank 2 complete}

We can now prove the main result of this section that the entropy metric on $\calX^1(\FF_2)$ is complete.  Given a marked metric graph $x = [(G,\rho,\ell)]$ in $\calX^1(\FF_2)$ we let 
\[\sfm(x) = \begin{cases}
\ell(e_1) + \ell(e_2) & \mbox{if } G = \calR_2, \\
\ell(e_1) + \ell(e_2) + 2\ell(e_3) & \mbox{if } G = \calB_2, \\
\ell(e_1) + \ell(e_2) + \ell(e_3) & \mbox{if } G = \Theta_2 \\
\end{cases}\]  
We remark that $\sfm \from \calX^1(\FF_2) \to \RR$ is an $\Out(\FF_2)$--invariant continuous function.

\begin{lemma}\label{lem:cauchy is thick}
Let $(x_n)_{n \in \NN}$ be a Cauchy sequence in $(\calX^1(\FF_2),d_\entropy)$.  Then there is an $L$ such that $\sfm(x_n) \leq L$ for all $n$.
\end{lemma}

\begin{proof}
Suppose that $(x_n)_{n \in \NN}$ is a sequence in $\calX^1(\FF_2)$ such that $\sfm(x_n) \to \infty$.  We will show that $(x_n)_{n \in \NN}$ is not Cauchy by showing that $\limsup d_\entropy(x_1,x_n) = \infty$.  

Let $m = \max\{\sfm(x_1),4\}$.  Given $N \geq 0$, we let $n$ be such that $\sqrt{\sfm(x_n)} \geq \sqrt{5}N + \sqrt{m}$.  Consider a path $\alpha_t \from [0,1] \to \calX^1(\FF_2)$ such that $\alpha_0 = x_1$ and $\alpha_1 = x_n$.  Let $\delta \in [0,1]$ be the minimal value such that $\sfm(\alpha_t) \geq 4$ for all $t \in [\delta,1]$.  In particular, $\sfm(\alpha_\delta) \leq m$. 

Combining Propositions~\ref{prop:2-rose}, \ref{prop:2-barbell} and \ref{prop:2-theta}, we find that 
\[ \calL_\entropy(\alpha_t|[\delta,1]) \geq \frac{1}{\sqrt{5}}\left( \sqrt{\sfm(x_n)} - \sqrt{\sfm(\alpha_\delta)}\right) \geq \frac{1}{\sqrt{5}}\left( \sqrt{\sfm(x_n)} - \sqrt{m}\right) \geq N. \]
This also provides a lower bound on $\calL_\entropy(\alpha_t|[0,1])$.  Since the path was arbitrary this also is a lower bound on $d_\entropy(x_1,x_n)$.

Therefore $d_\entropy(x_1,x_n) \geq N$, showing that $\limsup d_\entropy(x_1,x_n) = \infty$ as claimed.
\end{proof}

Given $L \geq 0$, we set $\calX^1_L(\FF_2) = \{ x \in \calX^1(\FF_2) \mid \sfm(x) \leq L \}$ and additionally given a marked graph $\rho \from \calR_2 \to G$, we set $\calX^1_L(G,\rho) = \calX^1(G,\rho) \cap \calX^1_L(\FF_2)$.  We remark that the closure of $\calX^1_L(G,\rho)$ is compact.  As $\calX^1(\FF_2)$ is locally finite and there are only finitely many topological types of graphs, the following statement holds.  For all $L, D \geq 0$, there is an $N$ such that if $x \in \calX^1_L(\FF_2)$ then there is a collection of marked graphs $(G_1,\rho_1),\ldots,(G_N,\rho_N)$ such that
\[ \{x' \in \calX^1_L(\FF_2) \mid d_\entropy(x,x') \leq D \} \subseteq  \bigcup_{k=1}^N \calX^1_L(G_k,\rho_k).\] 

\begin{proposition}\label{prop:rank 2 complete}
The metric space $(\calX^1(\FF_2),d_\entropy)$ is complete.
\end{proposition}

\begin{proof}
Let $(x_n)_{n \in \NN}$ be a Cauchy sequence in $(\calX^1(\FF_2),d_\entropy)$.  By Lemma~\ref{lem:cauchy is thick}, there is an $L$ such that $(x_n) \subset \calX^1_L(\FF_2)$.  Let $n_0$ be such that $d_\entropy(x_n,x_m) \leq 1$ if $n,m \geq n_0$.  By the above remark, there are finitely marked graphs $(G_1,\rho_1), \ldots, (G_N,\rho_N)$ such that
\[ \{ x_n | n \geq n_0 \} \subset \{ x' \in \calX^1_L(\FF_2) \mid d_\entropy(x_{n_0},x') \leq 1 \} \subseteq \bigcup_{k=1}^N \calX^1_L(G_k,\rho_k). \]  As the closure of this set in $\calX^1(\FF_2)$ is compact, the sequence $(x_n)_{n \in \NN}$ converges.   
\end{proof}


\section{The Moduli Space of the Rose}\label{sec:rose}

The purpose of this section is to examine the entropy metric on the moduli space of an $r$--rose $\calM^1(\calR_r)$.  We begin in Section~\ref{subsec:Fr} by computing the function $\oF_{\calR_r}$ introduced in Section~\ref{subsec:simplification}.  For the $r$--rose, we can strengthen Proposition~\ref{prop:oF} and conclude that $\calM^1(\calR_r)$ equals the set of length functions $\ell$ for which $\oF_{\calR_r}(\ell) = 0$.  Next, in Section~\ref{subsec:not complete} we show that $(\calM^1(\calR_r),d_{\entropy,\calR_r})$ for $r \geq 3$ is not complete by producing paths that have finite length yet no accumulation point.  Lastly, in Section~\ref{subsec:infinite diameter} we show that $(\calM^1(\calR_r),d_{\entropy,\calR_r})$ has infinite diameter.  Specifically, a path that shrinks the length of an edge to zero necessarily has infinite length.  


\subsection{\texorpdfstring{$\calM^1(\calR_r)$}{M\textasciicircum1(R\_r)} as a Zero Locus}\label{subsec:Fr}

In this section we compute the function $\oF_{\calR_r}$.  This appears as Proposition~\ref{prop:Frose}.  In Proposition~\ref{prop:oF = 0 iff h = 1} we prove that $\calM^1(\calR_r) = \{\ell \in \calM(\calR_r) \mid \oF_{\calR_r}(\ell) = 0 \}$; this strengthen Proposition~\ref{prop:oF} in this setting.

First we set some notation for working with the graph $\calR_r$.  We identify the unoriented edges of $\calR_{r}$ with the set $[r] = \{1,2,\ldots,r\}$.  To simplify the expressions, we will use $r$ as the identifying subscript rather than $\calR_{r}$ and we will use variables $\ell = (\ell^1,\ldots, \ell^r)$ to denote the length of the unoriented edges.  The matrix $\oA_{r,\ell} \in \Mat_{r}(\RR)$ has rows and columns indexed by $[r]$ and we have
\begin{equation}\label{eq:Ar}
\oA_{r,\ell}(i,j) = \exp(-\ell^i)(2-\delta(i,j))
\end{equation}
where $\delta(\param,\param)$ is the Kronecker delta function.  

For the calculations in this section, we need the following combinatorial identities.

\begin{lemma}\label{lem:comb identity}
For any $r \geq 1$ and any $x \in \RR$ the following equations hold.
\begin{align}
(1 + x)^{r-1}\bigl(x - (2r-1)\bigr) &= x^r\sum_{k=0}^{r} \genfrac{(}{)}{0pt}{0}{r}{k} (1-2k) x^{-k}.\label{eq:X} \\
(1 + x)^{r-1}\bigl(x + (2r+1)\bigr) &= x^r\sum_{k=0}^{r} \genfrac{(}{)}{0pt}{0}{r}{k} (1+2k) x^{-k}.\label{eq:Y}
\end{align}
\end{lemma}

\begin{proof}
Differentiate the equation $(1+x)^{r} = \sum_{k=0}^{r}\genfrac{(}{)}{0pt}{1}{r}{k} x^{r-k}$ and multiply it by $x$ to obtain the equality $rx(1+x)^{r-1} = \sum_{k=0}^{r} (r-k) \genfrac{(}{)}{0pt}{1}{r}{k} x^{r-k}$.  Therefore
\begin{align*}
2rx(1+x)^{r-1} - (2r-1)(1+x)^{r}  &= 2r\sum_{k=0}^{r} \genfrac{(}{)}{0pt}{0}{r}{k} x^{r-k} - \sum_{k=0}^{r} 2k \genfrac{(}{)}{0pt}{0}{r}{k} x^{r-k} \\
& \qquad - 2r\sum_{k=0}^{r} \genfrac{(}{)}{0pt}{0}{r}{k} x^{r-k} + \sum_{k=0}^{r} \genfrac{(}{)}{0pt}{0}{r}{k} x^{r-k} \\
& = \sum_{k=0}^{r} (1 - 2k) \genfrac{(}{)}{0pt}{0}{r}{k}x^{r-k}  = x^{r}\sum_{k=0}^{r}(1-2k)\genfrac{(}{)}{0pt}{0}{r}{k} x^{-k}.
\end{align*}
The left hand side in the above equation simplifies to $(1 + x)^{r-1}\bigl(x - (2r-1)\bigr)$.  This shows~\eqref{eq:X}.  In a similar manner one can derive \eqref{eq:Y}.
\end{proof}

\begin{corollary}\label{co:comb identity}
For any $r \geq 1$,
\begin{equation}\label{eq:X = 0}
 \sum_{k=0}^{r} \genfrac{(}{)}{0pt}{0}{r}{k} (1-2k)(2r-1)^{-k} = 0.
\end{equation}
\end{corollary}

\begin{proof}
Evaluate equation~\eqref{eq:X} with $x = 2r-1$. The left hand side becomes $0$.  Dividing the resulting equation by $(2r-1)^{r}$, we obtain \eqref{eq:X = 0}.
\end{proof}

Given a length function $\ell = (\ell^1,\ldots,\ell^n) \in \calM(\calR_r)$ and a subset $S \subseteq [r]$, we define $\ell^S = \sum_{k \in S} \ell^k$.  In particular, we have $\ell^\emptyset = 0$.  

\begin{proposition}\label{prop:Frose}
For any $r \geq 2$ and any length function $\ell \in \calM(\calR_r)$,
\begin{equation}\label{eq:Frose}
\oF_{r}(\ell) = \sum_{S \subseteq [r]} (1 - 2\abs{S})\exp\left(-\ell^S\right).
\end{equation}
\end{proposition}

\begin{proof}
Using the expansion of the determinant via permutations of $[r]$, we can express $\oF_{r}(\ell) = \det(I - \oA_{r,\ell})$ as
\begin{equation*}
\oF_{r}(\ell) = \sum_{S \subseteq [r]} c_{S,r}\exp\left(-\ell^{S}\right)
\end{equation*}
for some coefficients $c_{S,r} \in \RR$ depending on the subset $S \subseteq [r]$ and the rank $r$.  Further, it is apparent that the coefficient $c_{S,r}$ only depends on the cardinality of $S$.  It remains to determine these coefficients.  We will do so by induction.

For $r=2$, we compute:
\begin{align*}
\oF_{2}(\ell^{1},\ell^{2}) &= \det\left[\begin{smallmatrix} 1 - \exp(-\ell^{1}) & -2\exp(-\ell^{1}) \\ -2\exp(-\ell^{2}) & 1 - \exp(-\ell^{2}) \end{smallmatrix}\right] \\
& = 1 - \exp(-\ell^{1}) - \exp(-\ell^{2}) - 3\exp(-\ell^{1} - \ell^{2}).
\end{align*}
This shows the proposition for $r = 2$.

Suppose $r \geq 3$ and that the proposition holds for $r-1$.  That is, we assume that $c_{S,r-1} = 1-2\abs{S}$ for any $S \subseteq [r-1]$.  Since $c_{S,r}$ only depends on the cardinality of $S$, this implies that $c_{S,r} = 1 - 2\abs{S}$ for any $S \subseteq [r]$ where $\abs{S} < r$ as well.  Hence, it only remains to compute $c_{[r],r}$.  

To compute $c_{[r],r}$, we make use of Corollary~\ref{co:comb identity}.  Indeed, by Example~\ref{ex:rose}, we have $\entropy_r(\log(2r-1) \cdot \One) = 1$.  Therefore by Proposition~\ref{prop:oF} we obtain $\oF_{r}(\log(2r-1) \cdot \One) = 0$.  Hence
\begin{align*}
0 & = \oF_{r}(\log(2r-1) \cdot \One) \\
& = \sum_{S \subset [r]} (1 - 2\abs{S}) \exp\left(-\abs{S}\log(2r-1)\right) + c_{[r],r}\exp\left(-r\log(2r-1)\right) \\
& = \sum_{k=0}^{r-1} \genfrac{(}{)}{0pt}{0}{r}{k} (1-2k)(2r-1)^{-k} + c_{[r],r}(2r-1)^{-r}.
\end{align*}
By Corollary~\ref{co:comb identity}, we find that $c_{[r],r} = 1-2r$ as desired.
\end{proof}

\begin{example}\label{ex:rose moduli}
For $r = 2$ and $r = 3$, using the coordinates $x = \exp(-\ell^{1})$, $y = \exp(-\ell^{2})$ and $z = \exp(-\ell^{3})$ we find
\begin{align*}
\oF_{2}(\ell^{1},\ell^{2}) &=  1 - x - y - 3xy \text{, and}\\
\oF_{3}(\ell^{1},\ell^{2},\ell^{3}) &= 1 - x - y - z - 3xy - 3xz - 3yz - 5xyz.
\end{align*}
Figure~\ref{fig:moduli space} shows $\calM^{1}(\calR_{r})$ as a subset of $\calM(\calR_r)$ for $r = 2$ and $r = 3$. 
\end{example}

\begin{figure}[ht]
\centering
\includegraphics[width=0.4\textwidth]{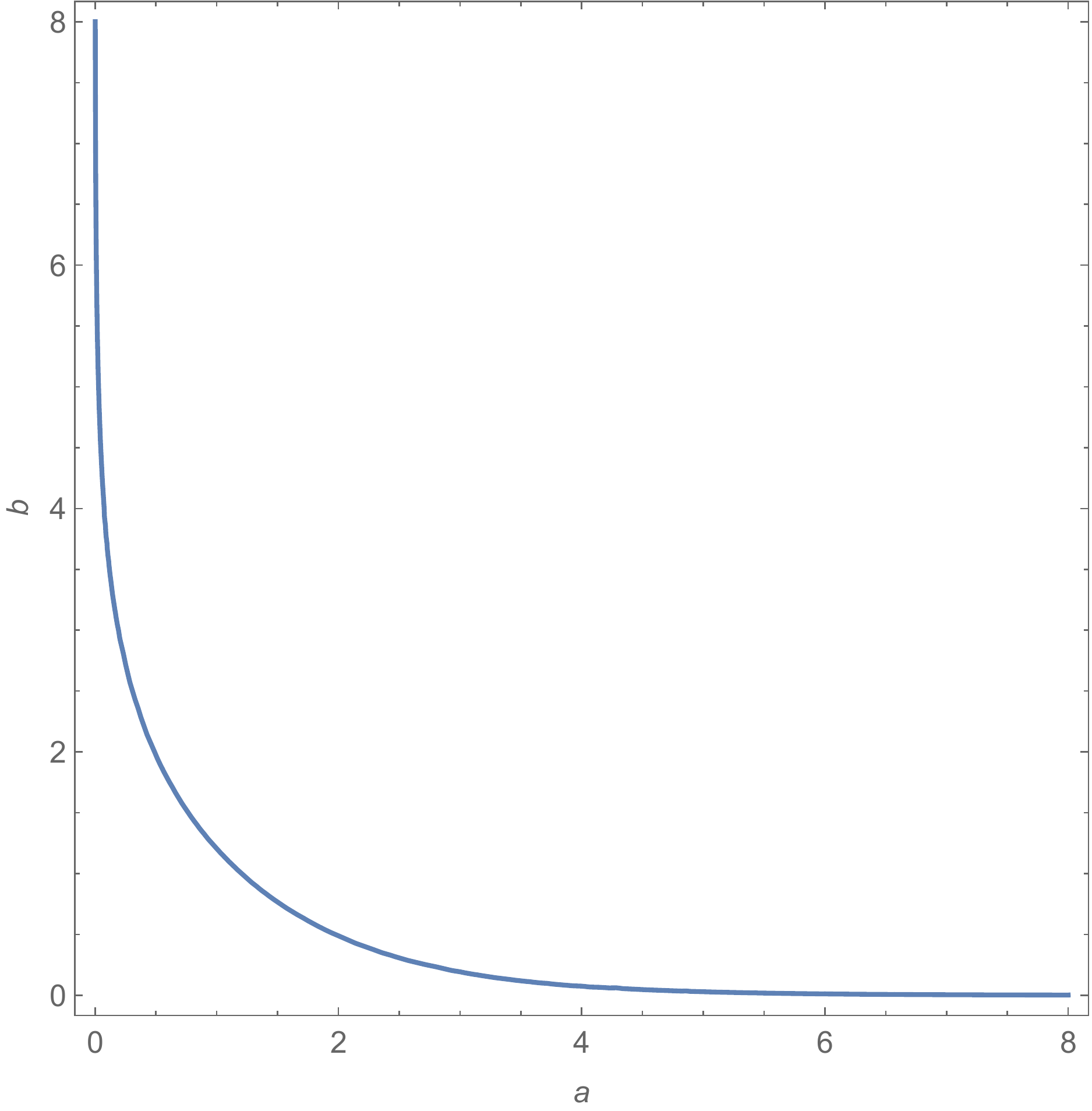}
\quad
\includegraphics[width=0.47\textwidth]{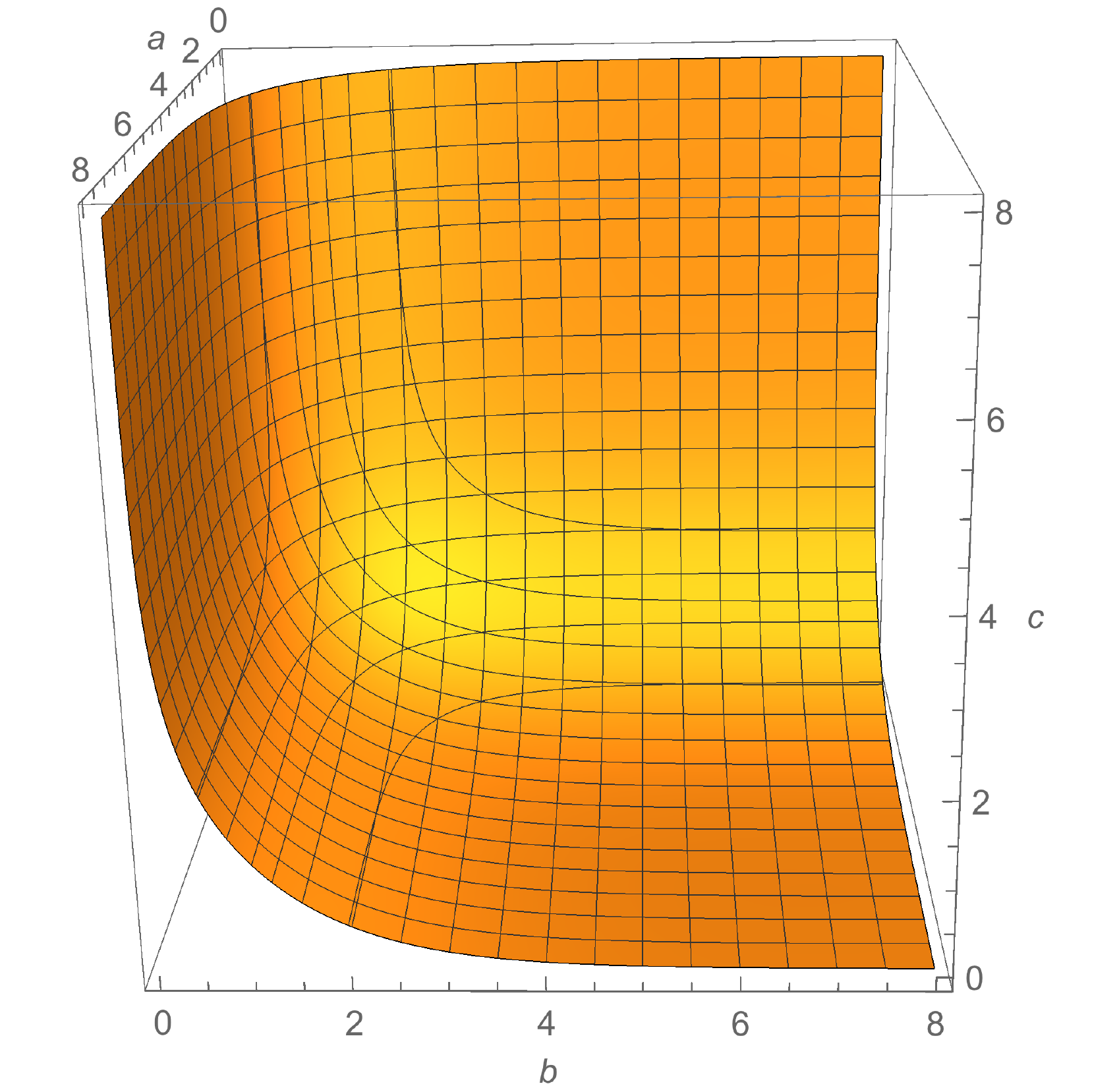}
\caption{The hypersurfaces $\calM^{1}(\calR_{r})$ for the roses with 2 and 3 petals.}\label{fig:moduli space}
\end{figure}

Using Proposition~\ref{prop:Frose}, we can provide a strengthening of Proposition~\ref{prop:oF} for the $r$--rose.

\begin{proposition}\label{prop:oF = 0 iff h = 1}
For any $r \geq 2$, the unit entropy moduli space $\calM^{1}(\calR_{r})
$ equals the level set $\{\ell \in \calM(\calR_r) \mid \oF_{r}(\ell) = 0 \}$.
\end{proposition}

\begin{proof}
By Proposition~\ref{prop:oF}, we have that $\calM^{1}(\calR_{r}) \subseteq \{\ell \in  \calM(\calR_r) \mid \oF_{r}(\ell) = 0 \}$.  Suppose that $\oF_{r}(\ell) = 0$ for some $\ell \in  \calM(\calR_r)$.  Set $h = \entropy_r(\ell)$.  We need to show that $h = 1$.    

Consider the function $p \from \RR_{>0} \to \RR$ defined by: $p(t) = \oF_{r}(t \cdot \ell)$.  We have $p(1) = \oF_{r}(\ell) = 0$.  As $\entropy_r(h \cdot \ell) = 1$, we have $p(h) = \oF_{r}(h \cdot \ell) = 0$ as well by Proposition~\ref{prop:oF}.  

Using the expression of $\oF_r(\ell)$ derived in Proposition~\ref{prop:Frose} we compute that
\begin{equation*}
p'(t) = \sum_{\substack{S \subseteq [r] \\ S \neq  \emptyset}} (2\abs{S} - 1)\ell^{S}\exp(-t \cdot \ell^{S}).
\end{equation*}
Therefore $p'(t) > 0$ for all $t \in \RR_{>0}$.  As $p(h \cdot \ell) = 0 = p(\ell)$, we must have that $h \cdot \ell = \ell$ and hence $h = 1$.
\end{proof}


\subsection{Finite Length Paths in \texorpdfstring{$\calM^1(\calR_r)$}{M\textasciicircum1(R\_r)} for \texorpdfstring{$r \geq 3$}{r >= 3}}\label{subsec:not complete}

Using the computation of $\oF_r$, in Proposition~\ref{prop:finite length path} we will compute the length of the path in $\calM^1(\calR_r)$ starting at $\log(2r-1) \cdot \One$ that blows up the length of one edge while shrinking the lengths of the others at the same rate.  As we will show, when $r$ is at least $3$, this path has finite length and thus the moduli space $(\calM^1(\calR_r),d_{\entropy,\calR_r})$ is not complete for $r \geq 3$.    

Before we begin, it is useful to introduce the following functions $X_{i}, Y_{i} \from \calM(\calR_r) \to \RR$ for each $i \in [r]$:
\begin{align}
X_{i}(\ell) &= \sum_{S \subseteq [r] - \{i\}} (1 - 2\abs{S}) \exp\left(-\ell^{S}\right),\label{eq:Xi} \text{ and} \\
Y_{i}(\ell) &= \sum_{S \subseteq [r] - \{i\}} (1 + 2\abs{S}) \exp\left(-\ell^{S}\right).\label{eq:Yi}
\end{align}
Both $X_{i}$ and $Y_{i}$ are constant with respect to $\ell^{i}$.  Using these functions, we can isolate the terms in $\oF_r(\ell)$ in which $\ell^i$ appears and write 
\begin{equation}\label{eq:F = 0}
\oF_{r}(\ell) = X_{i}(\ell) - \exp(-\ell^{i})Y_{i}(\ell). 
\end{equation}
Hence for $\ell \in \calM^{1}(\calR_{r})$, as $\oF_r(\ell) = 0$ we can solve for $\ell^i$ and write
\begin{equation}\label{eq:ell i}
\ell^{i} = \log\left(\frac{Y_{i}(\ell)}{X_{i}(\ell)}\right).
\end{equation}
Further, we find the following expression for the partial derivative of $\oF_r(\ell)$ with respect to $\ell^i$:
\begin{equation}\label{eq:partial F}
\PD{\oF_{r}}{i}(\ell) = \exp(-\ell^{i})Y_{i}(\ell).
\end{equation} 

We observe the following inequalites for any $\ell \in \calM^{1}(\calR_{r})$.

\begin{lemma}\label{lem:Xi and Yi bounds}
Let $r \geq 2$ and let $\ell \in \calM^1(\calR_r)$.  Then 
\begin{align}
0 & < X_{i}(\ell) < 1,\text{ and}\label{eq:Xi bound} \\ 
1 & < Y_{i}(\ell) < 4\label{eq:Yi bound}  
\end{align} 
\end{lemma}
\begin{proof}
For \eqref{eq:Xi bound}, we first note that $X_i(\ell) = \exp(-\ell^i)Y_i(\ell)$ for all $\ell \in \calM^1(\calR_r)$ by Proposition~\ref{prop:oF} and \eqref{eq:F = 0}.  Since every term in $Y_i(\ell)$ has a positive coefficient, we find that $0 < X_i(\ell)$.  As the term in $X_i(\ell)$ corresponding to $S = \emptyset$ is $1$ and all other terms have negative coefficients we find $X_i(\ell) < 1$.

For \eqref{eq:Yi bound}, we have that the term in $Y_i(\ell)$ corresponding to $S = \emptyset$ is $1$ and all other terms have positive coefficients, thus $1 < Y_i(\ell)$.  The terms in $1 - X_i(\ell)$ and $Y_i(\ell) - 1$ correspond to the nonempty subsets $S \subseteq [r] - \{i\}$.  The coefficient for the term corresponding to $S$ in $1 - X_i(\ell)$ is
\[ \frac{2\abs{S} + 1}{2\abs{S} - 1}  \] times the coefficient for the same term in $Y_i(\ell) - 1$.  As this ratio is bounded by $3$, we find that $Y_i(\ell) - 1 \leq 3(1 - X_i(\ell))$.  Hence, as $1 - X_i(\ell) < 1$ by \eqref{eq:Xi bound}, we have $Y_i(\ell) - 1 < 3$ and so $Y_i(\ell) < 4$.
\end{proof}
  
We record the following calculation.

\begin{lemma}\label{lem:ellr}
Let $r \geq 2$, and let $\ell \in \calM^1(\calR_r)$ be such that $\ell^i = \log(L)$ for $i \in [r-1]$ for some $L > 2r - 3$.  Then
\begin{equation}\label{eq:ellr}
\ell^r = \log\left(\frac{L + (2r-1)}{L - (2r - 3)}\right).
\end{equation}
\end{lemma}

\begin{proof}
For any $S \subseteq [r-1]$ we have $\exp(-\ell^S) = \exp(-\abs{S}\log L) = L^{-\abs{S}}$.  Hence by Lemma~\ref{lem:comb identity} we have that
\begin{align*}
X_r(\ell) &= \sum_{k=0}^{r-1} \genfrac{(}{)}{0pt}{0}{r-1}{k}(1-2k)L^{-k} = L^{-r+1}(L+1)^{r-2}(L - (2r - 3)), \text{ and} \\
Y_r(\ell) &= \sum_{k=0}^{r-1} \genfrac{(}{)}{0pt}{0}{r-1}{k} (1+2k)L^{-k} = L^{-r+1}(L+1)^{r-2}(L + (2r - 1)).
\end{align*}  
Therefore by \eqref{eq:ell i} we find that
\begin{equation*}
\ell^{r} = \log\left(\frac{Y_r(\ell)}{X_r(\ell)}\right) = \log \left(\frac{L + (2r - 1)}{L - (2r - 3)}\right).\qedhere
\end{equation*}
\end{proof}

For any $r \geq 3$, we will construct a path $\ell_t \from [0,1) \to \calM^1(\calR_{r})$ that has finite length and the property that $\ell^r_t \to \infty$ as $t \to 1^-$. 

\begin{proposition}\label{prop:finite length path}
Fix $r \geq 3$ and let $N_{t} = 2(r-t)-1$.  Let $\ell_t \from [0,1) \to \calM^1(\calR_{r})$ be the smooth path where $\ell^i_t = \log(N_{t})$ for $i \in [r-1]$.  Then $\calL_{\entropy,r}(\ell_t|[0,1))$ is finite and $\ell^r_t \to \infty$ as $t \to 1^-$.
\end{proposition}

\begin{proof}
Let $\ell_t \from [0,1) \to \calM^1(\calR_{r})$ be as in the statement.  Using Lemma~\ref{lem:ellr} we find
\begin{equation*}
\ell^{r}_t = \log \left(\frac{2r - 1 -t}{1-t}\right).
\end{equation*}
In particular, we have $\ell^r_t \to \infty$ when $t \to 1^-$ as claimed. 

We first provide a lower bound on $\I{\ell_t,\nabla \oF_r(\ell_t)}$.  This is the denominator of the expression for the entropy norm in Proposition~\ref{prop:o metric formulas}.  Using the expressions of the partial derivatives for $\oF_r(\ell_t)$ in \eqref{eq:partial F}, the fact that $1 < Y_i(\ell_t)$ \eqref{eq:Yi bound} and that $\ell^r_t \exp(-\ell^r_t)Y_r(\ell_t) > 0$, we have that
\begin{equation*}
\I{\ell_t,\nabla \oF_{r}(\ell_t)} = \sum_{i=1}^{r} \ell^{i}_t\exp(-\ell^{i}_t)Y_{i}(\ell_t) > \sum_{i=1}^{r-1} \ell^{i}_t\exp(-\ell^{i}_t) = (r-1)\frac{\log(N_{t})}{N_{t}} .
\end{equation*}
As $\log(N_t) \geq \log(N_1)$ and $N_t \leq N_0$ for all $t \in [0,1]$, we conclude that
\begin{equation}\label{eq:finite length d bound}
\I{\ell_t,\nabla \oF_{r}(\ell_t)} > (r-1)\frac{\log(N_{t})}{N_{t}} \geq (r-1)\frac{\log(N_{1})}{N_{0}} = (r-1)\frac{\log(2r-3)}{2r-1}.
\end{equation}

Next, we provide an upper bound on $\I{\ddot \ell_t, \nabla \oF_r(\ell_t)}$.  This is the numerator of the expression for the entropy norm in Proposition~\ref{prop:o metric formulas}.  To do so, we compute that
\begin{align*}
\ddot \ell^i_{t}  = -\frac{4}{N_{t}^{2}}, \mbox{ for } i \in [r-1], \mbox{ and }
\ddot \ell^r_{t} = -\frac{1}{(2r - 1 -t)^{2}} + \frac{1}{(1-t)^{2}}.
\end{align*}
In particular $\ddot \ell^{i}_t \exp(-\ell^{i}_t)Y_{i}(\ell_t) < 0$ for $i \in [r-1]$ and $\ddot \ell_t^r < \frac{1}{(1-t)^2}$.  Combining these with the expressions for the partial derivatives for $\oF_r(\ell_t)$ in \eqref{eq:partial F} and the fact that $Y_i(\ell_t) < 4$ \eqref{eq:Yi bound}, we have that
\begin{align}
\I{\ddot \ell_t, \nabla \oF_{r}(\ell_t) } &= \sum_{i=1}^{r} \ddot \ell^i_t\exp(-\ell^i_{t})Y_{i}(\ell_t) \notag \\
& < \ddot \ell^r_{t} \exp(-\ell^r_{t})Y_{i}(\ell_t) \notag \\
& < \frac{1}{(1-t)^{2}} \cdot \frac{1-t}{2r - 1 -t} \cdot 4 \leq \frac{2}{r-1} \cdot \frac{1}{1-t}.\label{eq:finite length n bound}
\end{align}
Proposition~\ref{prop:o metric formulas} together with the bounds appearing in \eqref{eq:finite length d bound} and \eqref{eq:finite length n bound} implies that
\begin{equation*}
\norm{(\ell_t,\dot \ell_t)}^2_{\entropy,\calR_r} = \frac{\I{\ddot \ell_t, \nabla \oF_{r}(\ell_t)}}{\I{\ell_t,\nabla \oF_{r}(\ell_t)}} \leq \frac{2(2r-1)}{(r-1)^2\log(2r - 3)} \cdot \frac{1}{1-t}.
\end{equation*}
Therefore, the entropy length of the path $\ell_t \from [0,1) \to \calM^{1}(\calR_{r})$ is finite as claimed.
\end{proof}

As a consequence, we obtain that $\calM^1(\calR_r)$ is not complete when $r \geq 3$.
  
\begin{proposition}\label{prop:moduli rose not complete}
For any $r \geq 3$, the moduli space $(\calM^1(\calR_r),d_{\entropy,\calR_r})$ is not complete.
\end{proposition}

In Section~\ref{sec:not complete}, we will use Proposition~\ref{prop:finite length path} to show that $(\calX^1(\FF_r),d_\entropy)$ is not complete as well when $r \geq 3$.


\subsection{The Diameter of \texorpdfstring{$\calM^1(\calR_r)$}{M\textasciicircum1(R\_r)} Is Infinite}\label{subsec:infinite diameter}

In this section we show that $(\calM^1(\calR_r),d_{\entropy,\calR_r})$ has infinite diameter by showing that any path that shrinks an edge to $0$ has infinite length.  Before we begin, it is useful to introduce the following functions.  For distinct $i, j \in [r]$ we define:
\begin{align}
X_{ij}(\ell) &= \sum_{S \subseteq [r] - \{i,j\}} (1 + 2\abs{S}) \exp\left(-\ell^{S}\right),\label{eq:Xij} \mbox{ and} \\
Y_{ij}(\ell) &= \sum_{S \subseteq [r] - \{i,j\}} (3 + 2\abs{S}) \exp\left(-\ell^{S}\right).\label{eq:Yij}
\end{align}
As in Lemma~\ref{lem:Xi and Yi bounds}, we observe that for any $\ell \in \calM^1(\calR_r)$ we have that
\begin{equation}\label{eq:Yij bound}
3 < Y_{ij}(\ell) < 3Y_i(\ell) < 12.
\end{equation}
Notice that both $X_{ij}$ and $Y_{ij}$ are constant with respect to both $\ell^{i}$ and $\ell^{j}$.  For any distinct $i,j \in [r]$, these functions allow us to write
\begin{equation}\label{eq:write Yi}
Y_{i}(\ell) = X_{ij}(\ell) + \exp(-\ell^j)Y_{ij}(\ell).
\end{equation}
Using \eqref{eq:write Yi} plus the expressions for the partial derivatives for $\oF_r(\ell)$ in \eqref{eq:partial F} we find the following expressions for the second partial derivatives of $\oF_r(\ell)$:
\begin{align}
\PPD{\oF_{r}}{i}{i}(\ell) &= -\exp(-\ell^i)Y_{i}(\ell), \text{ and}\label{eq:partial ii F} \\
\PPD{\oF_{r}}{i}{j}(\ell) &= -\exp(-\ell^i - \ell^j)Y_{ij}(\ell) \text{ for } i \neq j.\label{eq:partial ij F}
\end{align} 

The following technical lemma is the main tool for estimating length.  Intuitively, it says that when one of the edge lengths---$\ell^r$ in the statement---is short, the length of a path is bounded below by the difference in the square roots of the lengths of second shortest edge---$\ell^1$ in the statement---at the endpoints of the path.  In the statement below, shortness of $\ell^r$ is guaranteed by taking $\ell^1$ large enough. 

\begin{lemma}\label{lem:square root bound}
Let $r \geq 2$.  There is an $L_r$ with the following property.  Suppose $\ell_t \from [0,1] \to \calM^{1}(\calR_{r})$ is a piecewise smooth path such that for all $t \in [0,1]$:
\begin{enumerate}
\item $\ell^1_t = \min\{ \ell^i_t \mid i \in [r - 1] \}$, 
\item $\ell^1_0 \geq L_r$, and
\item $\dot \ell^1_t > 0$.
\end{enumerate}
Then
\[ \calL_{\entropy,\calR_r}(\ell_t|[0,1]) \geq \frac{1}{2\sqrt{2}B_{1}} \left(\sqrt{B_{1}\ell^1_1 + B_{2}} - \sqrt{B_{1}\ell^1_0 + B_{2}}\right)
 \]
where $B_1 = 4(r-1)$ and $B_2 = 2^{r+3}(2r-1)$. 
\end{lemma}

\begin{proof}
Let $\ell_t \from [0,1] \to \calM^1(\calR_r)$ be as in the statement.  By (3) we may reparametrize the path so that $\ell^1_t = t$.  Let $L_r$ be large enough so that \[\max\{2^{r}(2r-3)\exp(-L_r), 288r\exp(-L_r)\} \leq 1.\]  

The method of proof is similar to the calculations performed in Section~\ref{sec:rank 2}.  Specifically, using the expression
\begin{equation}
\norm{(\ell_t,\dot \ell_t)}^2_{\entropy,\calR_r} = \frac{-\I{\dot \ell_t, \bH[\oF_{r}(\ell_t)]\dot\ell_t}}{\I{\ell_t,\nabla \oF_r(\ell_t)}}
\end{equation}
from Proposition~\ref{prop:o metric formulas}, we show that the square of the entropy norm along this path is bounded from below by $1/2(B_1 t + B_2)$.  This is done by showing that the denominator is bounded from above by $\exp(-t)(B_1t + B_2)$~\eqref{eq:rose volume upper bound} and that the numerator is bounded from below by $\frac{1}{2}\exp(-t)$~\eqref{eq:rose hessian lower bound}.  

We first provide the upper bound on $\I{\ell_t,\nabla \oF_r(\ell_t)}$.

As $\ell^i_t \geq \ell^1_t = t$ for all $i \in [r-1]$, we have that $\exp(-\ell^S_t) \leq \exp(-t)$ for all nonempty subsets $S \subseteq [r-1]$.  Since $1 - 2\abs{S} \geq -(2r-3)$ for any nonempty subset $S \subseteq [r-1]$, using the definition of $X_i(\ell_t)$~\eqref{eq:Xi} we have that
\begin{equation}\label{eq:Xr bound}
X_r(\ell_t) = \sum_{S \subseteq [r - 1]}(1 - 2\abs{S})\exp(-\ell_t^S) \geq 1 - 2^{r-1}(2r-3)\exp(-t).
\end{equation}
Therefore
\begin{equation}\label{eq:1-Xr bound}
1 - X_r(\ell_t) \leq 2^{r-1}(2r-3)\exp(-t).
\end{equation}
As $t \geq L_r$ we additionally find that
\begin{equation}\label{eq:Xr one-half}
X_r(\ell_t) \geq \frac{1}{2}.
\end{equation}
As $0 < X_r(\ell_t) < 1$~\eqref{eq:Xi bound} and $-\log(1-x) \leq \frac{x}{1-x}$ for all $0 < x < 1$, using~\eqref{eq:1-Xr bound} and~\eqref{eq:Xr one-half} we find that
\[ -\log(X_r(\ell_t)) = -\log(1 - (1 - X_r(\ell_t))) \leq \frac{1-X_r(\ell_t)}{X_r(\ell_t)} \leq 2^r(2r-3)\exp(-t).\] 
Similarly as $1 < Y_r(\ell_t)$~\eqref{eq:Yi bound} and $\log(x) \leq x - 1$ for all $x \geq 1$, using the definition of $Y_i(\ell_t)$~\eqref{eq:Yi} we have that
\[ \log(Y_r(\ell_t)) \leq Y_r(\ell_t) - 1 = \sum_{\substack{S \subseteq [r-1] \\ S \neq \emptyset}} (1 + 2\abs{S})\exp(-\ell^S_t) \leq 2^{r-1}(2r-1)\exp(-t). \]
Thus by \eqref{eq:ell i}, we find
\begin{equation}\label{eq:short edge length}
\ell^r_{t} = \log(Y_{r}(\ell_t)) - \log(X_{r}(\ell_t)) \leq 2^{r+1}(2r-1)\exp(-t). 
\end{equation}  
As $x\exp(-x)$ is decreasing for $x > 1$, we have that $\ell^i_{t}\exp(-\ell^i_{t}) \leq \ell^1_{t}\exp(-\ell^1_t) = t\exp(-t)$ for $i \in [r-1]$.  Using the expressions for the partial derivatives of $\oF_r(\ell_t)$ in~\eqref{eq:partial F} and the fact that $Y_i(\ell_t) < 4$~\eqref{eq:Yi bound}, we have that
\begin{align*}
\I{\ell_t,\nabla \oF_{r}(\ell_t)} &= \sum_{i=1}^{r} \ell^i_{t}\exp(-\ell^i_{t})Y_{i}(\ell_t)\notag \\
& < 4\bigl((r-1)t\exp(-t) + \ell^r_{t}\exp(-\ell^r_{t}) \bigr) \\
& \leq 4\exp(-t)\bigl((r-1)t + 2^{r+1}(2r-1)\bigr).\notag
\end{align*}
As defined above, we have that $B_{1} = 4(r-1)$ and $B_{2} = 2^{r+3}(2r-1)$.  Hence
\begin{equation}\label{eq:rose volume upper bound}
\I{\ell_t,\nabla \oF_{r}(\ell_t)} \leq \exp(-t)(B_{1}t + B_{2}).
\end{equation}

Next we provide a lower bound on $-\I{\dot \ell_t,\bH[\oF_r(\ell_t)]\dot \ell_t}$.  Using the expressions for the second partial derivatives of $\oF_r(\ell_t)$ in~\eqref{eq:partial ii F} and~\eqref{eq:partial ij F}, we have that
\begin{equation}\label{eq:hessian}
-\I{\dot\ell_t,\bH[\oF_{r}(\ell_t)]\dot\ell_t}  = \sum_{i=1}^{r} (\dot \ell^i_{t})^{2}\exp(-\ell^i_{t})Y_{i}(\ell_t) + 
\sum_{i=1}^{r-1} \sum_{j = i+1}^{r} 2\dot \ell^i_{t} \dot \ell^j_{t} \exp(-\ell^i_{t}-\ell^j_{t})Y_{ij}(\ell_t).
\end{equation}
The following claim says that the diagonal terms in $\bH[\oF_r(\ell_t)]$ dominate in the current setting, that is, when one of the edge lengths is small.

\begin{claim}\label{claim:infinite}
$\displaystyle \frac{1}{2}\I{(\dot\ell_t)^{2},\nabla \oF_{r}(\ell_t)} \leq -\I{\dot\ell_t,\bH[\oF_{r}(\ell_t)]\dot\ell_t} \leq \frac{3}{2}\I{(\dot\ell_t)^{2},\nabla \oF_{r}(\ell_t)}.$
\end{claim}

\begin{proof}[Proof of Claim~\ref{claim:infinite}]
We observe that the first summand in~\eqref{eq:hessian} is exactly $\I{(\dot \ell_t)^2,\nabla \oF_r(\ell)}$.  The claim thus proved by showing that the second summand has absolute value bounded above by $\frac{1}{2}\I{(\dot \ell_t)^2,\nabla \oF_r(\ell_t)}$.  We accomplish this by breaking this summand into various pieces.  

To begin, we focus on the terms in this summand where $j = r$.  

Let $K_{r} \subseteq [r - 1]$ be the set of indices where $\abs{2\dot \ell^i_{t} \exp(-\ell^i_{t})Y_{ir}(\ell_t)} \leq \frac{1}{2r}\abs{\dot \ell^r_{t}Y_{r}(\ell_t)}$.  Summing over the elements in $K_r$ we find that
\begin{equation}\label{eq:sum over Kr}
\abs{\sum_{i \in K_{r}} 2\dot \ell^i_{t} \dot \ell^r_{t}\exp(-\ell^i_{t} - \ell^r_{t})Y_{ir}(\ell_t)} \leq \frac{1}{2}(\dot \ell^r_{t})^{2}\exp(-\ell^r_{t})Y_r(\ell_t).
\end{equation}
From the definition of $L_r$ we have $24r\exp(-L_r) \leq 1/12$.  Thus if $i < r$ and $i \notin K_r$ as $\ell^i_t \geq L_r$ and $Y_{ir}(\ell_t) < 3\max\{Y_i(\ell_t), Y_r(\ell_t)\}$~\eqref{eq:Yij bound} we have that
\begin{equation}
2\abs{\dot \ell^r_{t}Y_{ir}(\ell_t)} \leq 6\abs{\dot \ell^r_{t}Y_r(\ell_t)} \leq \abs{24r \dot \ell^i_{t} \exp(-\ell^i_{t})Y_{ir}(\ell_t)} \leq \frac{1}{12}\abs{\dot \ell^i_t Y_{ir}(\ell_t)} \leq \frac{1}{4}\abs{\dot \ell^i_{t}Y_i(\ell_t)}.
\end{equation}
Thus for $i < r$ and $i \notin K_r$ we have that
\begin{equation}\label{eq:not in Kr}
\abs{2\dot \ell^i_{t} \dot \ell^r_{t}\exp(-\ell^i_{t} - \ell^r_{t})Y_{ir}(\ell_t)} \leq \abs{2\dot \ell_t^i \dot\ell_t^r \exp(-\ell_t^i)Y_{ir}(\ell_t)} \leq \frac{1}{4}(\dot \ell^i_{t})^{2}\exp(-\ell^i_{t})Y_{i}(\ell_t).
\end{equation}

Next we turn our attention to the terms where $j < r$.

For $i \in [r - 1]$ we let $K_{i} \subseteq [r - 1]$ be the set of indices where $\abs{\dot \ell^i_{t}} > \abs{\dot \ell^j_{t}}$ or where $\abs{\dot \ell^i_{t}} = \abs{\dot \ell^j_{t}}$ and $j > i$.  We observe that for any distinct pair of indices $i,j \in [r-1]$ either $j \in K_i$ and $i \notin K_j$ or $i \in K_j$ and $j \notin K_i$.  From the definition of $L_r$ we have $2\exp(-L_r) \leq 1/12r$.  Hence as $Y_{ij}(\ell_t) < 3Y_i(\ell_t)$~\eqref{eq:Yij bound}, we find that $2\exp(-\ell^j_{t})Y_{ij}(\ell_t) \leq \frac{1}{4r}Y_{i}(\ell_t)$ for $j \in [r-1]$.  Therefore, summing over the indices in $K_i$ we find that
\begin{equation}
\label{eq:sum over Ki}
\abs{\sum_{j \in K_{i}} 2\dot \ell^i_{t} \dot \ell^j_{t}\exp(-\ell^i_{t} - \ell^j_{t})Y_{ij}(\ell_t) } \leq \frac{1}{4}(\dot \ell^i_{t})^{2}\exp(-\ell^i_{t})Y_{i}(\ell_t).
\end{equation}

Rearranging the terms and using~\eqref{eq:sum over Kr},~\eqref{eq:not in Kr} and~\eqref{eq:sum over Ki} we find that
\begin{align*}
\abs{\sum_{i=1}^{r-1}\sum_{j=i+1}^{r}2\dot \ell^i_{t} \dot \ell^j_{t} \exp(-\ell^i_{t} - \ell^j_{t})Y_{ij}(\ell_t)} \leq & \, \abs{\sum_{i=1}^{r-1} \sum_{j \in K_{i}}2\dot \ell^i_{t} \dot \ell^j_{t} \exp(-\ell^i_{t} - \ell^j_{t})Y_{ij}(\ell_t)} \\
& \, + \abs{\sum_{i \in K_{r}}2\dot \ell^i_{t} \dot \ell^r_{t} \exp(-\ell^i_{t} - \ell^r_{t})Y_{ir}(\ell_t)} \\
& \, + \abs{\sum_{i \notin K_{r}}2\dot \ell^i_{t} \dot \ell^r_{t} \exp(-\ell^i_{t} - \ell^r_{t})Y_{ir}(\ell_t)} \\
\leq & \, \sum_{i=1}^{r-1}\frac{1}{4}(\dot \ell^i_{t})^{2}\exp(-\ell^i_{t})Y_{i}(\ell_t) \\
& + \frac{1}{2}(\dot \ell^r_{t})^{2}\exp(-\ell^r_{t})Y_{r}(\ell_t) \\
& + \sum_{i \notin K_r} \frac{1}{4}(\dot \ell^i_{t})^{2}\exp(-\ell^i_{t})Y_{i}(\ell_t)  \\
&\leq \frac{1}{2}\I{(\dot\ell_t)^{2},\nabla \oF_{r}(\ell_t)}.
\end{align*}
As explained above, the claim now follows.
\end{proof}

Thus applying Claim~\ref{claim:infinite} and by focusing on the term in $\I{(\dot \ell_t)^2,\nabla \oF_r(\ell_t)}$ corresponding to $\ell^1_t$ and tossing out the rest---which are all nonnegative---we get our desired bound:
\begin{equation}\label{eq:rose hessian lower bound}
-\I{\dot\ell_t,\bH[\oF_{r}(\ell_t)]\dot\ell_t} \geq \frac{1}{2}\I{(\dot\ell_t)^{2},\nabla \oF_{r}(\ell_t)} \geq \frac{1}{2} (\dot \ell^1_{t})^{2}\exp(-\ell^1_t)Y_1(\ell_t) \geq \frac{1}{2}\exp(-t).
\end{equation}  
For the last inequality, recall that $1 < Y_1(\ell_t)$~\eqref{eq:Yi bound}.  Combining~\eqref{eq:rose hessian lower bound} with our previous bound on $\I{\ell_t,\nabla \oF_r(\ell_t)}$ we obtained in \eqref{eq:rose volume upper bound}, we see that
\begin{equation*}
\norm{(\ell_t,\dot\ell_t)}_{\entropy,\calR_r}^{2} = \frac{-\I{\dot\ell_t,\bH[\oF_{r}(\ell_t)]\dot\ell_t}}{\I{\ell_t,\nabla \oF_r(\ell_t)}} \geq \frac{1}{2(B_{1}t + B_{2})}.
\end{equation*}
Hence the length of this path in the entropy metric is at least
\begin{align*}
\int_{\ell^1_0}^{\ell^1_1} \sqrt{\frac{1}{2(B_{1}t + B_{2})}} \, dt &= \frac{1}{2\sqrt{2}B_{1}} \left(\sqrt{B_{1}\ell^1_1 + B_{2}} - \sqrt{B_{1}\ell^1_0 + B_{2}}\right).\qedhere
\end{align*}
\end{proof}

Before we can apply Lemma~\ref{lem:square root bound} to show that $(\calM^1(\calR_r),d_{\entropy,\calR_r})$ has infinite diameter, we require two more estimates.  The first states that  for a length function in $\calM^1(\calR_r)$ when $\ell^r$ is bounded from below, there is an upper bound on the length of the shortest edge that is not $r$.

\begin{lemma}\label{lem:4r+5}
Let $r \geq 2$.  If $\ell \in \calM^1(\calR_r)$ where $\ell^r \geq \log(3)$, then $\min\{\ell^i \mid i \in [r-1]\} \leq \log(4r - 5)$.
\end{lemma}

\begin{proof}
We first prove the lemma under the additional assumption that $\ell^i = \ell^1$ for any $i \in [r-1]$.  In this case, we have that $\ell^i = \log(L)$ for $i \in [r - 1]$ and some $L > 2r - 3$.  By Lemma~\ref{lem:ellr} we have 
\[ \log(3) \leq \ell^r = \log\left(\frac{L + (2r - 1)}{L - (2r - 3)}\right). \]
Hence we have that $3(L - (2r - 3)) \leq L + (2r-1)$, which implies that $L \leq 4r - 5$.  

Next we prove the general case.  Let $\ell \in \calM^1(\calR_r)$ be such that $\ell^r \geq \log(3)$.  Without loss of generality, we assume that $\ell^1 = \min\{\ell^i \mid i \in [r-1]\}$.  

If $\ell^1 \leq \log(2r - 3)$, then we are done.  

Otherwise, we may decrease the lengths $\ell^2,\ldots,\ell^{r-1}$ to be equal to $\ell^1$ while increasing $\ell^r$ to maintain the fact that the metric has unit entropy.  The assumption that $\ell^1 > \log(2r-3)$ ensures that $\ell^r$ is finite.  Denote the resulting metric by $\hat\ell$.  Observe that $\hat\ell^r \geq \ell^r \geq \log(3)$.  By the special case considered above, $\hat\ell^i \leq \log(4r - 5)$ for each $i \in [r-1]$.  As $\ell^1 = \hat\ell^1$, this completes the proof of the lemma. 
\end{proof}

The second estimate shows that when the length of an edge is small for a length function in $\calM^1(\calR_r)$, the lengths of the other edges must be very large.

\begin{lemma}\label{lem:definite length}
Let $r \geq 2$ and let $\ell \in \calM^1(\calR_r)$.  For any $\epsilon > 0$, if $\ell^i \leq \epsilon$ for some $i \in [r]$, then for any $j \in [r] - \{i\}$ we have $\ell^j > -\log \left(\exp(\epsilon) - 1\right)$.
\end{lemma}

\begin{proof}
The subrose consisting of the edges $i$ and $j$ has entropy less than or equal to $1$ (strictly less than $1$ when $r \geq 3$).  By Lemma~\ref{lem:ellr}, this implies that
\begin{equation*} 
\ell^j \geq \log \left(\frac{\exp(\ell^i) + 3}{\exp(\ell^i) - 1}\right) > -\log\left(\exp(\ell^i) - 1 \right) \geq -\log \left(\exp(\epsilon) - 1\right).\qedhere
\end{equation*} 
\end{proof}

Now we can prove the main inequality in this section that shows that any path that shrinks the length of an edge to zero must have infinite length.

\begin{proposition}\label{prop:infinite length}
Let $r \geq 2$.  For any $D > 0$, there is an $\epsilon > 0$ such that for any $\ell \in \calM^1(\calR_r)$ with $\min\{\ell^i \mid i \in [r]\} \leq \epsilon$ we have $d_{\entropy,\calR_r}(\log(2r-1) \cdot \One,\ell) \geq D$.  
\end{proposition}

\begin{proof}
Let $L_0 = \max\{\log(4r - 5), L_r\}$ where $L_r$ is the constant from Lemma~\ref{lem:square root bound}.  Fix an $\epsilon > 0$ such that
\begin{equation*}
\frac{1}{2\sqrt{2}B_1}\left(\sqrt{-B_1 \log(\exp(\epsilon)-1) + B_2} - \sqrt{B_1 L_0 + B_2}\right) \geq D.
\end{equation*}
Since $-\log(x - 1) \to \infty$ as $x \to 1^+$, such an $\epsilon$ exists.  Observe that for this $\epsilon$ we have that $-\log(\exp(\epsilon) - 1) \geq L_0$.

Let $\ell \in \calM^1(\calR_r)$ be such that $\min\{\ell^i \mid i \in [r]\} \leq \epsilon$ and let $\ell_t \from [0,1] \to \calM^1(\calR_r)$ be a piecewise smooth path where $\ell_0 = \log(2r-1) \cdot \One$ and $\ell_1 = \ell$.  We will show that the entropy length of this path is at least $D$.  As the path is arbitrary, this shows that $d_{\entropy,\calR_r}(\log(2r-1) \cdot \One,\ell) \geq D$ as desired.  

Without loss of generality, assume that $\ell^r = \min\{\ell^i \mid i \in [r]\}$.  Let $\delta_0 \in [0,1]$ be the minimal value so that $\ell^r_t = \min\{\ell^i_t \mid i \in [r], t \in [\delta_0,1]\}$.  We observe that $\ell^r_{\delta_0} \geq \log(3)$.  Indeed, there is another index $i \in [r-1]$ such that $\ell^i_{\delta_0} = \ell^r_{\delta_0}$.  If $\ell^r_{\delta_0} < \log(3)$, then the entropy of the subgraph consisting of the $i$th and $r$th edges is greater than $1$.  This contradicts the fact that the entropy of $\ell_{\delta_0}$ is equal to $1$.

As there is an automorphism of the $r$--rose permuting any two edges, by redefining $\ell_t$ if necessary, we may assume that $\ell^1_t = \min\{\ell^i_t \mid i \in [r - 1]\}$ for each $t \in [\delta_0,1]$.  By Lemma~\ref{lem:4r+5}, as $\ell^r_{\delta_0} \geq \log(3)$, we have that $\ell^1_{\delta_0} \leq \log(4r-5) \leq L_0$.  By Lemma~\ref{lem:definite length}, as $\ell^r_1 \leq \epsilon$, we have that $\ell^1_1 > -\log(\exp(\epsilon) - 1) \geq L_0$.  

Let $\delta_1 \in [\delta_0,1]$ be the minimal value so that $\ell^1_t \geq L_0$ for all $t \in [\delta_1,1]$.  As $\calL_{\entropy,\calR_r}(\ell_t|[0,1]) \geq \calL_{\entropy,\calR_r}(\ell_t|[\delta_1,1])$, it suffices to show that the latter is bounded below by $D$.

By only considering the portion of $\ell_t$ along the subintervals of $[\delta_1,1]$ with $\dot \ell^1_t > 0$, we find by Lemma~\ref{lem:square root bound} that
\begin{equation*} 
\calL_{\entropy,\calR_r}(\ell_t|[\delta_1,1]) \geq \frac{1}{2\sqrt{2}B_1}\left(\sqrt{B_1\ell^1_1 + B_2} - \sqrt{B_1\ell^1_{\delta_1} + B_2}\right).
\end{equation*}
As $\ell^r_1 \leq \epsilon$, we have $\ell^1_1 \geq -\log(\exp(\epsilon) - 1)$ by Lemma~\ref{lem:definite length}.  By definition $\ell^1_{\delta_1} = L_0$. Therefore
\begin{equation*} 
\calL_{\entropy,\calR_r}(\ell_t|[\delta_1,1]) \geq \frac{1}{2\sqrt{2}B_1}\left(\sqrt{-B_1\log(\exp(\epsilon)-1) + B_2} - \sqrt{B_1 L_0 + B_2}\right) \geq D.\qedhere 
\end{equation*}
\end{proof}


\section{Proof of Theorem~\ref{th:entropy metric not complete}}\label{sec:not complete}

In this section we give the proof of the first main result of this paper.  Theorem~\ref{th:entropy metric not complete} states that $(\calX^1(\FF_r),d_\entropy)$ is complete when $r = 2$ and not complete if $r \geq 3$.  

\begin{proof}[Proof of Theorem~\ref{th:entropy metric not complete}]
In Section~\ref{sec:rank 2}, we showed that $(\calX^1(\FF_2),d_\entropy)$ is complete (Proposition~\ref{prop:rank 2 complete}), and so it remains to show that $(\calX^1(\FF_r),d_\entropy)$ is not complete when $r \geq 3$.  This is a simple consequence of Proposition~\ref{prop:finite length path} as we explain now.

Let $r \geq 3$ and let $\ell_t \from [0,1) \to \calM^1(\calR_r)$ be the path described in Proposition~\ref{prop:finite length path}.  Using the natural homeomorphism $\calM^1(\calR_r) \leftrightarrow \calX^1(\calR_r,{\rm id})$, we can consider $\ell_t$ as a path in $\calX^1(\FF_r)$.

The sequence of length functions $(\ell_{1-1/n})_{n \in \NN}$ is a Cauchy sequence in $(\calX^1(\FF_r),d_\entropy)$ as the entropy distance on $\calM^1(\calR_r)$ is an upper bound on the entropy distance on $\calX^1(\FF_r)$.   

We claim that this sequence does not have a limit.  Indeed, any length function $\ell$ that does not lie in $\calX^1(\calR_r,{\rm id})$ has an open neighborhood in the weak topology that does not intersect $\calX^1(\calR_r,{\rm id})$.  As the metric topology and the weak topology agree, any possible limit of this sequence must lie in $\calX^1(\calR_r,{\rm id})$.  However, as $\ell_{1-1/n}^r \to \infty$ as $n \to \infty$, we see that for any $\ell \in \calX^1(\calR_r,{\rm id})$, there is an open neighborhood of $\ell$ in the weak topology $U \subset \calX^1(\FF_r)$ such that $\ell_{1-1/n} \in U$ for only finitely many $n$.  Hence, again as the metric topology and the weak topology agree, we see that this sequence does not have a limit in $\calX^1(\calR_r,{\rm id})$.
\end{proof}


\section{The Completion of \texorpdfstring{$(\calM^1(\calR_r),d_{\entropy,\calR_r})$}{(M\textasciicircum1(R\_r),d\_{h,R\_r})}}\label{sec:completion of rose}

The goal of this section is to prove Theorem~\ref{th:completion rose} that states that the completion of $(\calM^1(\calR_r),d_{\entropy,\calR_r})$ is homeomorphic to the complement of the vertices in an $(r-1)$--simplex.  Intuitively, the newly added completion points correspond to unit entropy metrics on the subroses of $\calR_r$ consisting of at least two edges.  Specifically, a face of dimension $k-1$ corresponds to unit entropy metrics on a sub--$k$--rose.  We observe that $\calR_1$ does not possess a metric with unit entropy.  This accounts for the missing vertices in the completion.    

There are two steps to the proof.  First, in Section~\ref{subsec:model} we introduce a model space $\hcalM^1(\calR_r)$ for the completion of $\calM^1(\calR_r)$ with respect to the entropy metric.  This model considers $\calM(\calR_r)$ as a subset of $[0,\infty]^r$ and adds the faces where at most $r-2$ of the coordinates are equal to $\infty$.  It is apparent from the construction that $\hcalM^1(\calR_r)$ is homeomorphic to the complement of the vertices in an $(r-1)$--simplex.  Proposition~\ref{prop:metric extends} shows that the distance function $d_{\entropy,\calR_r}$ on $\calM^1(\calR_r)$ extends to a distance function on $\hcalM^1(\calR_r)$.  It is clear that $\calM^1(\calR_r)$ is dense in $\hcalM^1(\calR_r)$.  In Section~\ref{subsec:proof rose} we complete the proof of Theorem~\ref{th:completion rose} by showing that $(\hcalM^1(\calR_r),d_{\entropy,\calR_r})$ is complete.  Example~\ref{ex:completion 2,3} illustrates $\hcalM^1(\calR_3)$ and contrasts this with the closure of $CV(\calR_3,{\rm id})$ considered as a subset of $\RR^{\FF_3}$ in the axis topology. 

Finally, in Section~\ref{subsec:rose-thin} we show that the cross-section of $\calM^1(\calR_r)$ consisting of length functions with one short edge goes to zero as the length of the short edge goes to zero.  This is important for Section~\ref{sec:rose bounded} where we show that $\calX^1(\calR_r,{\rm id})$ is a bounded subset of $(\calX^1(\FF_r),d_\entropy)$. 

\subsection{The Model Space \texorpdfstring{$\hcalM^1(\calR_r)$}{hat{M}\textasciicircum1(R\_r)}}\label{subsec:model}

In this section we introduce a model for the completion of $(\calM^1(\calR_r),d_{\entropy,\calR_r})$.  As mentioned above, we add the faces to $\calM^1(\calR_r)$ corresponding to subroses consisting of at least two edges where the rest of the edge lengths are infinite.

Topologize $[0,\infty]$ as a closed interval.  The natural inclusion $\calM^{1}(\calR_{r}) \subset (0,\infty)^{r} \subset [0,\infty]^{r}$ is an embedding. By setting $x + \infty = \infty$ and $\exp{(-\infty)} = 0$ we get that the functions $\oF_{r}$~\eqref{eq:Frose}, $X_{i}$~\eqref{eq:X}, $Y_{i}$~\eqref{eq:Yi}, $X_{ij}$~\eqref{eq:Xij} and $Y_{ij}$~\eqref{eq:Yij} from Section~\ref{sec:rose} extend to continuous functions on $[0,\infty]^{r}$, and the entropy function $\entropy_r(\ell)$ extends to $(0,\infty]^{r}$. 

Given a subset $S \subseteq [r]$, we identify the following subsets:
\begin{align*}
\calM_{S} &= \{ \ell \in (0,\infty]^{r} \mid \ell^i < \infty \text{ if } i \in S \text{ and } \ell^i = \infty \text{ if } i \notin S \}, \mbox{ and} \\
\calM^{1}_{S} &= \{ \ell \in \calM_{S} \mid \oF_{r}(\ell) = 0 \}.
\end{align*}
Notice that $\calM_{[r]} = \calM(\calR_{r})$ and that $\calM_{S} \cap \calM_{S'} = \emptyset$ if $S \neq S'$. We further observe that $\calM^1_\emptyset = \emptyset$ and $\calM^1_{\{i\}} = \emptyset$ for any $i \in [r]$.  For the latter, note that $X_i(\ell) = 1$ and $Y_i(\ell) = 1$ for any $\ell \in \calM_{\{i\}}$.  Thus for $\ell \in \calM_{\{i\}}$ we have that $\oF_r(\ell) = 1 - \exp(-\ell^i) > 0$.   

Fix $S \subseteq [r]$ with $\abs{S} > 1$ and let $\iota_{S} \from S \to \{1,2,\ldots, \abs{S}\}$ be the order preserving bijection and let $\varepsilon_{S} \from [0,\infty]^{\abs{S}} \to [0,\infty]^{r}$ be the embedding defined by 
\begin{equation*}\label{eq:embedding}
\varepsilon_{S}(\ell^{1},\ldots,\ell^{\abs{S}})^i = \begin{cases}
\ell^{\iota_{S}(i)} & \mbox{ if } i \in S, \\
\infty & \mbox{ else}.
\end{cases}
\end{equation*}
The function $\varepsilon_S$ allows us to consider a length function on $\calR_{\abs{S}}$ as a length function on $\calR_r$ where the edges not in $S$ have infinite length.  Indeed, with these definitions we have $\varepsilon_{S}(\calM(\calR_{\abs{S}})) = \calM_{S}$.  The following is immediate from Proposition~\ref{prop:Frose} and the definitions.

\begin{lemma}\label{lem:inclusion}
Let $r \geq 2$ and let $S \subseteq [r]$ have $\abs{S} > 1$.  Then the following are true. 
\begin{enumerate}
\item $\oF_{r} \circ \varepsilon_{S} = \oF_{\abs{S}}$.  
\item For $\ell \in \calM(\calR_{\abs{S}})$, we have $\entropy_r(\varepsilon_{S}(\ell)) = \entropy_{\abs{S}}(\ell)$.
\item $\varepsilon_{S}$ restricts to a homeomorphism $\calM^{1}(\calR_{\abs{S}}) \to \calM^{1}_{S}$.
\end{enumerate}
\end{lemma}

Next we define the sets
\begin{align*}
\hcalM(\calR_{r}) &= (0,\infty]^r = \bigcup_{S \subseteq [r]} \calM_{S}, \mbox{ and} \\
\hcalM^{1}(\calR_{r}) &= \{ \ell \in \hcalM(\calR_{r}) \mid \oF_{r}(\ell) = 0 \} = \bigcup_{S \subseteq [r]} \calM^1_S.
\end{align*}
The set $\hcalM^{1}(\calR_{r})$ is homeomorphic to the complement of the set of vertices in an $(r-1)$--simplex.  Applying Proposition~\ref{prop:oF = 0 iff h = 1}, Lemma~\ref{lem:inclusion}, and the above definitions we get the following.

\begin{proposition}\label{prop:oF = 0 iff h = 1 extended}
Let $r \geq 2$.  A length function $\ell \in \hcalM(\calR_r)$ lies in $\hcalM^{1}(\calR_{r})$ if and only if it has entropy equal to $1$.  \end{proposition}

As $Y_i(\ell)$ and $Y_{ij}(\ell)$ extend to continuous functions on $\hcalM(\calR_r)$, using the expressions for the partial derivatives of $\oF_r(\ell)$ in~\eqref{eq:partial F},~\eqref{eq:partial ii F} and~\eqref{eq:partial ij F}, we see that $\PD{\oF_r}{i}(\ell)$, $\PPD{\oF_r}{i}{i}(\ell)$ and $\PPD{\oF_r}{i}{j}(\ell)$ extend to continuous functions on $\hcalM(\calR_r)$.   Using these formulas, we see that extensions satisfy the following properties.

\begin{lemma}\label{lem:partials inclusion}
Let $r \geq 2$, let $S \subseteq [r]$ have $\abs{S} > 1$ and fix $\ell \in \calM(\calR_{\abs{S}})$.  Then for $i,j \in [r]$ the following hold.
\begin{align*}
\PD{\oF_r}{i}(\varepsilon_S(\ell)) &= \begin{cases}
\PD{\oF_{\abs{S}}}{\iota_S(i)}(\ell) & \text{ if } i \in S \\
0 & \text{ else }
\end{cases}\label{eq:PDF inclusion} \\
\PPD{\oF_r}{i}{j}(\varepsilon_S(\ell)) &= \begin{cases}
\PPD{\oF_{\abs{S}}}{\iota_S(i)}{\iota_S(j)}(\ell) & \text{ if } i, j \in S  \\
0 & \text{ else}
\end{cases}
\end{align*}
\end{lemma}

Hence both $\nabla \oF_{r}(\ell)$ and $\bH[\oF_{r}(\ell)]$ are well-defined for $\ell \in \hcalM(\calR_r)$.  Additionally, the inner product $\I{\ell,\nabla \oF_r(\ell)}$ also extends continuously as we next show.

\begin{lemma}\label{lem:total volume inclusion}
Let $r \geq 2$.  The function $\ell \mapsto \I{\ell,\nabla \oF_{r}(\ell)}$ has a continuous extension to $\hcalM^{1}(\calR_{r})$.  Moreover, if $S \subseteq [r]$ has $\abs{S} > 1$ and $\ell \in \calM_{\abs{S}}$, then 
\[ \I{\varepsilon_{S}(\ell),\nabla \oF_{r}(\varepsilon_{S}(\ell))} = \I{\ell,\nabla \oF_{\abs{S}}(\ell)} .\]
\end{lemma}

\begin{proof}
Both of these statements follow from the expressions for the partial derivatives of $\oF_r(\ell)$ in~\eqref{eq:partial F} as $x\exp(-x) \to 0$ when $x \to \infty$.  
\end{proof}

We define the \emph{tangent bundle} $T\hcalM^{1}(\calR_{r})$ as the subspace of $(\ell,\bv) \in \hcalM^1(\calR_r) \times \RR^{r}$ where $\I{\bv,\nabla \oF_{r}(\ell)} = 0$.  The subset of $(\ell,\bv) \in T\hcalM^{1}(\calR_{r})$ where $\ell \in \calM^{1}_{S}$ is denoted $T\calM^{1}_{S}$.  For consistency, we denote $T\calM_{[r]}^{1}$ by $T\calM^{1}(\calR_{r})$.  There is an embedding $\boldsymbol{\varepsilon_S} \from \RR^{\abs{S}} \to \RR^r$ defined by  
\begin{equation*}\label{eq:vector embedding}
\boldsymbol{\varepsilon_S}(\bv^{1},\ldots,\bv^{\abs{S}})^i = \begin{cases}
\bv^{\iota_{S}(i)} & \mbox{ if } i \in S, \\
0 & \mbox{ else}.
\end{cases}
\end{equation*}
This allows us to define an embedding $T_{S} \from T\calM^{1}(\calR_{\abs{S}}) \to T\hcalM^{1}(\calR_{r})$ whose image is contained in $T\calM^1_S$ by $T_S(\ell,\bv) = (\varepsilon_S(\ell),\boldsymbol{\varepsilon_S}(\bv))$.  Proposition~\ref{prop:o metric formulas} together with Lemma~\ref{lem:partials inclusion} and Lemma~\ref{lem:total volume inclusion} imply the following.   

\begin{proposition}\label{prop:norm extends}
Let $r \ge 2$.  The entropy norm $\norm{\param}_{\entropy,\calR_r} \from T\calM^{1}(\calR_{r}) \to \RR$ extends to a continuous semi-norm $\norm{\param}_{\entropy,\calR_r} \from T\hcalM^{1}(\calR_{r}) \to \RR$.  Moreover, the embedding maps $T_S \from T\calM^1(\calR_{\abs{S}}) \to T\hcalM^1(\calR_r)$ are norm-preserving.  Specifically, if $S \subseteq [r]$ has $\abs{S} > 1$, and $(\ell,\bv) = T_{S}(\ell_{S},\bv_S)$, then the extension satisfiies 
\begin{equation}\label{eq:norm extends}
\norm{(\ell,\bv)}_{\entropy,\calR_r} = \norm{(\ell_{S},\bv_{S})}_{\entropy,\calR_{\abs{S}}}
\end{equation}
\end{proposition}

The failure for the extension to be a norm is the following reason.  If $(\ell,\bv) \in \calM^1_S \times \RR^r$ and $\bv^i = 0$ for $i \in S$, then $(\ell,\bv) \in T\calM^1_S$ as $\PD{\oF_r(\ell)}{i} = 0$ whenever $i \notin S$.  Further, $\PPD{\oF_r}{i}{j}(\ell) = 0$ whenever $i \notin S$ or $j \notin S$. Thus $-\I{\bv,\bH[\oF_r](\ell)\bv} = 0$ and hence $\norm{(\ell,\bv)}_{\entropy,\calR_r} = 0$ for such points.

The following is the main result of this section.  It shows that $\hcalM^1(\calR_r)$ is contained in the completion of $(\calM^1(\calR_r),d_{\entropy,\calR_{r}})$.

\begin{proposition}\label{prop:metric extends}
Let $r \geq 2$.  The following statements hold.
\begin{enumerate}
\item The extension of the entropy norm defines a distance function $d_{\entropy,\calR_r}$ on $\hcalM^{1}(\calR_{r})$.\label{item:rose distance extends}
\item The inclusion $(\calM^{1}(\calR_{r}),d_{\entropy,\calR_r}) \to (\hcalM^{1}(\calR_{r}),d_{\entropy,\calR_r})$ is an isometric embedding.\label{item:rose iso embedding}
\item The topology induced by $d_{\entropy,\calR_r}$ equals the subspace topology $\hcalM^{1}(\calR_{r}) \subseteq [0,\infty]^{r}$.\label{item:rose topology}
\end{enumerate}
\end{proposition}

\begin{proof}
By definition $\hcalM^1(\calR_{2}) = \calM^1(\calR_{2})$.  Hence the proposition is obvious for $r = 2$.  We assume from now on that $r \geq 3$.  

First we need to show that for each $\ell,\ell' \in \hcalM^1(\calR_r)$ there is a path $\ell_t \from [0,1] \to \hcalM^1(\calR_r)$ with $\ell_0 = \ell$ and $\ell_1 = \ell'$ that has finite length so that $d_{\entropy,\calR_r}(\ell,\ell')$ is defined.  Notice that by Proposition~\ref{prop:norm extends} for each $S \subseteq [r]$ with $\abs{S} > 1$, any two points in $\calM^1_S$ can be joined by a path of finite length.  Hence it suffices to show that for any $S \subseteq [r]$ with $\abs{S} > 1$, that there is a path of finite length joining $\log(2r-1) \cdot \One \in \calM^1(\calR_r)$ to $\varepsilon_S(\log(2\abs{S} - 1) \cdot \One) \in \calM^1_S$.   

Recall that in Proposition~\ref{prop:finite length path}, we showed that there is a path $\ell_t \from [0,1] \to \hcalM^1(\calR_{r})$ where $\ell_0 = \log(2r-1) \cdot \One$ and $\ell_1 = \varepsilon_{[r-1]}(\log(2r-3) \cdot \One) \in \calM_{[r-1]}^1$ that has finite length.  

Now, given $S \subseteq [r]$ with $\abs{S} > 1$ we inductively define subsets $S_i$ for $i = 0,\ldots, r-\abs{S}$ by $S_0 = [r]$, and $S_{i+1} = S_i - \{\max \left(S_i - S\right)\}$ so that $S_{r - \abs{S}} = S$.  The calculation in Proposition~\ref{prop:finite length path} shows that there is a finite length path $(\ell_i)_t \from [0,1] \to \hcalM^1(\calR_r)$ with $(\ell_i)_0 = \varepsilon_{S_i}(\log(2(r-i) - 1) \cdot \One) \in \calM^1_{S_i}$ and $(\ell_i)_1 = \varepsilon_{S_{i+1}}(\log(2(r - (i+1)) - 1) \cdot \One) \in \calM^1_{S_{i+1}}$.  The concatenation of these paths is a path with finite length from $\log(2r - 1)\cdot \One \in \calM^1(\calR_r)$ to $\varepsilon_S(\log(2\abs{S} - 1) \cdot \One) \in \calM^1_S$.  Therefore $d_{\entropy,\calR_r}(\ell,\ell')$ is defined for all $\ell,\ell' \in \hcalM^1(\calR_r)$.

Next, we show that $d_{\entropy,\calR_r}$ defines a distance function on $\hcalM^1(\calR_r)$.  Symmetry and the triangle inequality obviously hold.  What needs to be checked is that if $\ell$ and $\ell'$ are distinct points in $\hcalM^1(\calR_r)$, then there is an $\epsilon$ such that any path from $\ell$ to $\ell'$ has length at least $\epsilon$.  

This is clear if at least one of $\ell$ and $\ell'$ lie in $\calM^1(\calR_r)$.  Indeed, say $\ell$ lies in $\calM^1(\calR_r)$.  Then there is an open set $U \subset \calM^1(\calR_r)$ containing $\ell$ with compact closure $\overline{U}$ such that $\ell' \notin \overline{U}$.  Further, there is an $\epsilon > 0$ such that if $d_{\entropy,\calR_r}(\ell,\ell'') < \epsilon$ then $\ell'' \in U$.  Hence any path from $\ell$ to $\ell'$ must have length at least $\epsilon$ and therefore $d_{\entropy,\calR_r}(\ell,\ell') \geq \epsilon > 0$.

It remains to consider the case where neither $\ell$ nor $\ell'$ lie in $\calM^1(\calR_r)$.  Suppose that there is a sequence of paths $(\ell_n)_t \from [0,1] \to \hcalM^1(\calR_r)$ from $\ell,\ell' \in \hcalM^1(\calR_r) - \calM^1(\calR_r)$ where $\calL_{\entropy,\calR_r}((\ell_n)_t | [0,1]) \to 0$.    As the lengths of the paths $(\ell_n)_t$ goes to $0$, by Proposition~\ref{prop:infinite length}, there is an $\epsilon > 0$ such that $((\ell_n)_t)^i \geq \epsilon$ for all $t \in [0,1]$ and all $n$.  Hence the images of the paths lie in a compact set of $\hcalM^1(\calR_r)$ and by the Arzel\`a--Ascoli theorem, there is a path $\ell_t \from [0,1] \to \hcalM^1(\calR_r)$ from $\ell$ to $\ell'$ with length $0$.    

The image of such a path must be contained in $\hcalM^1(\calR_r) - \calM^1(\calR_r)$.  As 
\begin{equation*}
\hcalM^1(\calR_r) - \calM^1(\calR_r) = \bigcup_{S \subset [r]} \calM_S,
\end{equation*}
we must have that $\dot\ell_t = 0$ since the semi-norm is nondegenerate on any $T\calM^1_S$ by Proposition~\ref{prop:norm extends} and hence $\ell = \ell'$.

This shows \eqref{item:rose distance extends}.  

Item~\eqref{item:rose iso embedding} follows as any path in $\hcalM^{1}(\calR_{r})$ with endpoints in $\calM^{1}(\calR_{r})$ is close to a path entirely contained in $\calM^{1}(\calR_{r})$ by continuity of the semi-norm.  

Item~\eqref{item:rose topology} now follows by continuity of the semi-norm.  
\end{proof}

 
\subsection{Proof of Theorem~\ref{th:completion rose}}\label{subsec:proof rose}

We can now complete the proof of Theorem~\ref{th:completion rose} which states that the completion of $\calM^1(\calR_{r})$ with respect to the entropy metric 
is homeomorphic to the complement of the vertices of an $(r-1)$--simplex.  We accomplish this by showing that $(\hcalM^1(\calR_r),d_{\entropy,\calR_r})$ is the completion as we have already observed that this space is homeomorphic to the complement of the vertices of an $(r-1)$--simplex.

\begin{proof}[Proof of Theorem~\ref{th:completion rose}]
As the inclusion map $(\calM^{1}(\calR_{r}), d_{\entropy,\calR_r}) \to (\hcalM^{1}(\calR_{r}), d_{\entropy,\calR_r})$ is an isometric embedding by Proposition~\ref{prop:metric extends}\eqref{item:rose iso embedding} and the image is clearly dense, it remains to show that $(\hcalM^{1}(\calR_{r}), d_{\entropy,\calR_r})$ is complete.

To this end, let $(\ell_n )_{n \in \NN} \subset \hcalM^{1}(\calR_{r})$ be a Cauchy sequence.  Then for each $1 \leq i \leq r$, we have that $(\ell^i)_{n}$ limits to some $\ell_\infty^i \in [0,\infty]^{r}$ where $\oF_{r}(\ell^1_\infty,\ldots,\ell^r_{\infty}) = 0$.  What remains to be shown is that such a limiting length function $\ell_\infty$ belongs to $\hcalM^{1}(\calR_{r})$.  

Let $S = \{ i \in [r] \mid \ell^i_\infty \neq \infty\}$.  If $\ell^i \neq 0$ for all $i \in [r]$, then $\ell \in \calM^1_S \subset \hcalM^1(\calR_r)$ and we are done.  This is indeed always the case as by Proposition~\ref{prop:infinite length}, the limiting length $\ell^i_\infty$ is bounded away from zero for all $i$ since the sequence is Cauchy.    
\end{proof}

\begin{example}\label{ex:completion 2,3}
In this example, we compare the completion $\hcalM^1(\calR_3)$ with the closure of $CV(\calR_3,{\rm id})$ considered as a subset of $\RR^{\FF_3}$.  Recall, $CV(\calR_3,{\rm id}) \subset CV(\FF_3)$ is the space of length functions on $\calR_3$ with unit volume, i.e., the sum of the lengths of the edges is equal to 1.  For the current discussion, the marking is irrelevent.  For more information about the closure of $CV(\FF_r)$ in $\RR^{\FF_r}$ we refer the reader to the papers by Bestvina--Feighn~\cite{un:BF-OuterLimits} and  Cohen--Lustig~\cite{ar:CL95}.

We consider again Figure~\ref{fig:moduli space} in Section~\ref{subsec:Fr}, which shows $\calM^{1}(\calR_{r})$ for $r = 2$ and $r = 3$.  These images suggest that as the length of one of the edges goes to infinity, the moduli space $\calM^{1}(\calR_{3})$ limits to the moduli space $\calM^{1}(\calR_{2})$ for the subgraph consisting of the other two edges.  There are three such $\calR_{2}$ subgraphs in $\calR_{3}$, each contributing a 1--dimensional face to $\hcalM^{1}(\calR_{3})$.  Figure~\ref{fig:completion} shows a schematic of $\hcalM^{1}(\calR_{3})$ contrasted with $\overline{CV(\calR_{3},{\rm id})}$, the closure of $CV(\calR_3,{\rm id})$ considered as a subset of $\RR^{\FF_3}$.  The spaces are not homeomorphic; $\hcalM^{1}(\calR_{3})$ is a 2--simplex without vertices whereas $\overline{CV(\calR_{3},{\rm id})}$ is a 2--simplex.  A more striking difference comes from the duality between the newly added edges and vertices.  
\begin{itemize}

\item In $\hcalM^{1}(\calR_{3})$, as $c \to \infty$ we obtain a copy of $\calM^{1}(\calR_{2})$ for the subgraph on $a$ and $b$.  In $\overline{CV(\calR_{3},{\rm id})}$, the corresponding sequence would send $a,b \to 0$, $c \to 1$ and the result is a single point corresponding to the graph of groups decomposition of $\I{a,b,c}$ with underlying graph $\calR_1$ where $\I{a,b}$ is the vertex group and the edge group is trivial.  

\item In $\hcalM^{1}(\calR_{3})$, as $b,c \to \infty$ there is no limit.  This is a missing vertex of the 2--simplex; this stems from the fact that $\calR_1$ cannot be scaled to have unit entropy.  In $\overline{CV(\calR_{3},{\rm id})}$, the corresponding sequence would send $a \to 0$ and we obtain a 1--dimensional face in the closure corresponding to unit volume length functions on the graph of groups decomposition of $\I{a,b,c}$ with underlying graph $\calR_{2}$ where $\I{a}$ is the vertex group and both edge groups are trivial.  
\end{itemize}
\end{example}

\begin{figure}[ht]
\centering
\begin{tikzpicture}[scale=1.3]
\draw[thick,fill=black!10] (-3,3) node[circle,thin,inner sep=1.2pt,draw,fill=white] {} -- (-1.35,0) node[circle,thin,inner sep=1.2pt,draw,fill=white] {} -- (-4.65,0) node[circle,thin,inner sep=1.2pt,draw,fill=white] {} -- cycle;
\draw[fill] (-1.5,1.9) circle (1pt); 
\draw (-1.5,1.9) arc[radius=5pt, start angle=180, end angle=-180]; 
\draw (-1.5,1.9) arc[radius=5pt, start angle=0, end angle=360]; 
\node at (-2.0,1.9) {$a$};
\node at (-1.0,1.9) {$b$};
\node at (-1.5,1.5) {$c = \infty$};
\draw[fill] (-4.5,1.9) circle (1pt); 
\draw (-4.5,1.9) arc[radius=5pt, start angle=180, end angle=-180]; 
\draw (-4.5,1.9) arc[radius=5pt, start angle=0, end angle=360]; 
\node at (-5.0,1.9) {$b$};
\node at (-4.0,1.9) {$c$};
\node at (-4.5,1.5) {$a = \infty$};
\draw[fill] (-3,-0.4) circle (1pt); 
\draw (-3,-0.4) arc[radius=5pt, start angle=180, end angle=-180]; 
\draw (-3,-0.4) arc[radius=5pt, start angle=0, end angle=360]; 
\node at (-3.5,-0.4) {$a$};
\node at (-2.5,-0.4) {$c$};
\node at (-3,-0.8) {$b = \infty$};
\draw[thick,fill=black!10] (1.35,3) node[circle,inner sep=1.2pt,draw,fill=black] {} -- (4.65,3) node[circle,inner sep=1.2pt,draw,fill=black] {} -- (3,0) node[circle,inner sep=1.2pt,draw,fill=black] {} -- cycle;
\draw[fill] (0.8,3.3) circle (1pt);
\draw (0.8,3.3) arc[radius=5pt, start angle=180, end angle=-180]; 
\node at (0.3,3.3) {$\I{b,c}$};
\draw[fill] (5.2,3.3) circle (1pt);
\draw (5.2,3.3) arc[radius=5pt, start angle=0, end angle=360]; 
\node at (5.7,3.3) {$\I{a,b}$};
\draw[fill] (3.2,-0.4) circle (1pt);
\draw (3.2,-0.4) arc[radius=5pt, start angle=180, end angle=-180]; 
\node at (2.7,-0.4) {$\I{a,c}$};
\draw[fill] (4.6,1.5) circle (1pt);
\draw (4.6,1.5) arc[radius=5pt, start angle=180, end angle=-180]; 
\draw (4.6,1.5) arc[radius=5pt, start angle=0, end angle=360]; 
\node at (4.1,1.5) {$b$};
\node at (5.1,1.5) {$c$};
\node at (4.6,1.1) {$\I{a}$};
\draw[fill] (1.4,1.5) circle (1pt);
\draw (1.4,1.5) arc[radius=5pt, start angle=180, end angle=-180]; 
\draw (1.4,1.5) arc[radius=5pt, start angle=0, end angle=360]; 
\node at (0.9,1.5) {$a$};
\node at (1.9,1.5) {$b$};
\node at (1.4,1.1) {$\I{c}$};
\draw[fill] (3,3.7) circle (1pt);
\draw (3,3.7) arc[radius=5pt, start angle=180, end angle=-180]; 
\draw (3,3.7) arc[radius=5pt, start angle=0, end angle=360]; 
\node at (2.5,3.7) {$a$};
\node at (3.5,3.7) {$c$};
\node at (3,3.3) {$\I{b}$};
\end{tikzpicture}
\caption{The completion of entropy normalization $\hcalM^{1}(\calR_{3})$ contrasted with the closure of the volume normalization $\overline{CV(\calR_{3},{\rm id})}$ in $\RR^{\FF_3}$.}\label{fig:completion}
\end{figure}
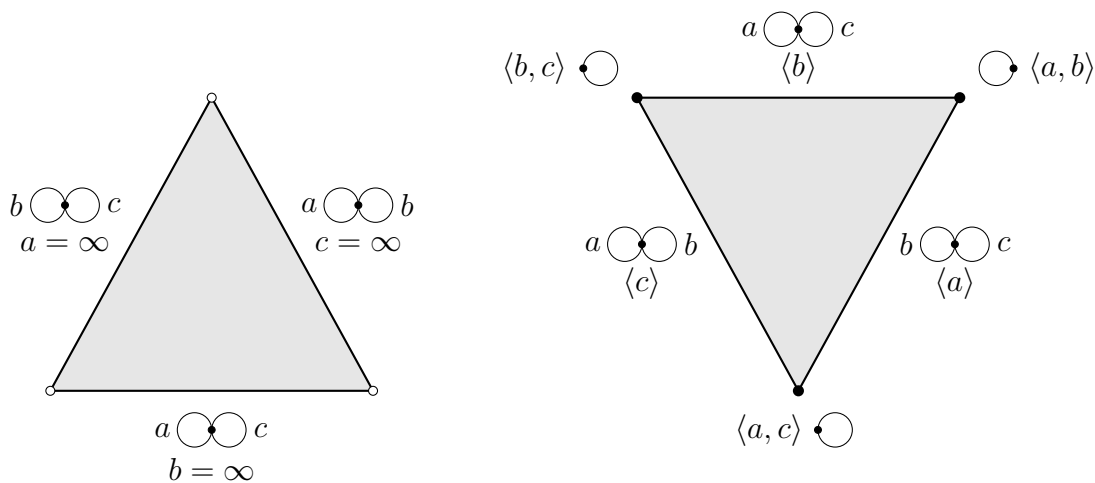
 

\subsection{The Thin Part of \texorpdfstring{$\calM^1(\calR_r)$}{M\textasciicircum1(R\_r)}}\label{subsec:rose-thin}

For $\epsilon > 0$ and $i \in [r]$, we define 
\[ \calS^{i}_{\epsilon} = \left\{ \ell \in \calM^{1}(\calR_r) \mid \ell^i = \epsilon \right\}. \] We use the letter ``S'' as we think of this subset as a slice of the moduli space.  The goal of this section is to prove the following proposition.

\begin{proposition}\label{prop:thin part}
Let $r \geq 2$ and let $i \in [r]$.  Then
\begin{equation*}
\lim_{\epsilon \rightarrow 0^+}  \diam(\calS^{i}_{\epsilon})  = 0.
\end{equation*}
\end{proposition}

Topologically, we have seen that $\hcalM^1(\calR_r)$ is homeomorphic to a simplex with its vertices removed.  Proposition~\ref{prop:thin part} shows that \em geometrically \em $\hcalM^1(\calR_r)$ is similar to an ideal hyperbolic simplex, with cross-sections whose diameter shrink to zero as we move towards an ideal vertex.

Given distinct $i,j \in [r]$ and $\epsilon > 0$, we let $\ell_{i,j}(\epsilon)$ denote the length function in $\calM^1_{\{i,j\}}$ where $(\ell_{i,j}(\epsilon))^i = \epsilon$.  As a result, we get that
\begin{equation}\label{eq:ell i,j j}
(\ell_{i,j}(\epsilon))^j = \log\left(\frac{\exp(-\epsilon) + 3}{\exp(-\epsilon) - 1}\right)
\end{equation}  
by Lemma~\ref{lem:ellr}.  

We begin with a technical lemma that bounds the length of a path in $\calS_\epsilon^i$ to such a point. 

\begin{lemma}\label{lem:thin part}
Let $r \geq 2$.  There is a constant $D$ with the following property.   Let $0 < \epsilon < \log(2)$ and let $i \in [r]$.  Suppose $\ell \in \calS_\epsilon^i$ and that $j \in [r] - \{i\}$ is such that $\ell^j = \min\{\ell^k \mid k \in [r] - \{i\}\}$.  Then
\begin{equation*}
d_{\entropy,\calR_r}(\ell,\ell_{i,j}(\epsilon)) \leq \frac{D}{-\log(\exp(\epsilon)-1)}.
\end{equation*}  
\end{lemma}

\begin{proof}
Using the notation of the lemma, we consider the path $\ell_t \from [0, \infty) \rightarrow \mathcal{S}^{i}_{\epsilon}$ defined by
\begin{equation}\label{eq:slice path}
\ell_t^k= \ell^k + t, \, k \neq i,j ; \; \ell^i_t = \epsilon. 
\end{equation}
Note that $\ell^j_t$ is not specified; its value is determined by the constraint that $\entropy_r(\ell_t) = 1$.  Notice that $\ell_t$ extends to a path $\ell_t \from [0,\infty] \to \hcalM^1(\calR_r)$ and $\ell_\infty = \ell_{i,j}(\epsilon) \in \calM^1_{\{i,j\}}$.  We observe that since the length of $\ell^k_t$ is increasing for $k \neq i,j$ and as $\ell^i_t$ is constant, we have $\ell^j_t \leq \ell^j$ for all $t$.

Given a subset $S \subseteq [r]$, let $\abs{S}_i = \abs{S} - 1$ if $i \in S$ and let $\abs{S}_i = \abs{S}$ otherwise.  With this definition, for a subset $S \subseteq [r]  - \{j\}$ we have that $\ell^S_t = \ell^S + \abs{S}_i t$.  Therefore using the definition of $X_j(\ell_t)$~\eqref{eq:Xi} and $Y_j(\ell_t)$~\eqref{eq:Yi} we find that
\begin{align}
X_j(\ell_t) & = \sum_{S \subseteq [r] - \{j\}} (1 - 2\abs{S}) \exp(-\ell^S - \abs{S}_i t)\label{al:Xj thin}, \mbox{ and} \\
Y_j(\ell_t) & = \sum_{S \subseteq [r] - \{j\}} (1 + 2\abs{S}) \exp(-\ell^S - \abs{S}_i t).\label{al:Yj thin}
\end{align}
Let $p(t) = \log(Y_j(\ell_t))$ and $q(t) = -\log(X_j(\ell_t))$ and so $\ell^j(t) = p(t) + q(t)$ by~\eqref{eq:ell i}.  The next two claims establish bounds on the second derivatives of $p(t)$ and $q(t)$. 

\begin{claim}\label{claim:C1}
There is a constant $C_1$ such that $\abs{\ddot p(t)} \leq C_1 \exp(-t)$.
\end{claim}

\begin{proof}[Proof of Claim~\ref{claim:C1}]
Using that fact that $1 < Y_j(\ell_t)$~\eqref{eq:Yi bound}, we have
\begin{align*}
\abs{\ddot p(t)} &= \abs{\frac{\ddot Y_j(\ell_t) Y_j(\ell_t) - \dot Y_j(\ell_t)^2}{Y_j(\ell_t)^2}} 
 \leq \abs{\frac{\ddot Y_j(\ell_t)}{Y_j(\ell_t)}} + \abs{\frac{\dot Y_j(\ell_t)}{Y_j(\ell_t)}}^2 
 \leq \abs{\ddot Y_j(\ell_t)} + \abs{\dot Y_j(\ell_t)}^2.
\end{align*}
The summands in $\abs{\ddot Y_j(\ell_t)}$ have the form 
\begin{equation}\label{eq:p terms 1}
\abs{S}_i^2(1 + 2\abs{S})\exp(-\ell^S - \abs{S}_i t). 
\end{equation}
The summands in $\abs{\dot Y_j(\ell_t)}^2$ have the form 
\begin{equation}\label{eq:p terms 2}
\abs{S}_i \abs{S'}_i (1 + 2\abs{S})(1 + 2\abs{S'})\exp(-\ell^S - \ell^{S'} - (\abs{S}_i + \abs{S'}_i)t).
\end{equation}
Each nonzero term in~\eqref{eq:p terms 1} and~\eqref{eq:p terms 2} has the form $A\exp(-B - Ct)$ where $A, B \geq 0$ and $C \geq 1$.  Hence each term is bounded by $A\exp(-t)$ for some $A \geq 0$.  The existence of $C_1$ is now clear.
\end{proof}

\begin{claim}\label{claim:C2}
There is a constant $C_2$ such that $\abs{\ddot q(t)} \leq C_2 \exp(-t)$.
\end{claim}

\begin{proof}[Proof of Claim~\ref{claim:C2}]
Using the facts that $1 < Y_j(\ell_t)$~\eqref{eq:Yi bound} and $\exp(-\ell^j_t)Y_j(\ell_t) = X_j(\ell_t)$~\eqref{eq:F = 0} we find that $\exp(-\ell^j_t) \leq \exp(-\ell^j_t)Y_j(\ell_t) = X_j(\ell_t)$.  Hence
\begin{align*}
\abs{\ddot q(t)} &= \abs{\frac{\ddot X_j(\ell_t^j) X_j(\ell_t) - \dot X_j(\ell_t)^2}{X_j(\ell_t)^2}} \leq \abs{\frac{\ddot X_j(\ell_t)}{X_j(\ell_t)}} + \abs{\frac{\dot X_j(\ell_t)}{X_j(\ell_t)}}^2  \leq \exp(\ell^j_t)\abs{\ddot X_j(\ell_t)} + \exp(2\ell^j_t)\abs{\dot X_j(\ell_t)}^2.
\end{align*}
The summands in $\exp(\ell^j_t)\abs{\ddot X_j(\ell_t)}$ have the form 
\begin{equation}\label{eq:q terms 1}
\abs{S}_i^2(1 + 2\abs{S})\exp(\ell^j_t - \ell^S - \abs{S}_i t).
\end{equation}  
The summands in $\exp(2\ell^j_t)\abs{\dot X_j(\ell_t)}^2$ have the form 
\begin{equation}\label{eq:q terms 2}
\abs{S}_i\abs{S'}_i(1 + 2\abs{S})(1 + 2\abs{S'})\exp(2\ell^j_t -\ell^S - \ell^{S'} - (\abs{S}_i + \abs{S'}_i)t).
\end{equation}
As $\ell^j_t \leq \min\{\ell^k \mid k \in [r] - \{i\}\}$, we find that $\ell^j_t - \ell^S \leq 0$ for all $S \subseteq [r] - \{j\}$ when $\abs{S}_i \neq 0$.  Hence, as above, each nonzero term in~\eqref{eq:q terms 1} and~\eqref{eq:q terms 2} has the form $A\exp(-B - Ct)$ where $A,B \geq 0$ and $C \geq 1$.  The existence of $C_2$ is now clear.
\end{proof}

Now we can bound the entropy norm of $(\ell_t, \dot \ell_t)$ using Proposition~\ref{prop:o metric formulas}.  As $\ddot \ell^k_t = 0$ for $k \neq j$ and $\ell^k_t\PD{\oF_r}{k}(\ell_t) = \ell^k_t\exp(-\ell^k_t)Y_k(\ell_t) > 0$~\eqref{eq:partial F} for all $k$ we find
\begin{equation*}
\norm{(\ell_t,\dot \ell_t)}^2_{\entropy,\calR_r} = \frac{\I{\ddot \ell_t,\nabla \oF_r(\ell_t)}}{\I{\ell_t,\nabla \oF_r(\ell_t)}} \leq \frac{\ddot \ell^j_t \PD{\oF_r}{j}(\ell_t)}{\ell^j_t\PD{\oF_r}{j}(\ell_t)} = \frac{\ddot \ell^j_t}{\ell^j_t} \leq \frac{\ddot \ell^j_t}{\ell^j}.
\end{equation*}

Thus we can bound the length of the path $\ell_t$ by
\begin{align*}
\calL_{\entropy,\calR_r}(\ell_t|[0,\infty)) & = \int_0^\infty \norm{(\ell_t,\dot \ell_t)}_{\entropy,\calR_r} \, dt  \leq \frac{1}{\ell^j}\int_0^\infty \sqrt{\ddot \ell^j_t} \, dt \\
& \leq \frac{\sqrt{C_1 + C_2}}{\ell^j} \int_0^\infty \exp(-t/2) \, dt =  \frac{2\sqrt{C_1 + C_2}}{\ell^j}.  
\end{align*}
As $\ell_j \geq -\log(\exp(\epsilon)-1)$~\eqref{eq:ell i,j j}, setting $D = 2\sqrt{C_1 + C_2}$ we complete the proof of the lemma. 
\end{proof}

\begin{proof}[Proof of Proposition~\ref{prop:thin part}]
Let $r \geq 2$ and let $D$ be the constant from Lemma~\ref{lem:thin part}.  Fix $i \in [r]$ and let $\ell_\epsilon \in \calS_\epsilon^i$ defined by $\ell^i_\epsilon = \epsilon$ and all other $\ell^j_\epsilon$'s are equal.  By Lemma~\ref{lem:thin part}, we have that \[d_{\entropy,\calR_r}(\ell_\epsilon,\ell_{i,j}(\epsilon)) \leq \frac{D}{-\log(\exp(\epsilon) - 1)}\] for all $j \in [r] - \{i\}$.  In particular, the set $\{\ell_{i,j}(\epsilon) \mid j \in [r] - \{i\}\}$ has diameter at most $\frac{2D}{-\log(\exp(\epsilon) - 1)}$.  

Again, by Lemma~\ref{lem:thin part}, each $\ell \in \calS_\epsilon^i$ has distance at most $\frac{D}{-\log(\exp(\epsilon)-1)}$ from some point in $\{\ell_{i,j}(\epsilon) \mid j \in [r] - \{i\}\}$.  Hence
\[ \diam(\calS_\epsilon^i) \leq  \frac{3D}{-\log(\exp(\epsilon) - 1)}. \]
As $-\log(\exp(\epsilon) - 1) \to \infty$ as $\epsilon \to 0^+$, the proof of the proposition is complete.
\end{proof}


\section{The Moduli Space of a Graph With a Separating Edge}\label{sec:separating}

The purpose of this section is to introduce tools to analyze the entropy metric on the moduli space of a graph with a separating edge.  Throughout this section, let $G = (V,E,o,\tau,\bar{\phantom{e}})$ be a finite connected graph which consists of two disjoint connected subgraphs $G_1$, with edges $e^1_1,\ldots,e^1_{n_1}$, and $G_2$, with edges $e^2_1,\ldots, e^2_{n_2}$ connected by an edge $e_0$.  We assume that both $G_1$ and $G_2$ have rank at least $2$.  We begin our analysis in Section~\ref{subsec:sep-finite length} by showing that there exist paths of finite length limiting to any unit entropy metric on either $G_1$ or $G_2$.  Using this, in Section~\ref{subsec:completion separating} we construct a space $\hcalM^1(G)$ that is similar to the construction of $\hcalM^1(\calR_r)$ from Section~\ref{sec:completion of rose}.  The main difference is that in this section, we do not bother to construct the entire completion of $(\calM^1(G),d_{\entropy,G})$, rather we just add the points that correspond to an entropy $1$ length function on $G_1 \cup e_0$ or $G_2 \cup e_0$ or $G_1 \cup G_2$.  This is sufficient for our purposes.  Proposition~\ref{prop:sep-metric extends} is the culmination of this analysis where we show that there is a map from $\hcalM^1(G)$ to the completion of $(\calM^1(G),d_{\entropy,G})$ that collapses these newly added length functions to a single point.

\subsection{Finite Length Paths in \texorpdfstring{$\calM^1(G)$}{M\textasciicircum1(G)}}\label{subsec:sep-finite length}
We seek to show that there is a finite length path in $(\calM^1(G),d_{\entropy,G})$ that limits onto an arbitrary unit entropy metric on either $G_1$ or $G_2$.  This calculation appears in Proposition~\ref{prop:sep-finite length}.  Given a length function $\ell \in \calM(G)$ we denote $\ell^{0} = \ell(e_0)$, $\ell^1 = (\ell(e^1_1),\ldots,\ell(e^1_{n_1}))$ and $\ell^2 = (\ell(e^2_1),\ldots,\ell(e^2_{n_2}))$.

Given a simplex $\Delta \in C_G$ and an edge $e \in E$ of $G$, we recall that $\Delta(e)$ denotes the number of times $e$ or $\bar e$ appears as a vertex in a simple cycle contained in $\Delta$.  By the construction of $C_G$ we have that $\Delta(e) \in \{0,1,2\}$ for any edge.  Further, $\Delta(e_0) \in \{0,2\}$ as $e_0$ is separating.  
 
Analogous to the functions $Y_i$ defined in Section~\ref{subsec:not complete} that allowed us to isolate the variable $\ell^i$ for the $r$--rose, we define a function $Y \from \calM(G) \to \RR$ by
\begin{equation*}
Y(\ell) = -\sum_{\substack{\Delta \in C_{G}  \\ \Delta(e_0) = 2}}(-1)^{\splx{\Delta}} \exp(-\ell(\Delta) + 2\ell^0).
\end{equation*}
Notice that this function is constant with respect to $\ell^0$ as we may write
\[ \ell(\Delta) = \sum_{e \in E_+} \Delta(e)\ell(e).  \]  Hence if $\Delta(e_0) = 2$, then $\ell(\Delta) - 2\ell^0$ has no dependence on $\ell^0$.  Also we remark that the function $Y\from \calM(G) \to \RR$ has a continuous extension to $[0,\infty]^{\abs{E_+}}$ and is bounded on $[0,\infty]^{\abs{E_+}}$.  With this notation we have the following expression for $F_G$.

\begin{lemma}\label{lem:sep-FG}
Let $\ell \in \calM(G)$.  Then $F_{G}(\ell) = F_{G_1}(\ell^1)F_{G_2}(\ell^2) - \exp(-2\ell^0)Y(\ell)$.
\end{lemma}

\begin{proof}
Let $\Delta$ be a simplex in $C_G$.  If $\Delta(e_0) = 0$, then $\Delta$ is the join of two (possibly empty) simplices $\Delta_1 \in C_{G_1}$ and $\Delta_2 \in C_{G_2}$.  Indeed, if $\Delta = \{\gamma^1_1,\ldots,\gamma^1_{m_1},\gamma^2_1,\ldots, \gamma^2_{m_2}  \}$, then $\Delta = \Delta_1 \ast \Delta_2$ where $\Delta_i = \{\gamma^i_1,\ldots,\gamma^i_{m_i}\}$ for $i = 1$, $2$.  We have $\splx{\Delta} = m_1 + m_2 = \splx{\Delta_1} + \splx{\Delta_2} $ and thus
\[ (-1)^{\splx{\Delta}}\exp(-\ell(\Delta)) = \bigl((-1)^{\splx{\Delta_1}}\exp(-\ell^1(\Delta_1))\bigr)\bigl((-1)^{\splx{\Delta_2}}\exp(-\ell^2(\Delta_2))\bigr). \]
Therefore, by Theorem~\ref{thm:F}, we find that
\begin{align*} 
\sum_{\substack{\Delta \in C_{G}  \\ \Delta(e_0) = 0}}(-1)^{\splx{\Delta}} \exp(-\ell(\Delta)) &= \left(\sum_{\Delta_1 \in C_{G_1}}(-1)^{\splx{\Delta_1}}\exp(-\ell^1(\Delta_1))\right)\left( \sum_{\Delta_2 \in C_{G_2}}(-1)^{\splx{\Delta_2}} \exp(-\ell^2(\Delta_2))\right) \\
&= F_{G_1}(\ell^1) F_{G_2}(\ell^2).
\end{align*}

If $\Delta(e_0) = 2$, then
\begin{align*}
(-1)^{\splx{\Delta}}\exp(-\ell(\Delta)) & = \exp(-2\ell^0)(-1)^{\splx{\Delta}}\exp(-\ell(\Delta) + 2\ell^0).
\end{align*}  
Hence, by the definition of $Y(\ell)$ we have
\[ \sum_{\substack{\Delta \in C_{G}  \\ \Delta(e_0) = 2}}(-1)^{\splx{\Delta}} \exp(-\ell(\Delta)) = -\exp(-2\ell^0)Y(\ell). \]

As $e_0$ is separating, there are no simplices in $C_G$ for which $\Delta(e_0) = 1$.  By Theorem~\ref{thm:F} again, the lemma follows.
\end{proof}

In particular, for $\ell \in \calM^1(G)$ we have $F_G(\ell) = 0$ by Lemma~\ref{lem:F} and thus
\begin{equation}
\ell^0 = \frac{1}{2}\log\left(\frac{Y(\ell)}{F_{G_1}(\ell^1)F_{G_2}(\ell^2)} \right).
\end{equation}
Using the above expression for $F_G$ and $\ell^0$, we will give a method of a finite length path in $G$ for which $\ell^0 \to \infty$.

\begin{proposition}\label{prop:sep-finite length}
Fix a length function $\ell' \in \calM^1(G_1)$ and let $1 < x_1,\ldots,x_{n_1} < \infty$ be such that $\ell'(e^1_i) = \log(x_i)$ for $1 \leq i \leq n_1$.  Let $\ell \in \calM^1(G)$ be such that $\ell(e^1_i) = \log(x_i + 1)$ for $1 \leq i \leq n_1$ and let $\ell_t \from [0,1) \to \hcalM^1(G)$ be the path defined by $\ell_t(e^1_i) = \log(x_i + 1 - t)$, $\ell_t^2 = \ell^2$ and 
\begin{equation}
\ell_t^0 = \frac{1}{2} \log\left(\frac{Y(\ell_t)}{F_{G_1}(\ell_t^1)F_{G_2}(\ell_t^2)}\right).
\end{equation}
Then $\calL_{\entropy,G}(\ell_t|[0,1))$ is finite and $\ell_t^0 \to \infty$ as $t \to 1^-$.
\end{proposition}

\begin{proof}
We will use the notation as in the statement of the proposition.  As $\entropy_{G_1}(\ell') = 1$, we must have that $\ell_t^0 \to \infty$ at $t \to 1^-$.  Indeed, it not then the limiting length function has unit entropy and with this length function, the subgraph $G_1$ also has unit entropy.

Notice that since $\ell_t(e)\PD{F_G}{e}(\ell_t) > 0$ for all $e \in E_+$ by Lemma~\ref{lem:F}\eqref{item:positive partial}, we have that 
\begin{equation}
\I{\ell_t,\nabla F_G(\ell_t)} = \sum_{e \in E_+} \ell_t(e)\PD{F_G}{e}(\ell_t) \geq \sum_{i=1}^{n_1} \ell_t(e^1_i)\PD{F_G}{e^1_i}(\ell_t).
\end{equation}  
By Lemma~\ref{lem:sep-FG}, we compute that
\begin{equation} 
\PD{F_G}{e^1_i}(\ell_t) = F_{G_2}(\ell_t^2)\PD{F_{G_1}}{e^1_i}(\ell_t^1) - \exp(-2\ell_t^0)\PD{Y}{e^1_i}(\ell_t). 
\end{equation}
Thus since $\PD{Y}{e^1_i}(\ell)$ is bounded on $\calM(G)$, we see that $\I{\ell_t,\nabla F_G(\ell_t)}$ has a limit as $t \to 1^-$ that is bounded below by $F_{G_2}(\ell^2)\I{\ell',\nabla F_{G_1}(\ell')}$, which is positive by Lemma~\ref{lem:F}\eqref{item:positive partial} and Lemma~\ref{lem: h < 1}.  (We will see in Lemma~\ref{lem:sep-gradient}\eqref{item:sep-extension 2} that the limit is in fact exactly equal to $F_{G_2}(\ell^2)\I{\ell',\nabla F_{G_1}(\ell')}$.)  Hence, there is an $\epsilon > 0$ such that 
\begin{equation}\label{eq:sep vol bound}
\I{\ell_t,\nabla F_G(\ell_t)} \geq \epsilon  
\end{equation}
for all $t \in [0,1)$.
 
We compute that $\ddot \ell_t(e^1_i) = \frac{-1}{x_i + 1 - t} < 0$, hence $\ddot \ell_t(e^1_i) \PD{F_G}{e^1_i}(\ell_t) < 0$ for all $0 < t < 1$ and $1 \leq i \leq n_1$.  Clearly $\ddot \ell_t(e^2_i) \PD{F_G}{e^2_i}(\ell_t) = 0$ for all $1 \leq i \leq n_2$ and $0 < t < 1$.  Thus  
\begin{equation}\label{eq:sep ddot bound}
\I{\ddot \ell_t,\nabla F_G(\ell_t)} \leq \ddot \ell_t^0\PD{F_G}{e_0}(\ell_t).
\end{equation}
To deal with this term, we need the following claim.

\begin{claim}\label{claim:p and q}
There are polynomials $p,q \in \RR[t]$ where $p(t), q(t) \neq 0$ for $t \in [0,1]$ such that
\[ \exp(-2\ell_t^0) = \frac{(1-t)p(t)}{q(t)}. \]
\end{claim} 

\begin{proof}[Proof of Claim~\ref{claim:p and q}]
As $F_{G}(\ell_t) = 0$, we have that
\[ \exp(-2\ell_t^0) = \frac{F_{G_1}(\ell_t^1)F_{G_2}(\ell_t^2)}{Y(\ell_t)}. \]

Let $\ell(E_i) = \sum_{j = 1}^{n_i} \ell(e^i_j)$ for $i = 1$, $2$.  Notice that we can write $F_{G_1}(\ell_t^1)$ as
\[ F_{G_1}(\ell_t^1) = \sum_{\Delta \in C_{G_1}} (-1)^{\splx{\Delta}}\prod_{i=1}^{n_1} (x_i +1 -t)^{-\Delta(e^1_i)}. \]
Hence $\exp(2\ell_t(E_1))F_{G_1}(\ell_t^1)$ is a polynomial in $t$ with real coefficents.  Also, we observe that $\exp(2\ell_t(E_2))F_{G_2}(\ell_t^2)$ is a nonzero constant with respect to $t$.  By the definition of $\ell_t$ we have that $F_{G_1}(\ell_1^1) = 0$.  Hence we can write \[\exp(2\ell_t(E_1)+2\ell_t(E_2))F_{G_1}(\ell_t^1)F_{G_2}(\ell_t^2) = (1-t)p(t) \] where $p(t) \in \RR[t]$.  As the left-hand side of this equation does not vanish for $t \in [0,1)$ by Lemma~\ref{lem: h < 1}, we see that $p(t) \neq 0$ for $t \in [0,1)$.  To show that $p(1) \neq 0$, we see that
\begin{align*} 
p(1) & = \frac{d}{dt}(\exp(2\ell_t(E_1)+2\ell_t(E_2))F_{G_1}(\ell_t^1)F_{G_2}(\ell_t^2)) \big|_{t=1} \\
& =  \exp(2\ell_1(E_1) + 2\ell_1(E_2))F_{G_2}(\ell_1^2)\I{\dot \ell_1^1,\nabla F_{G_1}(\ell_1^1)}.
\end{align*}
As $\ell_1^1 \in \calM^1(G_1)$, we have that $\nabla F_{G_1}(\ell_1^1)$ is non-zero and parallel to $\nabla \entropy_{G_1}(\ell_1^1)$ by Lemma~\ref{lem:F}.  Since $\entropy_{G_1}(\ell_t^1)$ is increasing with respect to $t$ as every edge length is decreasing (past $t = 1$ as well), we have that $\I{\dot \ell_1^1,\nabla \entropy_{G_1}(\ell_1^1)} \neq 0$.  Hence $p(1) \neq 0$ as well.  

In a similar way, we observe that we can write
\[ \exp(2\ell_t(E_1) + 2\ell_t(E_2))Y(\ell_t) = q(t) \] for some $q(t) \in \RR[t]$.  As $Y(\ell_t) = \exp(2\ell^0)F_{G_1}(\ell_t^1)F_{G_2}(\ell_t^2)$ by Lemma~\ref{lem:sep-FG}, we see that $Y(\ell_t) \neq 0$ for $t \in [0,1)$ by Lemma~\ref{lem: h < 1} and hence $q(t) \neq 0$ for $t \in [0,1)$ as well.  As $t \to 1^-$, we have that $\ell_t^0 \to \infty$ and thus $\frac{(1-t)p(t)}{q(t)} \to 0$ as $t \to 1^-$.  As $p(1) \neq 0$, we must have that $q(1) \neq 0$ as well.
\end{proof}
By Claim~\ref{claim:p and q} we compute that
\begin{align*}
\ddot \ell_t^0 &= \frac{1}{2}\left(\frac{1}{(1-t)^2} - \frac{\ddot p(t) p(t) + (\dot p(t))^2}{(p(t))^2} + \frac{\ddot q(t) q(t) + (\dot q(t))^2}{(q(t))^2} \right).
\end{align*}
Using Lemma~\ref{lem:sep-FG} and Claim~\ref{claim:p and q}, we find that \[\PD{F_G}{e_0}(\ell) = 2\exp(-2\ell^0)Y(\ell) = 2Y(\ell)\frac{(1-t)p(t)}{q(t)}.\]  Hence we see that there exists a constant $C > 0$ such that
\begin{equation}\label{eq:sep ell0} 
\ddot \ell_t^0 \PD{F_G}{e_0}(\ell_t) \leq \frac{C}{1-t}.
\end{equation}
Therefore we have by combining Proposition~\ref{prop:metric formulas} with \eqref{eq:sep vol bound}, \eqref{eq:sep ddot bound} and \eqref{eq:sep ell0} that
\begin{equation*}
\norm{(\ell_t, \dot \ell_t)}^2_{\entropy,G} = \frac{\I{\ddot \ell_t,\nabla F_G(\ell_t)}}{\I{\ell_t,\nabla F_G(\ell_t)}} \leq \frac{\ddot \ell_t^0 \PD{F_G}{e_0}(\ell_t)}{\I{\ell_t,\nabla F_G(\ell_t)}} \leq \frac{C}{\epsilon(1-t)}.
\end{equation*}
Hence the length of $\ell_t$ is finite.
\end{proof}

\subsection{The Model Space \texorpdfstring{$\hcalM^1(G)$}{hat{M}\textasciicircum1(G)}}\label{subsec:completion separating}  The previous example shows that we should expect some points in the completion of $(\calM^1(G),d_{\entropy,G})$ to correspond to unit entropy metrics on $G_1$ or $G_2$.  For the model, we add these points to $\calM^1(G)$ as well as points that correspond to unit entropy metrics on $G_1 \cup G_2$.  To this end, we set:
\begin{align*}
\calM_1 & = \{ \ell \in (0,\infty]^{\abs{E_+}} \mid \ell^1 \in \calM(G_1) \mbox{ and } \ell^2 = \infty \cdot \One \}, \\
\calM_2 & = \{ \ell \in (0,\infty]^{\abs{E_+}} \mid \ell^1 = \infty \cdot \One \mbox{ and } \ell^2 \in \calM(G_2) \}, \mbox{ and} \\
\calM_{1,2} & = \{ \ell \in (0,\infty]^{\abs{E_+}} \mid \ell^0 = \infty, \ell^1 \in \calM(G_1), \mbox{ and } \ell^2 \in \calM(G_2) \}.
\end{align*}
We consider their union $\hcalM(G) = \calM(G) \cup \calM_1 \cup \calM_2 \cup \calM_{1,2}$ as a subset of $(0,\infty]^{\abs{E_+}}$, endowed with the subspace topology as in Section~\ref{sec:completion of rose}.  There are obvious embeddings $\varepsilon_i \from \calM(G_i) \to \calM_i$ for $i = 1$, $2$ where we set $\varepsilon_i(\ell)^0 = \infty$ and an obvious embedding $\varepsilon_{1,2} \from \calM(G_1) \times \calM(G_2) \to \calM_{1,2}$ as well.  Next, we define
\begin{align*}
\calM^1_1 & = \{ \ell \in \calM_1 \mid \entropy_{G_1}(\ell^1) = 1\}, \\
\calM^1_2 & = \{ \ell \in \calM_2 \mid \entropy_{G_2}(\ell^2) = 1\}, \mbox{ and} \\
\calM^1_{1,2} & = \{ \ell \in \calM_{1,2} \mid \max\{\entropy_{G_1}(\ell^1),\entropy_{G_2}(\ell^2)\} = 1 \}.
\end{align*} 
Our model space is the union of these sets.  Specifically, we define
\begin{equation}
\hcalM^1(G) = \calM^1(G) \cup \calM^1_1 \cup \calM^1_2 \cup \calM^1_{1,2}.
\end{equation}

Using~\eqref{eq:partial-e} we see that each partial derivative of $F_G$ extends to a bounded continuous function on $\hcalM(G)$.  Naturality of these extensions is the same as in Lemma~\ref{lem:partials inclusion}.  Even more, the inner product $\I{\ell,\nabla F_G(\ell)}$ has a continuous extension as was the case in Lemma~\ref{lem:total volume inclusion}.
 
\begin{lemma}\label{lem:sep-gradient}
The function $\ell \mapsto \I{\ell,\nabla F_G(\ell)}$ has a continuous extension to $\hcalM(G)$.  This extension satisfies the following.
\begin{enumerate}
\item \label{item:sep-extension 1} If $i \in \{1,2\}$ and $\ell \in \calM_i$, then
\[\I{\ell,\nabla F_{G}(\ell)} =  \I{\ell^i,\nabla F_{G_i}(\ell^i)}. \]
\item \label{item:sep-extension 2} If $\ell \in \calM_{1,2}$, then
\[\I{\ell,\nabla F_{G}(\ell)} =  F_{G_2}(\ell^2)\I{\ell^1,\nabla F_{G_1}(\ell^1)} + F_{G_1}(\ell^1)\I{\ell^2,\nabla F_{G_2}(\ell^2)}. \]
\item  \label{item:sep-extension 3} For all $\ell \in \hcalM(G)$, we have $\I{\ell,\nabla F_G(\ell)} \geq 0$ with equality if and only if $\entropy_{G_1}(\ell^1) = \entropy_{G_2}(\ell^2) = 1$.
\end{enumerate} 
\end{lemma}

\begin{proof}
From Lemma~\ref{lem:F}\eqref{item:positive partial} and the expression for $\PD{F_G}{e}(\ell)$ in \eqref{eq:partial-e}, we see that there is a constant $A > 0$ such that $0 < \PD{F_G}{e}(\ell) \leq A\exp(-\ell(e))$ for any edge $e \in E_+$.  The existence of the continuous extension now follows for the same reason as for Lemma~\ref{lem:total volume inclusion}.  

If $\ell^1 = \infty \cdot \One$, then $Y(\ell) = 0$.  Indeed, this follows as every simple cycle in $G$ that contains $e_0$ must also contain an edge in $G_1$ as $e_0$ is separating.  Likewise, if $\ell^2 = \infty \cdot \One$, then $Y(\ell) = 0$ as well.  Hence $\PD{F_G}{e_0}(\ell) = 2\exp(-2\ell^0)Y(\ell) = 0$ for $\ell \in \calM_i$ when $i = 1$, $2$.  This, together with the paragraph above, shows~\eqref{item:sep-extension 1}.

Using Lemma~\ref{lem:sep-FG}, we compute the following expression for $\I{\ell,\nabla F_G(\ell)}$.
\begin{multline*}
\I{\ell,\nabla F_{G}(\ell)} = F_{G_2}(\ell^2)\I{\ell^1,\nabla F_{G_1}(\ell^1)} + F_{G_1}(\ell^1)\I{\ell^2,\nabla F_{G_2}(\ell^2)} \\
- \exp(-2\ell^0)\left(\I{\hat{\ell},\nabla Y(\hat\ell)} - 2\ell^0 Y(\hat\ell)\right)
\end{multline*}
From this~\eqref{item:sep-extension 2} is apparent.

As $\I{\ell,\nabla F_G(\ell)} > 0$ for all $\ell \in \calM^1(G)$ by Lemma~\ref{lem:F}\eqref{item:positive partial}, the extension is clearly non-negative.  Statement~\eqref{item:sep-extension 3} now follows from~\eqref{item:sep-extension 1} and \eqref{item:sep-extension 2} as $\I{\ell,\nabla F_{G_i}(\ell)} > 0$ for any $\ell \in \calM^1(G_i)$ again by Lemma~\ref{lem:F}\eqref{item:positive partial} and $F_{G_i}(\ell) > 0$ for any $\ell \in \calM^1(G_i)$ if $\entropy(G_i)(\ell) < 1$ by Lemma~\ref{lem: h < 1}.  
\end{proof}

Next we partition $\hcalM^1(G)$ into two subsets that we call the \emph{singular points} and \emph{regular points} respectively.
\begin{align*}
\hcalM^1(G)_{\rm sing} & = \{ \ell \in \hcalM^1(G) \mid \entropy_{G_1}(\ell^1) = \entropy_{G_2}(\ell^2) = 1 \}, \mbox{ and} \\
\hcalM^1(G)_{\rm reg} & = \hcalM^1(G) - \hcalM^1(G)_{\rm sing}. 
\end{align*} 
Notice that $\hcalM^1(G)_{\rm sing}$ is a subset of $\calM^1_{1,2}$.  

As in Section~\ref{subsec:model} we also define the \emph{tangent bundle} $T\hcalM^1(G)$ to be the subspace of $(\ell,\bv) \in \hcalM^1(G) \times \RR^{\abs{E_+}}$ where $\I{\bv,\nabla F_G(\ell)} = 0$.  There are obvious embeddings $T\varepsilon_i \from T\calM^1(G_i) \to T\hcalM^1(G)$ for $i = 1$, $2$ defined using $\varepsilon_i \from \calM^1(G_i) \to \calM_i$ as in Section~\ref{subsec:model}.  We define $T\hcalM^1_{\rm reg}(G)$ to be the subset of $(\ell,\bv) \in T\hcalM^1(G)$ where $\ell \in \hcalM^1(G)_{\rm reg}$.  Proposition~\ref{prop:metric formulas} together with Lemma~\ref{lem:sep-gradient} implies the following.   

\begin{proposition}\label{prop:sep-norm extends}
The entropy norm $\norm{\param}_{\entropy,G} \from T\calM^{1}(G) \to \RR$ extends to a continuous semi-norm $\norm{\param}_{\entropy,G} \from T\hcalM^1(G)_{\rm reg} \to \RR$.  Moreover, the embedding maps $T\varepsilon_i \from T\calM^1(G_i) \to T\hcalM^1(G)$ are norm-preserving.  
\end{proposition}

As in Section~\ref{subsec:model}, we have the following proposition that shows us that there is a map from $\hcalM^1(G)$ to the completion of $\calM^1(G)$ with respect to the entropy metric.

\begin{proposition}\label{prop:sep-metric extends}
The following statements hold.
\begin{enumerate}
\item The entropy norm defines a pseudo-metric $d_{\entropy,G}$ on $\hcalM^{1}(G)$.\label{item:sep-pseudo-metric}
\item We have $\diam(\calM^1_1 \cup \calM^1_2 \cup \calM^1_{1,2}) = 0$.\label{item:sep-short cut}
\item The inclusion $(\calM^{1}(G),d_{\entropy,G}) \to (\hcalM^{1}(G),d_{\entropy,G})$ is an isometric embedding.\label{item:sep-isometric embedding}
\end{enumerate}
\end{proposition}

\begin{proof}
As in Proposition~\ref{prop:metric extends}, we need to show that for any $\ell, \ell' \in \hcalM^1(G)$ there is a path $\ell_t \from [0,1] \to \hcalM(G)$ with $\ell_0 = \ell$ and $\ell_1 = \ell'$ that has finite length.

To this end, fix a point $\ell \in \calM^1(G)$.  There are several cases depending on if $\ell' \in \calM^1_1$, $\ell' \in \calM^1_2$ or $\ell' \in \calM^1_{1,2}$.  

We first deal with the case that $\ell' \in \calM^1_{1,2}$.  Without loss of generality we assume that $\entropy_{G_1}((\ell')^1) = 1$.  In Proposition~\ref{prop:sep-finite length} we produced a path $\ell_t \from [0,1] \to \hcalM^1(G)$ where $\ell_0 \in \calM^1(G)$ and $\ell_1 \in \calM_{1,2}$ is such that $\ell_1^1 = (\ell')^1$.  We can concatenate the path $\ell_t$ with a path $\tilde \ell_t \from [0,1] \to \calM^1_{1,2}$ from $\ell_1$ to $\ell'$ in $\calM^1_{1,2}$ as follows.  
We define the path by \[\tilde \ell_t = (\infty,\ell^1_1,(1-t) \cdot \ell_1^2 + t \cdot (\ell')^2)\] and we observe that $\tilde \ell_0 = \ell_1$ to $\tilde \ell_1 = \ell'$.  Note that by the convexity of entropy we have that \[\entropy_{G_2}((1-t) \cdot \ell_1^2 + t \cdot (\ell')^2) \leq 1\] with equality only possible at the endpoints.  Hence the interior of the path $\tilde \ell_t$ lies in $\hcalM^1(G)_{\rm reg}$.  Further, \[\norm{(\tilde \ell_t, \vardottilde{\ell}_t)}_{\entropy,G} = 0 \] as the length of edges in $G_2$ are changing linearly.  Hence there is a path of finite length from a length function in $\calM^1(G)$ to any length function in $\calM^1_{1,2}$.  

Next, we deal with the case that $\ell' \in \calM_1$; the case of $\ell' \in \calM_2$ is symmetric.  We will show that we can connect $\ell'$ to a length function in $\calM_{1,2}$ with a concatenation of two paths that have finite length---in fact each has length $0$.  Let $\ell'' \in \calM(G_2)$ have entropy less than $1$.  The paths $(\ell_1)_t$ and $(\ell_2)_t$ are as follows:
\begin{align*}
(\ell_1)_t & \from [0,\infty] \to \hcalM^1(G) & (\ell_1)_t & = ((\ell')^0 + t,(\ell')^1,\infty \cdot \One) \\
(\ell_2)_t & \from [0,\infty] \to \hcalM^1(G) & (\ell_2)_t & = (\infty,(\ell')^1,(\ell')^2 + t \cdot \One).
\end{align*}
The concatenation of $(\ell_1)_t|[0,\infty]$ and $(\ell_2)_t|[\infty,0]$ is a path from $\ell'$ to $\ell''$.  We observe that $(\ddot \ell_1)_t = 0$ and $(\ddot \ell_2)_t = 0$ as edge lengths are changing linearly.  Also, we observe that the interiors of these paths lie in $\hcalM^1(G)_{\rm reg}$.  Hence $\norm{((\ell_k)_t,(\dot \ell_k)_t)}_{\entropy,G} = 0$ for $k = 1$, $2$ showing that the paths have finite---in fact zero---length.

This shows \eqref{item:sep-pseudo-metric}.

The previous argument shows that for any $\ell \in \calM^1_1$, there is an $\ell' \in \calM^1_{1,2}$ such that $d_{\entropy,G}(\ell,\ell') = 0$.  Likewise the analogous statement holds for $\ell \in \calM^1_2$.  Given, $\ell, \ell' \in \calM_{1,2}$, we will show that $d_{\entropy,G}(\ell,\ell') = 0$, completing the proof of~\eqref{item:sep-short cut}.  There are four cases depending on the entropies of the length functions $\ell^1$, $\ell^2$, $(\ell')^1$ and $(\ell')^2$.  

The first case we consider is when $\entropy_{G_1}(\ell^1) = 1$ and $\entropy_{G_2}((\ell')^2) = 1$.  In this case we consider the concatenation of the two paths $(\ell_1)_t$ and $(\ell_2)_t$ that are defined as follows:
\begin{align*}
(\ell_1)_t & \from [0,1] \to \hcalM^1(G) & (\ell_1)_t & = (\infty,\ell^1,(1-t) \cdot \ell^2 + t\cdot (\ell')^2) \\
(\ell_2)_t & \from [0,1] \to \hcalM^1(G) & (\ell_2)_t & = (\infty,(1-t) \cdot \ell^1 + t \cdot (\ell')^1,(\ell')^2).
\end{align*}
As above, the interiors of these paths lie in $\hcalM^1(G)_{\rm reg}$ and they have length $0$ since the lengths of edges are changing linearly.  This completes this case.

The case where $\entropy_{G_2}(\ell^2) = 1$ and $\entropy_{G_1}((\ell')^1) = 1$ is similar.

Next we consider the case where $\entropy_{G_1}(\ell^1) = 1$ and $\entropy_{G_1}((\ell')^1) = 1$.  Fix a length function $\ell'' \in \calM^1_{1,2}$ with $\entropy_{G_2}((\ell'')^2) = 1$.  By the first argument, we can connect both $\ell$ and $\ell'$ to $\ell''$ with paths of length $0$.  Concatenating these two paths shows that this case holds as well.

The case where $\entropy_{G_2}(\ell^2) = 1$ and $\entropy_{G_2}((\ell')^2) = 1$ is similar.

This completes the proof of~\eqref{item:sep-short cut}. 

We observe that any path in $\hcalM^1(G)$ is close to a path in $\calM^1(G)$.  Thus by Proposition~\ref{prop:sep-norm extends}, we have that the inclusion $(\calM^1(G),d_{\entropy,G}) \to (\hcalM^1(G),d_{\entropy,G})$ is an isometric embedding, hence~\eqref{item:sep-isometric embedding} holds.
\end{proof}


\section{\texorpdfstring{$\calX^1(\calR_r,{\rm id})$}{X\textasciicircum1(R\_r,id)} Has Bounded Diameter in \texorpdfstring{$\calX^1(\FF_r)$}{X\textasciicircum1(F\_r)}}\label{sec:rose bounded}

In this section, we make use of the collapsing phenomona witnessed in the previous section to show that even though $(\calM^1(\calR_r),d_{\entropy,\calR_r})$ has infinite diameter (Proposition~\ref{prop:infinite length}), the subspace $(\calX^1(\calR_r,{\rm id}),d_\entropy) \subset (\calX^1(\FF_r),d_\entropy)$ has bounded diameter.  The idea is as follows.  Using the natural bijection $\calM^1(\calR_r) \leftrightarrow \calX^1(\calR_r,{\rm id})$, since $\hcalM^1(\calR_r)$ is the completion of $(\calM^1(\calR_r),d_{\entropy,\calR_r})$ (Section~\ref{sec:completion of rose}) there is a map $\Phi \from \hcalM^1(\calR_r) \to \ccalX^1(\FF_r)$ where $\ccalX^1(\FF_r)$ is the completion of $(\calX^1(\FF_r),d_\entropy)$.  Indeed, if a sequence $(\ell_n) \subset (\calM^1(\calR_r),d_{\entropy,\calR_r})$ is Cauchy, then so is its image under $\Phi$ in $(\calX^1(\FF_r),d_\entropy)$ as $\Phi$ is 1--Lipschitz.  As $\hcalM^1(\calR_r)$ is homeomorphic to the complement of the vertices of an $(r-1)$--simplex, we can consider $\Phi$ as map $\Phi \from \Delta^{r-1} - V \to \ccalX^1(\FF_r)$ where $\Delta^{r-1}$ is the standard $(r-1)$--dimensional simplex and $V \subset \Delta^{r-1}$ is the set of vertices.  We will show that the map $\Phi$ extends to the vertex set $V$.  Since the image $\Phi(\Delta^{r-1})$ is compact and contains $\calX^1(\calR_r,{\rm id})$, it follows that $(\calX^1(\calR_r,{\rm id}),d_{\entropy})$ has bounded diameter. 

This is accomplished by considering $\calM^1(\calR_r)$ as the face of $\calM^1(G)$ for a particular graph $G$ that has a separating edge and using the tools developed in Section~\ref{sec:separating}.  Lemma~\ref{lem:single point} establishes that the subset $\cup \{\calM_S \mid 1 < \abs{S} < r-1 \} \subset \hcalM^1(\calR_r)$ is collapsed to a single point in the completion of $(\calX^1(\FF_r),d_\entropy)$.  We recall that $\calM^1_S \subset \hcalM^1(\calR_r)$ is the subset of unit entropy length functions on the subrose $\calR_{\abs{S}} \subseteq \calR_r$ utilizing the edges in $S \subseteq [r]$---the lengths of an edge in $[r] - S$ is $\infty$.  The subset $\calM^1_S$ corresponds to an $(\abs{S}-1)$--dimensional face of $\Delta^{r-1}$.  Thus Lemma~\ref{lem:single point} shows that the entire $(r-3)$--skeleton of $\Delta^{r-1}$ is collapsed to a point by $\Phi$ in $\ccalX^1(\FF_r)$.

\begin{lemma}\label{lem:single point}
For $r \geq 4$, $\Phi$ maps the subset $\cup \{\calM_S \mid 1 < \abs{S} < r-1 \} \subset \hcalM^1(\calR_r)$ to a single point in $\ccalX^1(\FF_r)$.
\end{lemma}

\begin{proof}
Fix $r \geq 4$ and let $S$ be a subset of $[r]$ with $1 < \abs{S} < r - 1$.  To begin, we claim that the image of $\calM^1_S$ is a single point in $\ccalX^1(\FF_r)$.  To this end, we set $n_1 = \abs{S}$ and $n_2 = r - \abs{S}$.  Notice that $n_1,n_2 \geq 2$.  Let $G_{n_1,n_2}$ be the graph that consists of two vertices $v_1$ and $v_2$, and edges $e_0$, $e^1_1,\ldots,e^1_{n_1}$ and $e^2_1,\ldots,e^2_{n_2}$.  The edges are attached via the following table.
\begin{center}
\renewcommand{\arraystretch}{1.3}
\begin{tabular}{c|c|c}
 & $o$ & $\tau$ \\
 \hline
$e_0$ & $v_1$ & $v_2$ \\
\hline 
$e^1_i$ & $v_1$ & $v_1$ \\
\hline
$e^2_i$ & $v_2$ & $v_2$ \\
\hline
\end{tabular}
\end{center} 
See Figure~\ref{fig:double barbell}.  We adopt the notation introduced in Section~\ref{sec:separating} for $G_{n_1,n_2}$.  

\begin{figure}[ht]
\centering
\begin{tikzpicture}[scale=1.3]
\draw[very thick,black] (-1.5,0) -- (1.5,0);
\filldraw (-1.5,0) circle [radius=1.75pt];
\filldraw (1.5,0) circle [radius=1.75pt];
\begin{scope}[xshift=-1.5cm]
\draw[very thick] (0,0) to[out=80,in=0] (0,1);
\draw[very thick] (0,1) to[out=180,in=100] (0,0);
\node at (0,1.3) {$e^1_1$};
\filldraw[rotate=95] (0,0.6) circle [radius=1.25pt];
\filldraw[rotate=115] (0,0.6) circle [radius=1.25pt];
\filldraw[rotate=135] (0,0.6) circle [radius=1.25pt];
\end{scope}
\begin{scope}[xshift=-1.5cm,rotate=45]
\draw[very thick] (0,0) to[out=80,in=0] (0,1);
\draw[very thick] (0,1) to[out=180,in=100] (0,0);
\node at (0,1.3) {$e^1_2$};
\end{scope}
\begin{scope}[xshift=-1.5cm,rotate=180]
\draw[very thick] (0,0) to[out=80,in=0] (0,1);
\draw[very thick] (0,1) to[out=180,in=100] (0,0);
\node at (0,1.3) {$e^1_{n_1}$};
\end{scope}
\begin{scope}[xshift=1.5cm]
\draw[very thick] (0,0) to[out=80,in=0] (0,1);
\draw[very thick] (0,1) to[out=180,in=100] (0,0);
\node at (0,1.3) {$e^2_1$};
\filldraw[rotate=-95] (0,0.6) circle [radius=1.25pt];
\filldraw[rotate=-115] (0,0.6) circle [radius=1.25pt];
\filldraw[rotate=-135] (0,0.6) circle [radius=1.25pt];
\end{scope}
\begin{scope}[xshift=1.5cm,rotate=-45]
\draw[very thick] (0,0) to[out=80,in=0] (0,1);
\draw[very thick] (0,1) to[out=180,in=100] (0,0);
\node at (0,1.3) {$e^2_2$};
\end{scope}
\begin{scope}[xshift=1.5cm,rotate=180]
\draw[very thick] (0,0) to[out=80,in=0] (0,1);
\draw[very thick] (0,1) to[out=180,in=100] (0,0);
\node at (0,1.3) {$e^2_{n_2}$};
\end{scope}
\node at (0,0.2) {$e_0$};
\node at (-1.2,0.2) {$v_1$};
\node at (1.2,0.2) {$v_2$};
\end{tikzpicture}
\caption{The graph $G_{n_1,n_2}$: there are $n_1$ loop edges attached to $v_1$ and $n_2$ loop edges attached to $v_2$.}\label{fig:double barbell}
\end{figure}
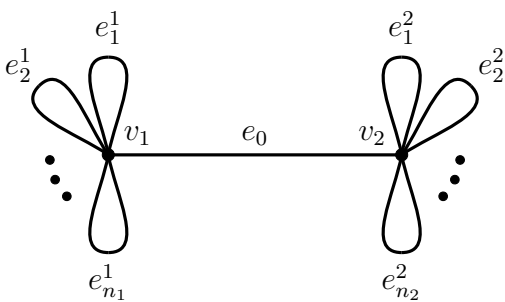 

Let $c \from G_{n_1,n_2} \to \calR_r$ be the map induced by collapsing the edge $e_0$ and let $\rho \from \calR_r \to G_{n_1,n_2}$ be a map so that $c \circ \rho$ is homotopic to ${\rm id} \from \calR_r \to \calR_r$.  Thus \[\calX^1(\calR_r,{\rm id}) \subset \ocalX^1(G_{n_1,n_2},\rho) = \{ [(G,\rho',\ell)] \in \ocalX(G_{n_1,n_2},\rho) \mid \entropy_G(\ell) = 1 \}.\]  Specifically, viewing $\ocalX^1(G_{n_1,n_2},\rho)$ as a subset of $[0,\infty)^{1 + n_1 + n_2}$, we see that $\calX^1(\calR_r,{\rm id})$ is the image of the embedding $\varepsilon \from \calM^1(\calR_r) \to [0,\infty)^{1 + n_1 + n_2}$ where $\varepsilon(\ell)^0 = 0$. 

Moreover, $\varepsilon$ extends to an embedding $\hcalM^1(\calR_r) \to [0,\infty]^{1 + n_1 + n_2}$ in the same way.  Under this embedding, $\varepsilon(\calM^1_S)$ is the face of $\calM^1_1 \subset \hcalM^1(G_{n_1,n_2}) \subset [0,\infty]^{1 + n_1 + n_2}$.  Indeed, the set $\calM^1_1$ is homeomorphic to $(0,\infty] \times \calM^1_S$.  By Proposition~\ref{prop:sep-metric extends}\eqref{item:sep-short cut}, the set $\calM^1_1$ maps to a single point in $\ccalX^1(\FF_r)$.  Hence so does $\calM^1_S$, completing the proof of our claim.

As the diameter of the thin part goes to zero as the short edge goes to zero (Proposition~\ref{prop:thin part}), the point that $\calM^1_S$ is sent to is independent of $S$.
\end{proof}

The main result of this section now follows easily.

\begin{proposition}\label{prop:simplex bounded diameter}
For $r \geq 4$, the subspace $(\calX^1(\calR_r,{\rm id}),d_{\entropy}) \subset (\calX^1(\FF_r),d_\entropy)$ has bounded diameter.
\end{proposition}

\begin{proof}
As explained above in the introduction to this section, by Theorem~\ref{th:completion rose}, there is a map $\Phi \from \Delta^{r-1} - V \to \ccalX^1(\FF_r)$.  By Lemma~\ref{lem:single point}, the map $\Phi$ extends to $V$ and the entire $(r-3)$--skeleton of $\Delta^{r-1}$ is mapped to a single point.  As $\Delta^{r-1}$ is compact, $\Phi(\Delta^{r-1})$ is compact and hence has bounded diameter.  Thus $\calX^1(\calR_r,{\rm id}) \subset \Phi(\Delta^{r-1}) \subset \calX^1(\FF)$ has bounded diameter too.
\end{proof}


\section{Proof of Theorem~\ref{th:entropy bounded}}\label{sec:entropy bounded}

The goal of this final section is the proof of the last of the main results.  Theorem~\ref{th:entropy bounded} states that the $\Out(\FF_r)$--invariant subcomplex $(\calX^1(\calR_r,{\rm id}) \cdot \Out(\FF_r),d_\entropy)$ has bounded diameter and that moreover, the action of $\Out(\FF_r)$ on the completion of $(\calX^1(\FF_r),d_\entropy)$ has a global fixed point.  This point is the image of $\cup \{\calM_S \mid 1 < \abs{S} < r-1 \}$ for any marking of the rose.  We show that the image of this point in the completion is independent of the marking.  This is done by showing that it is independent for markings that differ by a single simple move---we call such markings \emph{Nielsen adjacent}.  This is accomplished again by making use of a graph with a separating edge and the analysis in Section~\ref{sec:separating}.  This simple move suffices to connect any two markings and the theorem easily follows. 

\begin{proof}[Proof of Theorem~\ref{th:entropy bounded}]
Given a marked rose $(\calR_r,\rho)$,  there is an embedding $\Phi_\rho \from \calM^1(\calR_r) \to \calX^1(\FF_r)$ whose image is $\calX^1(\calR_r,\rho)$.  As in Section~\ref{sec:rose bounded}, this map extends to $\Phi_\rho \from \hcalM^1(\calR_r) \to \ccalX^1(\FF_r)$ where $\ccalX^1(\FF_r)$ is the completion of $\calX^1(\FF_r)$ with the entropy metric.  By Lemma~\ref{lem:single point}, $\Phi_\rho$ maps $\cup \{ \calM^1_S \mid 1 < \abs{S} < r - 1 \}$ to a single point in $\ccalX^1(\FF_r)$.  Let $x_{\rho}$ denote this point in $\ocalX(\FF)$.

Given an integer $r > 2$, we define a graph $\Gamma_{r}$ that has two vertices $v$ and $w$, and edges $e^1_1,\ldots,e^1_{r-2}$ and $e^2_0,e^2_1,e^2_{3}$.  The edges are attached via the following table.
\begin{center}
\renewcommand{\arraystretch}{1.3}
\begin{tabular}{c|c|c}
 & $o$ & $\tau$ \\
 \hline
$e^1_i$ & $v$ & $v$ \\
\hline
$e^2_i$ & $v$ & $w$ \\
\hline
\end{tabular}
\end{center} 
See Figure~\ref{fig:rose-theta}.  We call such a graph a \emph{rose-theta graph}.  

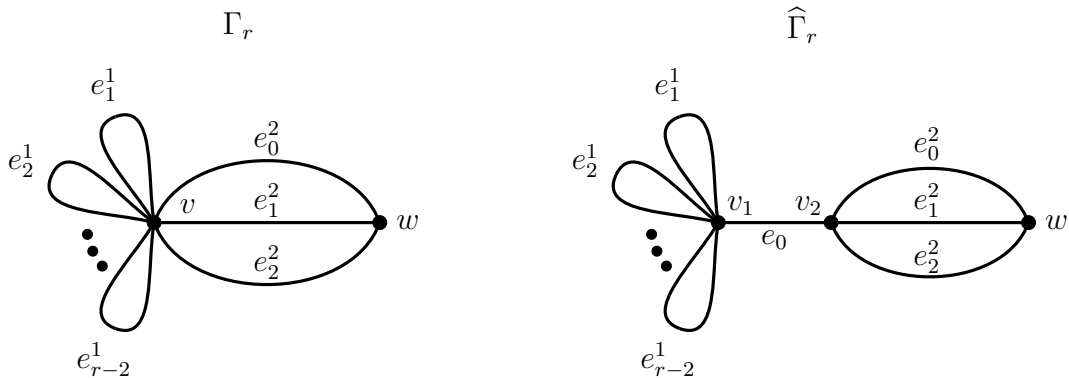
\begin{figure}[ht]
\centering
\begin{tikzpicture}[scale=1.5]
\node at (-0.75,1.75) {$\Gamma_r$};
\filldraw (-1.5,0) circle [radius=1.75pt];
\filldraw (0.5,0) circle [radius=1.75pt];
\begin{scope}[xshift=-1.5cm,rotate=20]
\draw[very thick] (0,0) to[out=80,in=0] (0,1);
\draw[very thick] (0,1) to[out=180,in=100] (0,0);
\node at (0,1.3) {$e^1_1$};
\filldraw[rotate=80] (0,0.6) circle [radius=1.25pt];
\filldraw[rotate=95] (0,0.6) circle [radius=1.25pt];
\filldraw[rotate=110] (0,0.6) circle [radius=1.25pt];
\end{scope}
\begin{scope}[xshift=-1.5cm,rotate=65]
\draw[very thick] (0,0) to[out=80,in=0] (0,1);
\draw[very thick] (0,1) to[out=180,in=100] (0,0);
\node at (0,1.3) {$e^1_2$};
\end{scope}
\begin{scope}[xshift=-1.5cm,rotate=160]
\draw[very thick] (0,0) to[out=80,in=0] (0,1);
\draw[very thick] (0,1) to[out=180,in=100] (0,0);
\node at (0,1.3) {$e^1_{r-2}$};
\end{scope}
\draw[very thick] (-1.5,0) to [out=70,in=110] (0.5,0);
\draw[very thick] (-1.5,0) to [out=0,in=180] (0.5,0);
\draw[very thick] (-1.5,0) to [out=-70,in=-110] (0.5,0);
\node at (-0.5,0.75) {$e^2_0$};
\node at (-0.5,0.2) {$e^2_1$};
\node at (-0.5,-0.35) {$e^2_2$};
\node at (-1.2,0.15) {$v$};
\node at (0.75,0) {$w$};
\begin{scope}[xshift=5cm]
\node at (-0.75,1.75) {$\widehat \Gamma_r$};
\draw[very thick] (-1.5,0) -- (-0.5,0);
\filldraw (-1.5,0) circle [radius=1.75pt];
\filldraw (-0.5,0) circle [radius=1.75pt];
\filldraw (1.25,0) circle [radius=1.75pt];
\begin{scope}[xshift=-1.5cm,rotate=20]
\draw[very thick] (0,0) to[out=80,in=0] (0,1);
\draw[very thick] (0,1) to[out=180,in=100] (0,0);
\node at (0,1.3) {$e^1_1$};
\filldraw[rotate=80] (0,0.6) circle [radius=1.25pt];
\filldraw[rotate=95] (0,0.6) circle [radius=1.25pt];
\filldraw[rotate=110] (0,0.6) circle [radius=1.25pt];
\end{scope}
\begin{scope}[xshift=-1.5cm,rotate=65]
\draw[very thick] (0,0) to[out=80,in=0] (0,1);
\draw[very thick] (0,1) to[out=180,in=100] (0,0);
\node at (0,1.3) {$e^1_2$};
\end{scope}
\begin{scope}[xshift=-1.5cm,rotate=160]
\draw[very thick] (0,0) to[out=80,in=0] (0,1);
\draw[very thick] (0,1) to[out=180,in=100] (0,0);
\node at (0,1.3) {$e^1_{r-2}$};
\end{scope}
\draw[very thick] (-0.5,0) to [out=70,in=110] (1.25,0);
\draw[very thick] (-0.5,0) to [out=0,in=180] (1.25,0);
\draw[very thick] (-0.5,0) to [out=-70,in=-110] (1.25,0);
\node at (0.35,0.7) {$e^2_0$};
\node at (0.35,0.2) {$e^2_1$};
\node at (0.35,-0.3) {$e^2_2$};
\node at (-1,-0.15) {$e_0$};
\node at (-1.3,0.15) {$v_1$};
\node at (-0.7,0.15) {$v_2$};
\node at (1.5,0) {$w$};
\end{scope}
\end{tikzpicture}
\caption{The graphs $\Gamma_r$ and $\widehat \Gamma_r$.  In $\Gamma_{r}$ there are $r-2$ loop edges attached to $v$ and $3$ edges connecting $v$ to $w$.  In $\widehat \Gamma_{r}$, there are $r-2$ loop edges attached to $v_1$, $3$ edges connecting $v_2$ to $w$ and a separating edge connecting $v_1$ to $v_2$.}\label{fig:rose-theta}
\end{figure} 

Given $i \in \{0,1,2\}$, collapsing the edge $e^2_i$ induces a map $c_i \from \Gamma_r \to \calR_{r}$.  We say two marked roses $(\calR_r,\rho_1)$ and $(\calR_r,\rho_2)$ are \emph{Nielsen adjacent} if there is a marked rose-theta graph $(\Gamma_r,\rho)$ such that $\rho_i \simeq c_i \circ \rho$ for $i = 1$, $2$.  Given any two marked roses, $(\calR_r,\rho)$ and $(\calR_r,\rho')$, it is well known that there is a sequence of markings $\rho = \rho_1,\ldots,\rho_n = \rho'$ such that $(\calR_r, \rho_{i-1})$ and $(\calR_r,\rho_i)$ are Nielsen adjacent for $i = 2,\ldots,n$.  For instance, see~\cite{ar:CV86}.

We will prove the theorem by showing that if $(\calR_r,\rho_1)$ and $(\calR_r,\rho_2)$ are Nielsen adjacent, then $x_{\rho_1} = x_{\rho_2}$.  Notice the collection $\{x_{{\rm id} \cdot \phi}\}$ is invariant under the action by $\Out(\FF_r)$.  Hence this also shows that the action of $\Out(\FF_r)$ on $\ccalX^1(\FF_r)$ has a global fixed point.  

To this end, let $(\Gamma_r,\rho)$ be the marked rose-theta graph such that $\rho_i \simeq c_i \circ \rho$.  We need to introduce a separating edge to take advantage of the shortcuts utilized in Section~\ref{sec:separating}.  Let $\widehat \Gamma_{r}$ be the graph obtained from blowing up the vertex $v$ in $\Gamma_r$.  Specifically, in $\widehat{\Gamma}_r$ there are three vertices $v_1$, $v_2$ and $w$, and edges $e_0$, $e^1_1,\ldots,e^1_{r-2}$ and $e^2_0,e^2_1,e^2_{3}$.  The edges are attached via the following table.
\begin{center}
\renewcommand{\arraystretch}{1.3}
\begin{tabular}{c|c|c}
& $o$ & $\tau$ \\
\hline
$e_0$ & $v_1$ & $v_2$ \\
\hline
$e^1_i$ & $v_1$ & $v_1$ \\
\hline
$e^2_i$ & $v_2$ & $w$ \\
\hline
\end{tabular}
\end{center} 
See Figugre~\ref{fig:rose-theta}.  We adopt the notation from Section~\ref{sec:separating} for $\widehat \Gamma_r$.

Let $c \from \widehat \Gamma_r \to \Gamma_r$ be the map that collapses the edge $e_0$.  There is a marking $\hat \rho \from \calR_r \to \widehat \Gamma_r$ such that $c \circ \hat \rho \simeq \rho$.  Viewing $\ocalX^1(\widehat{\Gamma}_r,\hat \rho)$ as a subset of $[0,\infty]^{r+2}$, there are two embeddings corresponding to $\rho_1$ and $\rho_2$ denoted $\varepsilon_1, \varepsilon_2 \from \hcalM^1(\calR_r) \to [0,\infty]^{r+2}$ where $\varepsilon_i(\ell)^0 = \varepsilon_i(\ell)(e^2_i) = 0$ for $i = 1$, $2$.  

Let $S \subset [r]$ denote the set of edges $\{c_i(e^1_1),\ldots, c_i(e^1_{r-2})\}$ in $\calR_r$.  Notice that this set is independent of $i$.  Both $\varepsilon_1(\calM^1_S)$ and $\varepsilon_2(\calM^1_S)$ are faces of $\calM^1_{1,2} \subset \hcalM^1(\widehat{\Gamma}_r)$.  Hence by Proposition~\ref{prop:sep-metric extends}\eqref{item:sep-short cut} we have that $x_{\rho_1} = x_{\rho_2}$.        
\end{proof}


\bibliography{thermo}
\bibliographystyle{acm}

\end{document}